\documentclass[11pt]{amsart}
\usepackage[english]{babel}

\usepackage[textwidth=17.5cm,textheight=22cm,
]{geometry}

\newcommand{\lvt}{\left|\kern-1.35pt\left|\kern-1.3pt\left|}
\newcommand{\rvt}{\right|\kern-1.3pt\right|\kern-1.35pt\right|}

\usepackage{tensor}
\usepackage[foot]{amsaddr}
\usepackage[english]{babel}
\usepackage[
pdftex,hyperfootnotes]{hyperref}
\usepackage[dvipsnames,svgnames,table,x11names,table]{xcolor}
\usepackage{pgfplots}
\usepackage{tikz}
\usepackage{tikz-3dplot}
\usetikzlibrary{calc,shadows,shapes.callouts,shapes.geometric,shapes.misc,positioning,patterns,decorations.pathmorphing,decorations.markings,decorations.fractals,decorations.pathreplacing,shadings,fadings,arrows.meta}

\usepackage{longtable}
\usepackage{enumerate}
\allowdisplaybreaks

\usepackage{fancybox}
\usepackage[utf8]{inputenc}

\usepackage{tcolorbox}

\usepackage{rotating,multirow,pdflscape}
\usepackage{amssymb,latexsym,amsmath,amsthm}
\usepackage{mathrsfs}
\usepackage{mathtools,arydshln,mathdots}
\mathtoolsset{showonlyrefs}
\usepackage{bigints}
\usepackage{comment}

\usepackage[renew-dots,renew-matrix]{nicematrix}

\makeatletter
\renewcommand*\env@matrix[1][*\c@MaxMatrixCols c]{%
	\hskip -\arraycolsep
	\let\@ifnextchar\new@ifnextchar
	\array{#1}}
\makeatother

\mathtoolsset{centercolon}
\usepackage[toc,page]{appendix}

\usepackage{tocvsec2}

\usepackage{Baskervaldx}
\usepackage{newtxmath}

\usepackage{framed}
\theoremstyle{plain}

\newtheorem{Theorem}{Theorem}[section]
\newtheorem{Corollary}[Theorem]{Corollary}
\newtheorem{lemma}[Theorem]{Lemma}
\newtheorem{pro}[Theorem]{Proposition}

\newtheorem{Definition}[Theorem]{Definition}

\newtheorem{Remark}[Theorem]{Remark}

\newcommand{\ii}{\operatorname{i}}

\renewcommand{\d}{\mathrm{d}}
\newcommand{\Exp}[1]{\operatorname{e}^{#1}}

\newcommand{\diag}{\operatorname{diag}}

\newcommand{\C}{\mathbb{C}}
\newcommand{\T}{\mathbb{T}}
\newcommand{\N}{\mathbb{N}}

\newcommand{\Z}{\mathbb{Z}}

\renewcommand{\L}{\mathscr{L}}
\newcommand{\U}{\mathscr{U}}
\renewcommand{\l}{\mathscr{l}}
\renewcommand{\u}{\mathscr{u}}
\newcommand{\A}{\mathscr{A}}
\newcommand{\B}{\mathscr{B}}
\newcommand{\M}{\mathscr{M}}
\newcommand{\Tscr}{\mathscr{T}}
\newcommand{\Rscr}{\mathscr{R}}
\newcommand{\Sscr}{\mathscr{S}}

\allowdisplaybreaks[1]

\usepackage{orcidlink}

\usepackage{drawmatrix}

\begin{document}

\title[Mixed Multiple Orthogonal Laurent Polynomials on the Unit Circle]{Mixed Multiple Orthogonal Laurent Polynomials\\ on the Unit Circle}

\author[EJ Huertas]{Edmundo J. Huertas$^{1}$\orcidlink{0000-0001-6802-3303}}

\address{$^{1}$Departamento de Física y Matemáticas, Universidad de Alcalá,
	Ctra. Madrid-Barcelona, Km. 33,600.
	28805 - Alcalá de Henares, Madrid, Spain}

\email{$^{1}$edmundo.huertas@uah.es}

\author[M Mañas]{Manuel Mañas$^{2}$\orcidlink{0000-0003-3764-5737}}
\address{$^2$Departamento de Física Teórica, Universidad Complutense de Madrid, Plaza Ciencias 1, 28040-Madrid, Spain 
}
\email{$^2$manuel.manas@ucm.es}

\keywords{Mixed multiple orthogonal Laurent polynomials, unit circle, Christoffel--Darboux formulas, ABC theorem, recurrence relations, Christoffel perturbations, Geronimus perturbations}

\subjclass{33C45, 33C47,42C05, 15A23}

\date{\emph{\today}}
\maketitle

\begin{abstract}
Mixed orthogonal Laurent polynomials on the unit circle of CMV type are constructed  utilizing a matrix of moments and its Gauss--Borel factorization and employing a multiple extension of the CMV ordering. A systematic analysis of the associated multiple orthogonality and biorthogonality relations, and an examination of the degrees of the Laurent polynomials is given. Recurrence relations, expressed in terms of banded matrices, are found. These recurrence relations lay the groundwork for corresponding Christoffel--Darboux kernels and relations, as well as for elucidating the ABC theorem. The paper also develops the theory of diagonal Christoffel and Geronimus perturbations of the matrix of measures.  Christoffel formulas are found for both perturbations.
\end{abstract}
\allowdisplaybreaks 
\tableofcontents


\section{Introduction}

Defined with respect to multiple weight functions or measures, multiple orthogonal polynomials form a distinctive class of polynomials that play an important role in fields such as numerical analysis, approximation theory, and mathematical physics. Unlike traditional orthogonal polynomials associated with a single weight function, these polynomials satisfy simultaneous orthogonality conditions, making them valuable for tackling complex analytical problems. For a comprehensive introduction to multiple orthogonal polynomials, see \cite{ismail} and \cite{Nikishin-Sorokin}, while \cite{afm} explores their connections with integrable systems. Their significance also extends to the study of Markov processes and generalized random walks beyond the standard birth-and-death framework, as shown in \cite{bidiagonal2, Reference 10, Reference 11, Reference 12}.

The concept of mixed multiple orthogonal polynomials was first introduced by Sorokin in 1994 \cite{sorokin1} and later expanded in 1997 in collaboration with van Iseghem \cite{Sorokin_Van_Iseghem_1}, as they explored matrix orthogonality of vector polynomials, see also \cite{Sorokin_Van_Iseghem_2,Sorokin_Van_Iseghem_3}. This framework re-emerged in 2004 \cite{Evi_Arno} within the study of multiple non-intersecting Brownian motions, where the term “mixed multiple orthogonal” was formally coined. Further discussions can be found in \cite{mixto_FUG}, \cite{afm}, and \cite{Ulises-Sergio-Judith}. In these studies, a rectangular \( q \times p \) matrix of weights, each of rank 1 at every point in the support, was constructed. However, the most relevant frameworks for the present discussion are found in \cite{Sorokin_Van_Iseghem_1} and \cite{Ulises-Sergio-Judith}, where the \( q \times p \) rectangular matrix consists of functionals or measures.

Mixed multiple orthogonal polynomials play a significant role in the spectral theory of banded operators. For certain banded matrices that admit a positive bidiagonal factorization, it has been demonstrated that a Favard-type spectral theorem holds in this extended framework, generalizing the classical tridiagonal case of Jacobi matrices; see \cite{bidiagonal1, bidiagonal2, bidiagonal3}.

Orthogonal polynomials on the unit circle have been the subject of extensive study; see, for example, \cite{Simon}. A major development emerged with the introduction of the Cantero–Moral–Velázquez (CMV) ordering and the discussion of CMV orthogonal Laurent polynomials \cite{CMV,Simon,am}. These ideas have since been extended to various settings, such as matrix orthogonality \cite{am2} and multivariate Laurent orthogonal polynomials on tori \cite{am3}. Additionally, the Christoffel and Geronimus perturbations have been examined in these extended cases \cite{amt,am4,GMM}, see also \cite{am3}. In the matrix case, the Riemann–Hilbert problem was discussed in connection with the matrix discrete Painlevé II equation \cite{CM}.

However, to date, almost only two papers have explored possible constructions of multiple orthogonal polynomials on the unit circle: the seminal paper \cite{MV-CA08} and later \cite{CDR-JCAM15}. Since then, there has been little further development. In January 2024, the first author announced at the Biennial Meeting of The Spanish Royal Mathematical Society the results presented in this paper; see this \href{https://2024.bienalrsme.com/sites/default/files/S2%20PolnoOrtog_FE_app.pdf}{link}. As we were completing the final revisions of this paper, two new papers on the subject appeared \cite{KM, k}.

This paper develops the framework of mixed multiple orthogonal Laurent polynomials, beginning with an introduction to CMV mixed multiple Laurent orthogonality in Section 2. This section establishes foundational concepts, including the construction of the moment matrix and its Gauss–Borel factorization, which are essential to understanding the properties and structure of these polynomials.

Section 3 explores mixed multiple orthogonal Laurent polynomials in depth, covering their orthogonality properties, recurrence relations, Szegő recurrence matrices, and various recurrence structures specific to both matrix and scalar cases. Additionally, this section introduces the Christoffel–Darboux kernels, which are central to the analysis within this framework, as well as second kind functions.

The study of perturbations begins in Section 4, which addresses diagonal Christoffel perturbations. This section examines the connector matrix, connection formulas, and the Christoffel formulas for key polynomial functions, analyzing how these perturbations modify and extend the classical structure.

Section 5 then explores Geronimus perturbations, presenting foundational tools such as the connector matrix and connection formulas in this context. This section also includes the study of diagonal Geronimus perturbations with singular parts, as well as Christoffel–Geronimus formulas for diagonal Geronimus  transformations.


\label{S01-Intro}

\section{CMV mixed multiple Laurent orthogonality in the unit circle}

\subsection{The  moment matrix and its Gauss--Borel factorization}

\label{S02-Defs}

A measure $\mu$ in unit circle 
\begin{align*}
	\begin{aligned}
		\mathbb{T}&\coloneq \{z\in\C:|z|=1\}, & z&=\Exp{\ii \theta}, &\theta&\in[-\pi,\pi],
	\end{aligned}
\end{align*}
can be always split into its absolutely continuous part with respect to the normalized Lebesgue measure $\frac{\d\theta}{2\pi}$ and its singular part $\mu_s$,
\begin{align*}
	\mu(\theta)=w(\theta)\frac{\d\theta}{2\pi}+\mu_s(\theta).
\end{align*}
Notice that on $\mathbb T$ we have $\d z=\d \Exp{\ii \theta}=\ii \Exp{\ii \theta}\d\theta$, so that
\begin{align*}
\frac{1}{2\pi}	\d\theta=\frac {\d z}{2\pi \ii z}.
\end{align*}
Hence, we can write for $z\in\T$
\begin{align*}
\mu(z)=w(z)\frac{\d z}{2\pi\ii z}+\d\mu_s.
\end{align*}
We assume that the support of the measure $\mu$ is a infinite set in $\T$.
Given two function $f,g$ over $\T$ two possible inner products can be considered
\begin{align*}
\begin{aligned}
	&	\int_{-\pi}^\pi \overline{f(\theta)} g(\theta) \d\mu(\theta), &
	&\int_{-\pi}^\pi {f(\theta)} \overline{g(\theta) }\d\mu(\theta),
\end{aligned}
\end{align*} 
where the overline denotes the complex conjugate.
In terms of integrals in the unit circle these integrals are written
\begin{align*}
	\begin{aligned}
		&	\int_\T \bar f(z^{-1}) g(z) \d\mu(z), &
		&\int_\T {f(z)} \bar g(z^{-1}) \d\mu(\theta),
	\end{aligned}
\end{align*} 
respectively.

Further, as we are going to discuss mixed multiple orthogonal systems in the unit circle we need to 
 consider a  $q\times p$ rectangular matrix, $q,p\in \mathbb{N}$, with $q$ rows and $p$ columns,
of Borel complex measures, with support having infinite points,  on the unit circle $\mathbb{T}$%
\begin{align}
	\mu (z)&\coloneq%
	\begin{bmatrix}
		\mu _{1,1} (z)& \cdots &\mu _{1,p}(z) \\ 
		\vdots &  & \vdots \\ 
	\mu _{q,1} (z)& \cdots & \mu _{q,p}(z)%
	\end{bmatrix}  \label{MatrixMeasures}
\end{align}%
with its splitting in terms of absolutely continuous and singular parts is
\begin{gather*}
	\d\mu (z)=w(z)\frac{\d z}{2\pi \ii z}+\d\mu_s,\\
\begin{aligned}
	w&=	\begin{bmatrix}
		w _{1,1} (z)& \cdots &  w _{1,p}(z) \\ 
		\vdots &  & \vdots \\ 
		w _{q,1} (z)& \cdots & w _{q,p}(z)%
	\end{bmatrix}, &	\mu_s&= \begin{bmatrix}\mu _{s,1,1} (z)& \cdots & \mu _{s,1,p}(z) \\ 
		\vdots &  & \vdots \\ 
		\mu _{s,q,1} (z)& \cdots & \mu _{s,q,p}(z).%
	\end{bmatrix}.
\end{aligned}
\end{gather*}

The corresponding moments are
 \begin{align*}
\begin{aligned}
		c_{n}&=\oint_{\mathbb{T}}z^{n}\d\mu (z)=\int_{-\pi}^\pi \Exp{\ii n\theta}\d \mu(\theta)\in \mathbb{C}^{q\times p}, & n&\in\Z,
\end{aligned}
\end{align*}
are, in turn, $q\times p$ complex matrices with entries $c_{b,a,n}$ $a\in\{1,\dots,p\}$ and $q\in\{1,\dots,q\}$.
The Fourier series of this rectangular matrix of measures is
\begin{align*}
	F_\mu(u)\coloneq\sum_{n=-\infty}^{\infty} c_n u^n,
\end{align*}
 for absolutely continuous measures, $\d \mu(\theta)=w(\theta) \frac{\d \theta}{2\pi}$
  satisfies 
 \begin{align*}
 	F_\mu(\theta)=w(\theta).
 \end{align*}
 Let $A(0 ; r, R)=\{z \in \mathbb{C}: r<|z|<R\}$ denote the annulus around $z=0$ with interior and exterior radii $r$ and $R$. For $a\in\{1,\dots,p\}$ and $b\in\{1,\dots,q\}$, consider  
\begin{align*}
\begin{aligned}
		 R_{b,a, \pm}:=\left(\limsup _{n \rightarrow \infty} \sqrt[n]{\left|c_{b,a, \pm n}\right|}\right)^{\mp 1}
\end{aligned}
\end{align*}
 and 
\[ \begin{aligned}
 	R_{+}&=\min _{\substack{a\in\{1,\dots,p\},\\b\in\{1,\dots,q\}}} R_{b,a,+}, & R_{-}&=\max  _{\substack{a\in\{1,\dots,p\},\\b\in\{1,\dots,q\}}} R_{b,a,-}. 
 \end{aligned}\]
 According to the Cauchy--Hadamard theorem, the series $F_\mu(z)$ converges uniformly in any compact set $K$ that belongs to  the annulus $A\left(0 ; R_{-}, R_{+}\right)$.

Let us introduce  basic objects in our discussion. 
\begin{Definition}[CMV monomial matrices]
 Given $r\in \mathbb{N}$, and being $I_{r}$ the 
identity matrix of size $r\times r$, we define the following semi-infinite
monomial matrix%
\begin{align}
Z_{[r]}(z)\coloneq%
\begin{bmatrix}
I_{r} \\ 
z^{-1}I_{r} \\ 
zI_{r} \\ 
z^{-2}I_{r} \\ 
z^{2}I_{r} \\ 
\vdots%
\end{bmatrix},
\label{ZnmonCMV}
\end{align}%
and denote by $Z_{[r]}^{(j)}$, $j=1,\ldots ,r$ the $j$-th semi-infinite
column of $Z_{[r]}$, such that%
\begin{align}
Z_{[r]}(z)=%
\begin{bmatrix}
Z_{[r]}^{(1)} & Z_{[r]}^{(2)} & \cdots & Z_{[r]}^{(r)}%
\end{bmatrix}%
.  \label{ZnmonCMVj}
\end{align}%
Provided that on the unit circle it holds $\bar{z}=z^{-1}$, one also has the
following monomial matrix that will be useful in the sequel%
\begin{align}
\left( Z_{[r]}(z)\right) ^{\dag }=\overline{Z_{[r]}^{\top }(z)}=\bar{Z}%
_{[r]}^{\top }(\bar{z})=Z_{[r]}^{\top }(z^{-1})=%
\begin{bmatrix}
I_{r} & zI_{r} & z^{-1}I_{r} & z^{2}I_{r} & \Cdots[shorten-end=6pt]%
\end{bmatrix}%
.  \label{cojugaZ}
\end{align}
\end{Definition}

We also denote by $\left( Z_{[r]}^{(i)}\right) ^{\top }$, $i=1,\ldots ,r$ the $i$-th semi-infinite \textit{row} of $Z_{[r]}^{\top }$,
such that
\begin{align*}
Z_{[r]}^{\top }=
\begin{bmatrix}
\left( Z_{[r]}^{(1)}\right) ^{\top } \\ 
\left( Z_{[r]}^{(2)}\right) ^{\top } \\ 
\vdots \\ 
\left( Z_{[r]}^{(r)}\right) ^{\top }
\end{bmatrix}%
.
\end{align*}
Then, the corresponding matrix of moments
are given by:

\begin{Definition}[CMV moment matrices]
\label{DefBlockM} 
Let us consider the following semi-infinite left CMV moment matrix:
\begin{align}
\begin{aligned}
	M_\mu&\coloneq\oint_{\mathbb{T}}Z_{[q]}(z)\d\mu
(z)\,Z_{[p]}^{\top }(z^{-1})=\int_{-\pi}^\pi Z_{[q]}(\Exp{\ii\theta})\d\mu
(\theta)\, Z_{[p]}(\Exp{-\ii\theta})
\\&=%
\left[\begin{NiceMatrix}
c_{0} & c_{1} & c_{-1} & c_{2} & \cdots \\ 
c_{-1} & c_{0} & c_{-2} & c_{1} & \cdots \\ 
c_{1} & c_{2} & c_{0} & c_{3} & \cdots \\ 
c_{-2} & c_{-1} & c_{-3} & c_{0} & \cdots \\ 
\vdots & \vdots & \vdots & \vdots & %
\end{NiceMatrix}\right]%
, 
\end{aligned} \label{MMleft}
\end{align}%
and semi-infinite right CMV moment matrix:%

\begin{equation}
\begin{aligned}
	\M_\mu&\coloneq\oint_{\mathbb{T}}Z_{[q]}(z^{-1})\d\mu (z)\,Z_{[p]}^{\top }(z)=\int_{-\pi}^\pi Z_{[q]}(\Exp{-\ii\theta})\d\mu (\theta)\,Z_{[p]}^{\top }(\Exp{\ii\theta})\
\\&=%
\left[\begin{NiceMatrix}
c_{0} & c_{-1} & c_{1} & c_{-2} & \Cdots \\ 
c_{1} & c_{0} & c_{2} & c_{-1} & \Cdots \\ 
c_{-1} & c_{-2} & c_{0} & c_{-3} & \Cdots \\ 
c_{2} & c_{1} & c_{3} & c_{0} & \Cdots \\ 
\Vdots & \Vdots & \Vdots & \Vdots & %
\end{NiceMatrix}\right]. 
\end{aligned} \label{MMright}
\end{equation}
\end{Definition}

Some basic properties of these moments matrices are:
\begin{pro}\label{pro:sym}
	The moment matrices fulfill
\[	\begin{aligned}
		\overline{M_\mu}&=\M_{\bar \mu}, & (M_\mu)^\top&=\M_{\mu^\top}, &
		(M_\mu)^\dagger&= M_{\mu^\dagger}, & (\M_\mu)^\dagger&= \M_{\mu^\dagger}.
	\end{aligned}\]
\end{pro}

For the sake of simplicity, if not required we will drop the subindex indicating the measure $\mu$.

The  Gauss factorization of the above
matrices \eqref{MMleft}  and \eqref{MMright} will be instrumental %
\begin{align}
M=&L^{-1}\bar U^{-1},  \label{LUMMleft} \\
\M =&\bar{\mathcal{L}}^{-1}\mathcal{U}^{-1},
\label{LUMMright}
\end{align}%
such that $L$ and $\mathcal{L}$ ($U$ and $\mathcal{U}$), are
respectively semi-infinite invertible lower (upper) triangular matrices.  Notice that all the diagonal entries of $L$ or $U$ are assumed to be different from zero.
The matrices $L$ and $U$ are not uniquely determined by $M$, indeed the transformation $L\to DL$, $U\to UD^{-1}$ will provide new possible factors whenever $D$ is a diagonal matrix with nonzero diagonal entries. Similar comments hold for $\L$ and $\U$. These factorizations exist if and only if the leading principal minors are nonzero.

Taking into account the symmetries described in Proposition \ref{pro:sym} we find
\begin{pro}\label{pro:LU_sym}
The upper an lower triangular matrices in the Gauss--Borel factorization can be chosen so that
	\begin{align*}
		&\begin{aligned}
			{L_\mu}&=\L_{\bar \mu}, & {U_\mu}&=\U_{\bar \mu},
		\end{aligned}& &
		 \begin{aligned}
		 	(U_\mu)^\top&=\L_{\mu^\top}, &(L_\mu)^\top&=\U_{\mu^\top}, 
		 \end{aligned}&&
	\begin{aligned}
			(U_\mu)^\dagger&= L_{\mu^\dagger}, 
			&	(\U_\mu)^\dagger&= \L_{\mu^\dagger}.
	\end{aligned}
	\end{align*}
\end{pro}
Hence, 	we find
\begin{Corollary}\label{cor:LU-sym}
	\begin{enumerate}
		\item If the matrix of measures is real, i.e. $\bar \mu=\mu$, then
		\begin{align*}
		\begin{aligned}
				{L}&=\L, &  {U}&=\U.
		\end{aligned}
		\end{align*}
		\item For $p=q$, if the matrix of measures is symmetric, i.e. $\mu=\mu^\top$, then 
		 		 \begin{align*}
		 \begin{aligned}
		 		U^\top&=\L, &L^\top&=\U.
		 \end{aligned}
		 \end{align*}
				\item For $p=q$, if the matrix of measures is Hermitian, i.e. $\mu=\mu^\dagger$, then  
				\begin{align*}
				\begin{aligned}
					U^\dagger&= L, 	&	\U^\dagger&= \L.
				\end{aligned}
				\end{align*}
				\item  For $p=q$, real symmetric matrix of measures we find
				\begin{align*}
					\begin{aligned}
						U&=L^\top, & \L&=L, & \U&=L^\top. 
					\end{aligned}
				\end{align*}
	\end{enumerate}
\end{Corollary}

\subsection{Mixed multiple orthogonal Laurent polynomials on the unit circle}

Building upon the Gauss factorization \eqref{LUMMleft} of the CMV left moment matrix \( M \), we introduce the following sets of matrices with  Laurent polynomial entries:
\begin{Definition}[CMV left Laurent polynomials]
\label{CMVBlockPolBA} 
Associated with the CMV left
moment matrix ${M}$, let us introduce the following semi-infinite $\infty\times q$  matrix of Laurent polynomials%
\begin{align}
\label{Bpol} \begin{aligned}
	B(z)&\coloneq LZ_{[q]}(z) \\&=\left[\begin{NiceMatrix}
{L}_{0,0} & 0_{q} & 0_{q} & 0_{q} & \Cdots \\ 
{L}_{1,0} & {L}_{1,1} & 0_{q} & 0_{q} & \Cdots \\ 
{L}_{2,0} & {L}_{2,1} & {L}_{2,2} & 0_{q} & \Cdots
\\ 
{L}_{3,0} & {L}_{3,1} & {L}_{3,2} & {L}_{3,3}
& \Cdots \\ 
\Vdots & \Vdots & \Vdots & \Vdots & %
\end{NiceMatrix}\right]%
\begin{bmatrix}
I_{q} \\ 
z^{-1}I_{q} \\ 
zI_{q} \\ 
z^{-2}I_{q} \\ 
z^{2}I_{r} \\ 
\vdots%
\end{bmatrix}
=%
\begin{bmatrix}
{B}_{0}(z) \\ 
{B}_{1}(z) \\ 
{B}_{2}(z) \\ 
{B}_{3}(z) \\ 
\vdots%
\end{bmatrix}%
=%
\begin{bmatrix}
B^{(1)}(z) & B^{(2)}(z) & \cdots & B^{(q)}(z)%
\end{bmatrix}%
\end{aligned}  
\end{align}%
where $I_{q}$ (identity matrix), ${L}_{i,j}$, and ${B}%
_{m}(z) $, are $q\times q$ square complex matrices.  $i=j$, ${L}_{i,i}$\
is a lower triangular complex matrix. When $i<j$, $L_{i,j}=0_{q}$, which
stands for the $q\times q$ zero matrix. When $i>j$, ${L}_{i,j}$\ is,
in general, a full matrix, and 
\begin{align}
\begin{aligned}
	B^{(b)}(z)&\coloneq LZ_{[q]}^{(b)}(z),& b&\in \{1,\ldots ,q\},
\end{aligned}
\label{Bpol(b)}
\end{align}%
is the $b$-th semi-infinite column vector of the above CMV matrix Laurent
polynomial $B(z)$. For $m\in\N_0 $, we also consider the
\textquotedblleft block polynomials,\textquotedblright which are $q\times q$
square matrices with Laurent polynomial entries, as follows%
\begin{align}
{B}^{[q]}_{m}(z)\coloneq%
\begin{bNiceMatrix}
B_{mq}^{(1)}(z) & \Cdots & B_{mq}^{(q)}(z) \\ 
\Vdots[shorten-end=-4pt] &  & \Vdots[shorten-end=-4pt] \\ 
B_{(m+1)q-1}^{(1)}(z) & \Cdots & B_{(m+1)q-1}^{(q)}(z)%
\end{bNiceMatrix}%
.  \label{BlockPolB}
\end{align}%
Analogously, we define%
\begin{align}
A(z)&\coloneq Z_{[p]}^{\top }(z)U \\&=%
\begin{bmatrix}
I_{p} & z^{-1}I_{p} & zI_{p} & z^{-2}I_{p} & z^{2}I_{r} &\cdots[shorten-end=6pt]%
\end{bmatrix}%
\left[\begin{NiceMatrix}
{U}_{0,0} & {U}_{0,1} & {U}_{0,2} & {U}_{0,3}
& \Cdots \\ 
0_{p} & {U}_{1,1} & {U}_{1,2} & {U}_{1,3} & \Cdots
\\ 
0_{p} & 0_{p} & {U}_{2,2} & {U}_{2,3} & \Cdots \\ 
0_{p} & 0_{p} & 0_{p} & {U}_{3,3} & \Cdots \\ 
\Vdots & \Vdots & \Vdots & \Vdots & %
\end{NiceMatrix}\right]
\label{Apol} \\
&=%
\begin{bmatrix}
{A}_{0}(z) & {A}_{1}(z) & {A}_{2}(z)&& \cdots[shorten-end=6pt]%
\end{bmatrix}%
=%
\begin{bmatrix}
A^{(1)}(z) \\ 
A^{(2)}(z) \\ 
\vdots \\ 
A^{(p)}(z)%
\end{bmatrix},
\notag
\end{align}%
where $I_{p}$ (identity matrix), and ${U}_{i,j}$, $A_m(z)$, are $%
p\times p$ square matrices. When $i=j$, ${U}_{i,i}$\ is an upper
triangular matrix. When $i>j$, ${U}_{i,j}=0_{p}$, which stands for
the $p\times p$ zero matrix. When $i<j$, ${U}_{i,j}$\ is, in
general, a full matrix, and 
\begin{align}
\begin{aligned}
	A^{(a)}(z)&\coloneq\left( Z_{[p]}^{(a)}(z)\right) ^{\top }U,&
a&\in \{1,\ldots ,p\} ,
\end{aligned} \label{Apol(a)}
\end{align}%
is the $a$-th row of the above CMV matrix Laurent polynomial $A(z)$.
Thus, for $m\in\N_0$, we also have the corresponding \textquotedblleft
block polynomials\textquotedblright\ ${A}_{m}(z)$, which are
the following $p\times p$ square matrices with polynomial entries%
\begin{align}
{A}^{[p]}_{m}(z)\coloneq%
\begin{bmatrix}
A_{mp}^{(1)}(z) & \cdots & A_{(m+1)p-1}^{(1)}(z) \\ 
\Vdots[shorten-end=2pt] &  & \Vdots[shorten-end=-4pt] \\ 
A_{mp}^{(p)}(z) & \cdots & A_{(m+1)p-1}^{(p)}(z)%
\end{bmatrix}%
.  \label{BlockPolA}
\end{align}
\end{Definition}


Notice that $B^{(b)}(z)$, $b\in \{1,\ldots ,q\}$ and $A^{(a)}(z)$, $%
a\in \{1,\ldots ,p\}$ are semi-infinite column (resp. row) vectors whose
entries are polynomials, namely%
\begin{align}
\begin{aligned}
	B^{(b)}(z)&=%
\begin{bmatrix}
B_{0}^{(b)}(z) \\[3pt]
B_{1}^{(b)}(z) \\[3pt]
B_{2}^{(b)}(z) \\ 
\vdots[shorten-end=4pt,shorten-start=-2pt]%
\end{bmatrix}%
,& A^{(a)}(z)&=%
\begin{bmatrix}
A_{0}^{(a)}(z) & A_{1}^{(a)}(z) & A_{2}^{(a)}(z) & \cdots[shorten-end=6pt]%
\end{bmatrix}.
\end{aligned}  \label{defAB}
\end{align}%
Concerning the degrees of every entry of the above vector polynomials, we state the next result. First, in what follows, given a Laurent polynomial
\begin{align*}
	L(z)=\frac{L_{-d_-}}{z^{d_-}}+\frac{L_{-d_-+1}}{z^{d_--1}}+\cdots+L_{d_+-1}z^{d_+-1}+L_{d_+} z^{d_+},
\end{align*}
with $L_{\pm d_\pm}\neq 0$ and $d_\pm\in\N_0$, we say that its degrees are
\begin{align*}
		\deg^\pm L&\coloneq d_{\pm}.
\end{align*}

\begin{lemma}[Degrees of the CMV  Laurent
polynomials]
\label{degreesMLP} The Laurent polynomials at the \( n \)-th entry, \( n \in \mathbb{N}_0 \), of every semi-infinite column vector \( B^{(b)}(z) \), \( b \in \{1,\ldots ,q\} \), of the CMV matrix Laurent polynomial \( B(z) \), satisfy
\begin{align*}
\begin{aligned}
		\deg ^{+}B_{n}^{(b)} &\leq \left\lceil \frac{n+2-b}{2q}\right\rceil -1, & 
	\deg ^{-}B_{n}^{(b)} &\leq \left\lceil \frac{n+2-b-q}{2q}\right\rceil,
\end{aligned}
\end{align*}%
and the \( \deg^+ \) and \( \deg^- \) bounds are attained for \( n=2Nq+a-1 \) and \( n=(2N+1)q+a-1 \), $N\in\N_0$,  respectively. Similarly, the Laurent polynomials at the \( n \)-th entry (for \( n\in\mathbb{N}_0 \)) of every semi-infinite row vector \( A^{(a)}(z) \), \( a\in \{1,\ldots ,p\} \) of the CMV matrix Laurent polynomial \( A(z) \), satisfy
\begin{align*}
	\begin{aligned}
		\deg^{+}A_{n}^{(a)} &\leq \left\lceil \frac{n+2-a}{2p}\right\rceil -1, & 
		\deg^{-}A_{n}^{(a)} &\leq \left\lceil \frac{n+2-a-p}{2p}\right\rceil,
	\end{aligned}
\end{align*}%
and the \( \deg^+ \) and \( \deg^- \) bounds are attained for \( n=2Np+a-1 \) and \( n=(2N+1)p+a-1 \), $N\in\N_0$, respectively.
\end{lemma}

\begin{proof}
Concerning the degrees of the Laurent polynomials $B^{(b)}(z)$, $b\in
\{1,\ldots ,q\}$ we multiply the semi-infinite lower triangular matrix $L$
by $Z_{[q]}(z)$ in order to obtain every block polynomial $B_{m}(z) $, and after this computation, it is a simple matter to obtain the
the possible maximum $(\deg ^{+})$ and minimum $(\deg ^{-})$
degrees for the polynomials in the $b$-th column of $B(z)$. Concerning the degrees of the Laurent polynomials $A_{n}^{(a)}(z^{-1})$, $%
a\in \{1,\ldots ,p\}$, similar arguments apply. See Appendix.
\end{proof}

Likewise, based on the Gauss--Borel factorization \eqref{LUMMright} of the
CMV left moment matrix $\mathcal M$, we proceed with the
following

\begin{Definition}[CMV right Laurent polynomials]
\label{CMVBlockPolBA copy(1)} Associated with the left moment matrix $\M$, let us define the matrix biorthogonal
polynomials%
\begin{align*}
\begin{aligned}
	\A(z)&\coloneq\mathcal{L}Z_{[q]}(z),  
&\B(z)&\coloneq Z_{[p]}^{\top }(z)\mathcal{U},
\end{aligned}
\end{align*}
which have size $\infty \times q$ and $p \times \infty$ respectively.
More specifically, 
\begin{align*}
 \begin{aligned}
		\A(z)&=\left[\begin{NiceMatrix}
			{\L}_{0,0} & 0_{q} & 0_{q} & 0_{q} & \Cdots \\ 
			{\L}_{1,0} & {\L}_{1,1} & 0_{q} & 0_{q} & \Cdots \\ 
			{\L}_{2,0} & {\L}_{2,1} & {\L}_{2,2} & 0_{q} & \Cdots
			\\ 
			{\L}_{3,0} & {\L}_{3,1} & {\L}_{3,2} & {\L}_{3,3}
			& \Cdots \\ 
			\Vdots & \Vdots & \Vdots & \Vdots & %
		\end{NiceMatrix}\right]%
		\begin{bmatrix}
			I_{q} \\ 
			z^{-1}I_{q} \\ 
			zI_{q} \\ 
			z^{-2}I_{q} \\ 
			z^{2}I_{r} \\ 
			\vdots%
		\end{bmatrix}
		=%
		\begin{bmatrix}
			{\A}_{0}(z) \\ 
			{\A}_{1}(z) \\ 
			{\A}_{2}(z) \\ 
			{\A}_{3}(z) \\ 
			\vdots%
		\end{bmatrix}%
		=%
		\begin{bmatrix}
			\A^{(1)}(z) & \A^{(2)}(z) & \cdots & \A^{(q)}(z)%
		\end{bmatrix},%
	\end{aligned}  
\end{align*}
where ${\L}_{i,j}$, and ${\A}%
_{m}(z) $, are $q\times q$ square complex matrices.  $i=j$, ${\L}_{i,i}$\
is a lower triangular complex matrix. When $i<j$, $\L_{i,j}=0_{q}$. When $i>j$, ${\L}_{i,j}$\ is,
in general, a full matrix, and 
\begin{align*}
\begin{aligned}
		\A^{(b)}(z)&\coloneq LZ_{[q]}^{(b)}(z),& b\in \{1,\ldots ,q\},
\end{aligned}
\end{align*}%
is the $b$-th semi-infinite column vector of the above CMV matrix Laurent
polynomial $\A(z)$. For $m\in\N_0 $ we also consider the
\textquotedblleft block polynomials\textquotedblright , which are $q\times q$
square matrices with polynomial entries, as follows%
\begin{align*}
	{\A}_{m}(z)\coloneq%
	\begin{bNiceMatrix}
		\A_{mq}^{(1)}(z) & \Cdots & \A_{mq}^{(q)}(z) \\ 
		\Vdots[shorten-end=-4pt] &  & \Vdots[shorten-end=-4pt] \\ 
		\A_{(m+1)q-1}^{(1)}(z) & \Cdots & \A_{(m+1)q-1}^{(q)}(z)%
	\end{bNiceMatrix}.%
\end{align*}%
We also  have%
\begin{align*}
	\B(z)&=%
	\begin{bmatrix}
		I_{p} & z^{-1}I_{p} & zI_{p} & z^{-2}I_{p} & z^{2}I_{r} &\cdots[shorten-end=6pt]%
	\end{bmatrix}%
	\left[\begin{NiceMatrix}
		{\U}_{0,0} & {\U}_{0,1} & {\U}_{0,2} & {\U}_{0,3}
		& \Cdots \\ 
		0_{p} & {\U}_{1,1} & {\U}_{1,2} & {\U}_{1,3} & \Cdots
		\\ 
		0_{p} & 0_{p} & {\U}_{2,2} & {\U}_{2,3} & \Cdots \\ 
		0_{p} & 0_{p} & 0_{p} & {\U}_{3,3} & \Cdots \\ 
		\Vdots & \Vdots & \Vdots & \Vdots & %
	\end{NiceMatrix}\right]
\\
	&=%
	\begin{bmatrix}
		{\B}_{0}(z) & {\B}_{1}(z) & {\B}_{2}(z)&& \cdots[shorten-end=6pt]%
	\end{bmatrix}%
	=%
	\begin{bmatrix}
		\B^{(1)}(z) \\ 
		\B^{(2)}(z) \\ 
		\vdots \\ 
		\B^{(p)}(z)%
	\end{bmatrix},
	\notag
\end{align*}%
where  ${\U}_{i,j}$, $\B_m(z)$, are $%
p\times p$ square matrices. When $i=j$, ${\U}_{i,i}$\ is an upper
triangular matrix. When $i>j$, ${\U}_{i,j}=0_{p}$. When $i<j$, ${\U}_{i,j}$\ is, in
general, a full matrix, and 
\begin{align*}
	\begin{aligned}
		\B^{(a)}(z)&\coloneq\left( Z_{[p]}^{(a)}(z)\right) ^{\top }\U,&
		a&\in \{1,\ldots ,p\} ,
	\end{aligned}
\end{align*}%
is the $a$-th row of the above CMV matrix Laurent polynomial $\B(z)$.
Thus, for $m\in\N_0$, we also have the corresponding \textquotedblleft
block polynomials\textquotedblright\ ${\B}_{m}(z)$, which are
the following $p\times p$ square matrices with polynomial entries%
\begin{align*}
	{\B}_{m}(z)\coloneq%
	\begin{bmatrix}
		\B_{mp}^{(1)}(z) & \cdots & \B_{(m+1)p-1}^{(1)}(z) \\ 
		\Vdots[shorten-end=2pt] &  & \Vdots[shorten-end=-4pt] \\ 
		\B_{mp}^{(p)}(z) & \cdots & \B_{(m+1)p-1}^{(p)}(z)%
	\end{bmatrix}.%
\end{align*}
\end{Definition}
As in Lemma \ref{degreesMLP} we have
\begin{lemma}[Degrees of the CMV right Laurent
	polynomials]
The Laurent polynomials at the \( n \)-th entry, \( n \in \mathbb{N}_0 \), of every semi-infinite column vector \( B^{(b)}(z) \), \( b \in \{1,\ldots ,q\} \), of the CMV matrix Laurent polynomial \( B(z) \), satisfy
\begin{align*}
\begin{aligned}
		\deg ^{+}\B_{n}^{(b)} &\leq \left\lceil \frac{n+2-b}{2q}\right\rceil -1, & 
	\deg ^{-}\B_{n}^{(b)} &\leq \left\lceil \frac{n+2-b-q}{2q}\right\rceil,
\end{aligned}
\end{align*}%
and the \( \deg^+ \) and \( \deg^- \) bounds are attained for \( n=2Nq+a-1 \) and \( n=(2N+1)q+a-1 \), respectively. Similarly, the Laurent polynomials at the \( n \)-th entry (for \( n\in\mathbb{N}_0 \)) of every semi-infinite row vector \( A^{(a)}(z) \), \( a\in \{1,\ldots ,p\} \) of the CMV matrix Laurent polynomial \( A(z) \), satisfy
\begin{align*}
	\begin{aligned}
		\deg^{+}\A_{n}^{(a)} &\leq \left\lceil \frac{n+2-a}{2p}\right\rceil -1, & 
		\deg^{-}\A_{n}^{(a)} &\leq \left\lceil \frac{n+2-a-p}{2p}\right\rceil,
	\end{aligned}
\end{align*}%
and the \( \deg^+ \) and \( \deg^- \) bounds are attained for \( n=2Np+a-1 \) and \( n=(2N+1)p+a-1 \), respectively.
\end{lemma}
According to Proposition \ref{pro:LU_sym}, we establish the following:
\begin{pro}
	The following relations between the matrix Laurent polynomials hold true:
\[	\begin{aligned}
		B_\mu(z)&=\A_{\bar{\mu}}(z), &	 A_\mu(z)&=\B_{\bar{\mu}}(z), \\ B_\mu^\top(z)&=\B_{\mu^\top}(z), &	A_\mu^\top(z)&=\A_{\mu^\top}(z),\\
		( B_\mu(\bar z))^\dagger&=A_{\mu^\dagger}(z), &	( \B_\mu(\bar z))^\dagger&=\A_{\mu^\dagger}(z).
	\end{aligned}\]
\end{pro}
Note the subindex indicating the associated measure to the orthogonal polynomial sequence.

In instances where the matrix of measures exhibits symmetry, the corresponding sequences of orthogonal polynomials are aligned:
\begin{pro}\label{pro:sym-poly}
	\begin{enumerate}
		\item If the matrix of measures is real, i.e., \( \bar{\mu}=\mu \), then
		\begin{align*}
			\begin{aligned}
				\A&= B, &  \B&= A.
			\end{aligned}
		\end{align*}
		\item When \( p=q \), and the matrix of measures is symmetric, i.e., \( \mu^\top=\mu \), then
		\begin{align*}
			\begin{aligned}
				B^\top&=\B, & A^\top&=\A.
			\end{aligned}
		\end{align*}
		\item  When \( p=q \), and the matrix of measures is Hermitian, i.e., \( \mu^\dagger=\mu \), then
		\begin{align*}
			\begin{aligned}
				(A(\bar z))^\dagger&=B(z), & (\A(\bar z))^\dagger&= \B(z).
			\end{aligned}
		\end{align*}
		\item  When \( p=q \), for real symmetric matrix of measures, i.e., \( \bar \mu=\mu \) and \( \mu^\top=\mu \), we have
		\begin{align*}
			\begin{aligned}
				A&= B^\top, & \A&=B, & \B=B^\top.
			\end{aligned}
		\end{align*}
	\end{enumerate}
\end{pro}

\subsection{Orthogonality properties}

We now present the basic orthogonality  properties of these Laurent polynomials due to their connection to the Gauss--Borel factorization.

\begin{pro}[Matrix biorthogonality]
	\label{PropMatrBiort} The biorthogonality between the matrix Laurent polynomials \( B(z) \) and \( A(z^{-1}) \) can be expressed as follows:
\[	\begin{aligned}
		\oint_\mathbb{T}B(z)\d\mu (z) \bar A(z^{-1})&=I, &
		\oint_\mathbb{T}\bar{\A}(z^{-1}) \d\mu (z) \mathcal{B}(z)=I,  
	\end{aligned}\]
	where \( I \) is the semi-infinite identity matrix, i.e., \( I_{m,n}=\delta_{n,m} \).
	
	Expanding this relation entrywise using definitions \eqref{Bpol(b)} and \eqref{Apol(a)}, we obtain:
\[	\begin{aligned}
		\sum_{b=1}^{q}\sum_{a=1}^{p}\oint_{\mathbb{T}}B_{n}^{(b)}(z)\d\mu
		_{b,a}(z)\overline{A_{m}^{(a)}(z)}&=\delta _{n,m},&
		\sum_{b=1}^{q}\sum_{a=1}^{p}\oint_{\mathbb{T}}\overline{\A_{n}^{(b)}(z)}\d\mu
		_{b,a}(z)\B_{m}^{(a)}(z)&=\delta _{n,m},
	\end{aligned}\]
	for \( n,m\in \mathbb{N}_0 \).
\end{pro}

\begin{proof}
	From \eqref{LUMMleft} and \eqref{LUMMright}, we have:
\[	
\begin{aligned}
		I&=\oint_{\mathbb{T}}LZ_{[q]}(z)\d\mu (z) Z_{[p]}^{\top
		}(z^{-1})\bar U, &
		I&=\oint_{\mathbb{T}}\bar{\L}Z_{[q]}(z^{-1})\d\mu (z) Z_{[p]}^{\top
		}(z) \U,
	\end{aligned}
	\]
	respectively. Now, taking into account Definitions \ref{CMVBlockPolBA} and \ref{CMVBlockPolBA copy(1)}, we obtain \( \oint_\mathbb{T}B(z)\d\mu (z) \bar A(z^{-1})=I \) and \( \oint_\mathbb{T}\bar{\A}(z^{-1}) \d\mu (z) \mathcal{B}(z)=I \). Entrywise biorthogonality follows directly from these relations.
\end{proof}

The matrix biorthogonality above has a direct extension to the block polynomials \eqref{BlockPolB} and \eqref{BlockPolA}, from which \( B(z) \) and \( A(z^{-1}) \) are formed. Thus, we establish the following corollary. The proof is omitted as it follows straightforwardly from the previous proposition.

\begin{Definition}[Truncated matrix polynomials]
	For \( r\in \mathbb{N} \), let \( {B}_{n}^{[r]}(z) \) denote the \( r\times q \) matrix of polynomial entries obtained from the \( \infty \times q \) matrix \( B(z) \) by taking consecutive matrices of size \( r\times q \). Similarly, let \( {A}_{n}^{[r]}(z^{-1}) \) denote the \( p\times r \) matrix of polynomial entries obtained from the \( p\times \infty \) matrix \( A(z^{-1}) \) by taking consecutive matrices of size \( p\times r \).
\end{Definition}

\begin{Corollary}[Matrix biorthogonality of block polynomials]
	The following biorthogonality relation holds:
	\begin{align}
		\oint_{\mathbb{T}}{B}_{n}^{[r]}(z)\d\mu (z) {\bar A}_{m}^{[r]}(z^{-1})=\delta _{n,m}I_{r},\quad n,m\in \mathbb{N}_0.
		\label{rMatrBiorth}
	\end{align}
	Moreover, when \( q=p \), we have
\[	\begin{aligned}
		\oint_{{T}}{B}_{n}(z)\d\mu (z) \overline{{A}_{m}(
			z)}&=\delta _{n,m}\,I_{p},& n,m&\in \mathbb{N}_0.
	\end{aligned}\]
	On the other hand, when \( q\neq p \), we obtain
\[	\begin{aligned}
		\oint_{{T}}B^{[p]}_{n}(z)\d\mu (z) \overline{{A}_{m}(z)}&=\delta _{n,m}\,I_{p},&
		\oint_{{T}}{B}_{n}(z)\d\mu (z) \overline{{A}^{[q]}_{m}(z)}&=\delta _{n,m}\,I_{q},
	\end{aligned}\]
	for \( n,m\in \mathbb{N}_0 \).
\end{Corollary}

\begin{Remark}
	Even in the case \( p=q \), the orthogonality presented here differs from that studied in the work \cite{am3}. What we present here is a Gauss factorization in the step-line of the moment matrices (see \eqref{LUMMleft} and \eqref{LUMMright}), whereas in \cite{am3}, the authors performed a block decomposition of the corresponding moment matrices, and those blocks do not appear in the present analysis.
\end{Remark}

This leads to the following natural definition of mixed-type multiple orthogonality on the unit circle.

\begin{Corollary}[Mixed multiple orthogonality]
	The orthogonality relations for the CMV matrix Laurent polynomials can be written as follows:
	\begin{align*}
		\begin{aligned}
			\oint_{\mathbb{T}}\left(\sum_{b=1}^{q}B_{n}^{(b)}(z)\d\mu _{b,a}(z)\right)\overline{	z^{k}}&=0, & k&\in\left\{-\left\lceil \frac{n+2-a}{2p}\right\rceil +1,\dots,{\left\lceil \frac{n+2-a-p}{2p}\right\rceil }\right\},& a&\in \{1,\ldots ,p\}, 
		\end{aligned}\\
		\begin{aligned}
			\oint_{\mathbb{T}}z^{k}\left(\sum_{a=1}^{p}\overline{A_{n}^{(a)}(z)}\d\mu
			_{b,a} (z)\right)&=0,   &
			k&\in\left\{-{\left\lceil \frac{n+2-b-q}{2q}\right\rceil , \cdots,\left\lceil \frac{n+2-b}{2q}\right\rceil -1}\right\},& b&\in \{1,\ldots ,q\},
		\end{aligned}\\
		\begin{aligned}
			\oint_{\mathbb{T}}\overline{z^{k}}\left(\sum_{a=1}^{p}{\B_{n}^{(a)}(z)}\d\mu
			_{b,a} (z)\right)&=0,   &
			k&\in\left\{-\left\lceil \frac{n+2-b}{2q}\right\rceil +1, \cdots,\left\lceil \frac{n+2-b-q}{2q}\right\rceil \right\},& b&\in \{1,\ldots ,q\},
		\end{aligned}\\
		\begin{aligned}
			\oint_{\mathbb{T}}	\left(\sum_{b=1}^{q}\overline{\A_{n}^{(b)}(z)}\d\mu _{b,a}(z)\right)z^{k}&=0, & k&\in\left\{-{\left\lceil \frac{n+2-a-p}{2p}\right\rceil },\dots, \left\lceil \frac{n+2-a}{2p}\right\rceil -1\right\},& a&\in \{1,\ldots ,p\}.
		\end{aligned}
	\end{align*}
\end{Corollary}

Therefore, with these orthogonality and biorthogonality properties in mind, we can aptly refer to these polynomials as mixed Multiple Orthogonal Laurent Polynomials on the Unit Circle (MOLPUC).

Along the article we will consider following Cauchy transforms of the CMV matrix Laurent polynomials, that is, the second kind functions
\begin{align}
	D(z) =&\oint_{\mathbb{T}}\frac{1}{z-u^{-1}}B(u)\d\mu (u), \label{skfD}\\
	C(z) =&\oint_{\mathbb{T}}\frac{1}{z^{-1}-u}\d\mu (u)A(u^{-1}),  \label{skfC}
\end{align}%
Observe that the size of $D(z)$  is $\infty \times p$, and the size of $C(z)$ i s $q\times \infty $.

\subsection{Recurrence relations}

To establish recurrence relations, we introduce the following matrices, which are of significant interest in the theory of mixed MOLPUC.

\begin{Definition}
	For any \( r \in \mathbb{N} \), and denoting \( I_{r} \) and \( 0_{r} \) as the \( r \times r \) identity matrix and the \( r \times r \) zero matrix respectively, we define the following semi-infinite block matrices:
	\begin{align}
		\Upsilon_{[r]} & =\left[\begin{NiceMatrix}
			\Block[borders={bottom,right,tikz=solid}]{2-2}{} 0_{r} & 0_{r} &
			\Block[borders={bottom,right,tikz=solid}]{2-2}{} I_{r} & 0_{r} &
			\Block[borders={bottom,right,tikz=solid}]{2-2}{} 0_{r} & 0_{r} &
			\Block[borders={bottom,right,tikz=solid}]{2-2}{} 0_{r} & 0_{r} &
			\Block[borders={bottom,tikz=solid}]{2-1}{} \cdots[shorten-start=6pt] \\ 
			I_{r} & 0_{r} & 0_{r} & 0_{r} & 0_{r} & 0_{r} & 0_{r} & 0_{r} & \cdots[shorten-start=6pt] \\
			\Block[borders={bottom,right,tikz=solid}]{2-2}{} 0_{r} & 0_{r} &
			\Block[borders={bottom,right,tikz=solid}]{2-2}{} 0_{r} & 0_{r} &
			\Block[borders={bottom,right,tikz=solid}]{2-2}{} I_{r} & 0_{r} &
			\Block[borders={bottom,right,tikz=solid}]{2-2}{} 0_{r} & 0_{r} &
			\Block[borders={bottom,tikz=solid}]{2-1}{} \cdots[shorten-start=6pt] \\ 
			0_{r} & I_{r} & 0_{r} & 0_{r} & 0_{r} & 0_{r} & 0_{r} & 0_{r} & \cdots[shorten-start=6pt] \\
			\Block[borders={bottom,right,tikz=solid}]{2-2}{} 0_{r} & 0_{r} &
			\Block[borders={bottom,right,tikz=solid}]{2-2}{} 0_{r} & 0_{r} &
			\Block[borders={bottom,right,tikz=solid}]{2-2}{} 0_{r} & 0_{r} &
			\Block[borders={bottom,right,tikz=solid}]{2-2}{} I_{r} & 0_{r} &
			\Block[borders={bottom,tikz=solid}]{2-1}{} \cdots[shorten-start=6pt] \\
			0_{r} & 0_{r} & 0_{r} & I_{r} & 0_{r} & 0_{r} & 0_{r} & 0_{r} & \cdots[shorten-start=6pt] \\
			\Block[borders={bottom,right,tikz=solid}]{2-2}{} 0_{r} & 0_{r} &
			\Block[borders={bottom,right,tikz=solid}]{2-2}{} 0_{r} & 0_{r} &
			\Block[borders={bottom,right,tikz=solid}]{2-2}{} 0_{r} & 0_{r} &
			\Block[borders={bottom,right,tikz=solid}]{2-2}{} I_{r} & 0_{r} & 
			\Block[borders={bottom,tikz=solid}]{2-1}{} \cdots[shorten-start=6pt] \\
			0_{r} & 0_{r} & 0_{r} & 0_{r} & 0_{r} & I_{r} & 0_{r} & 0_{r} & \cdots[shorten-start=6pt] \\
			\Block[borders={right,tikz=solid}]{1-2}{} \vdots[shorten=5pt] & \vdots[shorten=5pt] &
			\Block[borders={right,tikz=solid}]{1-2}{} \vdots[shorten=5pt] & \vdots[shorten=5pt] &
			\Block[borders={right,tikz=solid}]{1-2}{}
			\vdots[shorten=5pt] & \vdots[shorten=5pt] & 
			\Block[borders={right,tikz=solid}]{1-2}{} \vdots[shorten=5pt] & \vdots[shorten=5pt] &  
		\end{NiceMatrix}\right],  
		\label{Upsilon(r)}
	\end{align}
	the intertwining matrix:
	\begin{align}
		\eta_{[r]} & =\left[\begin{NiceMatrix}
			\Block[borders={bottom,right,tikz=solid}]{1-1}{} I_{r} &
			\Block[borders={bottom,right,tikz=solid}]{1-2}{} 0_{r} & 0_{r} &
			\Block[borders={bottom,right,tikz=solid}]{1-2}{} 0_{r} & 0_{r} &
			\Block[borders={bottom,right,tikz=solid}]{1-2}{} 0_{r} & 0_{r} &
			\Block[borders={bottom,tikz=solid}]{1-1}{} \cdots[shorten-start=6pt] \\
			\Block[borders={bottom,right,tikz=solid}]{2-1}{} 0_{r} &
			\Block[borders={bottom,right,tikz=solid}]{2-2}{} 0_{r} & I_{r} &
			\Block[borders={bottom,right,tikz=solid}]{2-2}{} 0_{r} & 0_{r} &
			\Block[borders={bottom,right,tikz=solid}]{2-2}{} 0_{r} & 0_{r} &
			\Block[borders={bottom,tikz=solid}]{2-1}{} \cdots[shorten-start=6pt] \\ 
			0_{r} & I_{r} & 0_{r} & 0_{r} & 0_{r} & 0_{r} & 0_{r} & \cdots[shorten-start=6pt] \\
			\Block[borders={bottom,right,tikz=solid}]{2-1}{} 0_{r} &
			\Block[borders={bottom,right,tikz=solid}]{2-2}{} 0_{r} & 0_{r} &
			\Block[borders={bottom,right,tikz=solid}]{2-2}{} 0_{r} & I_{r} &
			\Block[borders={bottom,right,tikz=solid}]{2-2}{} 0_{r} & 0_{r} &
			\Block[borders={bottom,tikz=solid}]{2-1}{} \cdots[shorten-start=6pt] \\ 
			0_{r} & 0_{r} & 0_{r} & I_{r} & 0_{r} & 0_{r} & 0_{r} & \cdots[shorten-start=6pt] \\
			\Block[borders={bottom,right,tikz=solid}]{2-1}{} 0_{r} &
			\Block[borders={bottom,right,tikz=solid}]{2-2}{} 0_{r} & 0_{r} &
			\Block[borders={bottom,right,tikz=solid}]{2-2}{} 0_{r} & 0_{r} & 
			\Block[borders={bottom,right,tikz=solid}]{2-2}{} 0_{r} & I_{r} &  
			\Block[borders={bottom,tikz=solid}]{2-1}{} \cdots[shorten-start=6pt] \\
			0_{r} & 0_{r} & 0_{r} & 0_{r} & 0_{r} & I_{r} & 0_{r} & \cdots[shorten-start=6pt] \\
			\Block[borders={right,tikz=solid}]{1-1}{} \vdots[shorten=5pt] &
			\Block[borders={right,tikz=solid}]{1-2}{} \vdots[shorten=5pt] & \vdots[shorten=5pt] &
			\Block[borders={right,tikz=solid}]{1-2}{} \vdots[shorten=5pt] & \vdots[shorten=5pt] & \Block[borders={right,tikz=solid}]{1-2}{} \vdots[shorten=5pt] & \vdots[shorten=5pt] & \end{NiceMatrix}\right],  
		\label{Eta(r)}
	\end{align}
		and the untangling matrix:
	\begin{align}
		\nu_{[r]} & =\left[\begin{NiceMatrix}
			\Block[borders={bottom,right,tikz=solid}]{2-2}{} 0_{r} & I_{r} &
			\Block[borders={bottom,right,tikz=solid}]{2-2}{} 0_{r} & 0_{r} &
			\Block[borders={bottom,right,tikz=solid}]{2-2}{} 0_{r} & 0_{r} &
			\Block[borders={bottom,right,tikz=solid}]{2-2}{} 0_{r} & 0_{r} &
			\Block[borders={bottom,tikz=solid}]{2-1}{} \cdots[shorten-start=6pt] \\ 
			I_{r} & 0_{r} & 0_{r} & 0_{r} & 0_{r} & 0_{r} & 0_{r} & 0_{r} & \cdots[shorten-start=6pt] \\
			\Block[borders={bottom,right,tikz=solid}]{2-2}{} 0_{r} & 0_{r} &
			\Block[borders={bottom,right,tikz=solid}]{2-2}{} 0_{r} & I_{r} &
			\Block[borders={bottom,right,tikz=solid}]{2-2}{} 0_{r} & 0_{r} &
			\Block[borders={bottom,right,tikz=solid}]{2-2}{} 0_{r} & 0_{r} &
			\Block[borders={bottom,tikz=solid}]{2-1}{} \cdots[shorten-start=6pt] \\ 
			0_{r} & 0_{r} & I_{r} & 0_{r} & 0_{r} & 0_{r} & 0_{r} & 0_{r} & \cdots[shorten-start=6pt] \\
			\Block[borders={bottom,right,tikz=solid}]{2-2}{} 0_{r} & 0_{r} &
			\Block[borders={bottom,right,tikz=solid}]{2-2}{} 0_{r} & 0_{r} &
			\Block[borders={bottom,right,tikz=solid}]{2-2}{} 0_{r} & I_{r} &
			\Block[borders={bottom,right,tikz=solid}]{2-2}{} 0_{r} & 0_{r} &
			\Block[borders={bottom,tikz=solid}]{2-1}{} \cdots[shorten-start=6pt] \\
			0_{r} & 0_{r} & 0_{r} & 0_{r} & I_{r} & 0_{r} & 0_{r} & 0_{r} & \cdots[shorten-start=6pt] \\
			\Block[borders={bottom,right,tikz=solid}]{2-2}{} 0_{r} & 0_{r} &
			\Block[borders={bottom,right,tikz=solid}]{2-2}{} 0_{r} & 0_{r} &
			\Block[borders={bottom,right,tikz=solid}]{2-2}{} 0_{r} & 0_{r} &
			\Block[borders={bottom,right,tikz=solid}]{2-2}{} 0_{r} & I_{r} & 
			\Block[borders={bottom,tikz=solid}]{2-1}{} \cdots[shorten-start=6pt] \\
			0_{r} & 0_{r} & 0_{r} & 0_{r} & 0_{r} & 0_{r} & I_{r} & 0_{r} & \cdots[shorten-start=6pt] \\
			\Block[borders={right,tikz=solid}]{1-2}{} \vdots[shorten=5pt] & \vdots[shorten=5pt] &
			\Block[borders={right,tikz=solid}]{1-2}{} \vdots[shorten=5pt] & \vdots[shorten=5pt] &
			\Block[borders={right,tikz=solid}]{1-2}{} \vdots[shorten=5pt] & \vdots[shorten=5pt] & 
			\Block[borders={right,tikz=solid}]{1-2}{} \vdots[shorten=5pt] & \vdots[shorten=5pt] &  
		\end{NiceMatrix}\right].  
		\label{Nu(r)}
	\end{align}
\end{Definition}

Notice that $\Upsilon_{[r]}$ is constructed as a sparse banded matrix, with its highest nonzero upper-diagonal and lowest nonzero lower-diagonal both positioned at $(2r+1)$. This characteristic renders $\Upsilon_{[r]}$ a $(4r+1)$-banded matrix, or succinctly, a pentadiagonal block matrix. For instance, when $r=1$, $\Upsilon _{[ 1]}$ forms a pentadiagonal semi-infinite matrix; while for $r=2$, $\Upsilon _{[ 2]}$ manifests as a pentadiagonal $2\times 2$ block matrix, akin to a ninediagonal semi-infinite matrix.

Regarding the successive powers of $\Upsilon_{[r]}$, it's straightforward to verify that $\left( \Upsilon _{[r]}\right) ^{d}$ maintains a sparse banded structure. Here, the highest nonzero upper-diagonal and lowest nonzero lower-diagonal occur at position $(2dr+1)$. This property extends $\left( \Upsilon _{[r]}\right) ^{d}$ to a $\left(4dr+1\right) $-diagonal matrix, effectively forming a pentadiagonal $dr\times dr$ block matrix. The composition of $\left( \Upsilon _{[r]}\right) ^{d}$ holds significance for subsequent analyses.

\begin{pro}
For any given \( r \in \mathbb{N} \), the following factorizations of the matrix \( \Upsilon_{[r]} \) and its transpose hold:
\begin{equation}
	\label{factorMain}\begin{aligned}
	\Upsilon_{[r]} &= \nu_{[r]}\eta_{[r]}, &
	\Upsilon_{[r]}^{\top } &= \eta_{[r]}\nu _{[r]} = \Upsilon_{[r]}^{-1}.
\end{aligned}
\end{equation}
As a direct consequence, the subsequent expressions are valid:
\[\begin{aligned}
	\Upsilon_{[r]}\nu_{[r]} &= \nu_{[r]}\Upsilon_{[r]}^{\top }, &
	\Upsilon_{[r]}^{\top }\eta_{[r]} &= \eta _{[r]}\Upsilon_{[r]},
\end{aligned}\]
and for any \( d \in \mathbb{N} \),
\[\begin{aligned}
	\left( \Upsilon_{[r]}\right) ^{d} &= \nu_{[r]}\left( \Upsilon_{[r]}^{\top }\right) ^{d-1}\eta_{[r]}, &
	\left( \Upsilon_{[r]}^{\top }\right) ^{d} &= \eta _{[r]}\left( \Upsilon_{[r]}\right) ^{d-1}\nu_{[r]}.
\end{aligned}\]
\end{pro}

\begin{proof}
The proof follows directly from the specific form of the semi-infinite matrices involved. Multiplying \( \nu_{[ r]} \) by \( \eta_{[r]} \) in the given order yields \( \Upsilon_{[r]} \). Notably, \( \nu_{[r]}\eta_{[r]} = \Upsilon_{[r]} \) holds whether the multiplication is performed by \( r\times r \) blocks (\( r\geq 2 \)) or in the standard manner. Similarly, multiplying \( \eta _{[r]} \) by \( \nu _{[r]} \) in sequence results in \( \Upsilon_{[r]}^{\top } \). The remaining expressions are straightforward consequences of the initial factorizations.
\end{proof}

Next, we present a series of specific properties concerning the relationships between the matrices \( \Upsilon_{[r]} \), \( \eta_{[r]} \), and \( \nu_{[r]} \):
\begin{pro}
The matrix \( \Upsilon_{[r]} \) is a banded matrix with \( 2r \) superdiagonals and \( 2r \) subdiagonals, while \( \eta_{[r]} \) and \( \nu_{[r]} \) are banded matrices with \( r \) superdiagonals and \( r \) subdiagonals. Additionally, the following matrix properties hold:
\begin{align}
	\Upsilon_{[r]}^{-1} &= \Upsilon_{[r]}^{\top} = \Upsilon_{[r]}^{\dag}, \label{prop01} \\
	\eta_{[r]} &= \eta_{[r]}^{-1} = \eta_{[r]}^{\top} = \eta_{[r]}^{\dag}, \label{prop11} \\
	\nu_{[r]} &= \nu_{[r]}^{-1} = \nu_{[r]}^{\top} = \nu_{[r]}^{\dag}. \label{prop21}
\end{align}
\end{pro}

Now, we delineate their actions on the monomials as outlined in Section \ref{S02-Defs}.

\begin{pro}
	The matrix \( \Upsilon_{[r]} \) complies with the following eigenvalue properties:
	\begin{align}
		\Upsilon_{[r]}Z_{[r]}(z) &= z\,Z_{[r]}(z),  \label{prop02} \\
		\Upsilon_{[r]}^{-1}Z_{[r]}(z) &= z^{-1}\,Z_{[r]}(z),  \label{prop03} \\
		Z_{[r]}^{\top }(z)\Upsilon_{[r]}^{-1} &= z\,Z_{[r]}^{\top }(z), \label{prop04} \\
		Z_{[r]}^{\top }(z)\Upsilon_{[r]} &= z^{-1}\,Z_{[r]}^{\top }(z). \label{prop05}
	\end{align}
	
	The intertwining matrix \( \eta_{[r]} \) conforms to the following properties:
	\begin{align}
		\eta_{[r]}Z_{[r]}(z) &= Z_{[r]}(z^{-1}) = Z_{[r]}(\bar{z}) = \overline{\bar{Z}_{[r]}(z)},  \label{prop12} \\
		\eta_{[r]}Z_{[r]}(z^{-1}) &= Z_{[r]}(z) = \overline{\bar{Z}_{[r]}(z^{-1})},  \label{prop13} \\
		Z_{[r]}^{\top }(z)\eta_{[r]} &= Z_{[r]}^{\top }(z^{-1}) = Z_{[r]}^{\top }(\bar{z}) = \left( Z_{[r]}(z) \right) ^{\dag }, \label{prop14} \\
		Z_{[r]}^{\top }(z^{-1})\eta_{[r]} &= Z_{[r]}^{\top }(z) = \left( Z_{[r]}(z^{-1}) \right) ^{\dag }. \label{prop15}
	\end{align}
	
	The untangling matrix \( \nu_{[r]} \) adheres to the ensuing properties:
	\begin{align}
		\nu_{[r]}Z_{[r]}(z) &= z^{-1}Z_{[r]}(z^{-1}) = \bar{z}\,Z_{[r]}(\bar{z}),  \label{prop22} \\
		\nu_{[r]}Z_{[r]}(z^{-1}) &= zZ_{[r]}(z),  \label{prop23} \\
		Z_{[r]}^{\top }(z^{-1})\nu_{[r]} &= zZ_{[r]}^{\top }(z), \label{prop24} \\
		Z_{[r]}^{\top }(z)\nu_{[r]} &= z^{-1}Z_{[r]}^{\top }(z^{-1}). \label{prop25}
	\end{align}
\end{pro}

\begin{proof}
Property \eqref{prop02} is easily verified from \eqref{Upsilon(r)}, which provides the explicit matrix expression for \( \Upsilon_{[ r]} \). To establish \eqref{prop03}, it suffices to left-multiply both sides of \eqref{prop02} by \( \Upsilon_{[r]}^{-1} \). For \eqref{prop04}, transposing both sides of \eqref{prop02} and considering \( \Upsilon_{[r]}^{\top }=\Upsilon_{[r]}^{-1} \) suffices. Similarly, for \eqref{prop05}, transposing both sides of \eqref{prop03} and considering \( (\Upsilon_{[r]}^{\top })^{\top }=\Upsilon_{[r]} \) is adequate.

As for the intertwining matrix \( \eta_{[r]} \), property \eqref{prop12} is straightforward to confirm from the matrix expression in \eqref{Eta(r)}. Notably, \( \eta_{[r]} \) interchanges \( z \) with \( z^{-1} \) in the argument of the monomial matrix \( Z_{[r]}(z) \), and given that we operate on the unit circle \( z\in \mathbb{T} \), we have \( \bar{z}=z^{-1} \). Property \eqref{prop13} follows the interchange of \( z \) with \( z^{-1} \) in \eqref{prop12}. Upon transposing \eqref{prop12}, property \eqref{prop14} emerges, and considering \( \eta_{[r]}=\eta_{[r]}^{-1} \) completes the proof. Furthermore, combining \eqref{cojugaZ} with property \eqref{prop14} yields \eqref{prop15} straightforwardly.

On the other hand, for the untangling matrix \( \nu_{[r]} \), property \eqref{prop22} results from considering \eqref{factorMain} in expression \eqref{prop02} and subsequently applying \eqref{prop12}. Property \eqref{prop23} ensues from the interchange of \( z \) with \( z^{-1} \) in \eqref{prop22}. Transposing \eqref{prop23} yields property \eqref{prop24}, and considering \eqref{prop21} completes the proof. Finally, property \eqref{prop25} follows from the interchange of \( z \) with \( z^{-1} \) in \eqref{prop24}.
\end{proof}

Next, we analyze the action of the spectral matrices mentioned above on the moment matrices \( M \) and \( \mathcal{M} \).
\begin{pro}
The matrices \( \Upsilon_{[q]} \), \( \eta_{[ q]} \), and \( \nu_{[ q]} \) satisfy the following properties:
\begin{align}
	\Upsilon_{[q]}M &= M\Upsilon_{[p]}, \label{YMMY1} \\
	\Upsilon_{[q]}^{\top }\mathcal{M} &= \mathcal{M}\Upsilon_{[p]}^{\top }, \label{YMMY2} \\
	\eta_{[ q]}M\eta_{[ p]} &= \nu_{[ q]}M\nu_{[ p]}, \label{YMMY3} \\
	\eta_{[ q]}\mathcal{M}\eta_{[ p]} &= \nu_{[ q]}\mathcal{M}\nu_{[ p]}. \label{YMMY4}
\end{align}
\end{pro}

\begin{proof}
Left-multiplying $\Upsilon_{[q]}$ by the moment matrix $M$ in \eqref{MMleft} yields, 
\begin{align*}
\Upsilon_{[q]}M=\oint_{\mathbb{T}}\Upsilon
_{[ q]}Z_{[q]}(z)\d\mu (z) \left( Z_{[p]}(z)\right) ^{\dag }.
\end{align*}%
From \eqref{prop02} we have%
\begin{align*}
\Upsilon_{[q]}M=\oint_\mathbb{T}zZ_{[q]}(z)\d\mu (z)\left( Z_{[p]}(z)\right) ^{\dag }=\oint_\mathbb{T}Z_{[q]}(z)\d\mu (z) \left( z^{-1}Z_{[p]}(z)\right) ^{\dag }
\end{align*}%
and next from \eqref{prop01} and \eqref{prop04}, we state%
\begin{align*}
\Upsilon_{[q]}M = &\oint_{\mathbb{T}}Z_{[q]}(z)\d\mu (z) \left( z^{-1}Z_{[p]}(z)\right) ^{\dag } \\
=&\oint_{\mathbb{T}}Z_{[q]}(z)\d\mu (z) \left( \Upsilon
_{[ p]}^{\top }Z_{[p]}(z)\right) ^{\dag } \\
=&\oint_{\mathbb{T}}Z_{[q]}(z)\d\mu (z) \left(
Z_{[p]}(z)\right) ^{\dag }\Upsilon_{[p]}=M\Upsilon
_{[ p]}.
\end{align*}%
Proceeding in a similar fashion with the moment matrix $\mathcal M$ given in \eqref{MMright}, we find \eqref{YMMY2}. Combining \eqref{factorMain} with $\Upsilon_{[q]}M=M\Upsilon_{[ p]}$ leads to $\nu _{[ q]}\eta _{[ q]}M= M\nu _{[ p]}\eta _{[ p]}$, and next multiplying both sides by $\nu _{[ q]}=\nu _{[ q]}^{-1}$ and $\eta_{[r]}=\eta_{[r]}^{-1}$ yields \eqref{YMMY3}, and proceeding in a similar fashion for $\Upsilon_{[q]}^{\top }\mathcal M=\mathcal M\Upsilon_{[p]}^{\top }$ we obtain \eqref{YMMY4}.
\end{proof}
Now, let's examine the effect of \( \eta_{[ q]} \) and \( \nu_{[ q]} \) on the moment matrices \( M \) and \( \mathcal{M} \).

\begin{pro}
The matrices \( \eta_{[r]} \) and \( \nu_{[r]} \) exhibit the following relationships with the moment matrices \( M \) and \( \mathcal{M} \):
\begin{align}
	\eta _{[ q]}M &=\mathcal M\eta _{[ p]},
	\label{simeta01} \\
	M\eta _{[ p]} &=\eta _{[ q]}\mathcal M,
	\label{simeta02} \\
	\nu _{[ q]}M &=\mathcal M\nu _{[ p]},
	\label{simnu01} \\
	M\nu _{[ p]} &=\nu _{[ q]}\mathcal M.
	\label{simnu02}
\end{align}
\end{pro}

\begin{proof}
From \eqref{MMleft} and \eqref{prop12}, we immediately observe that
\begin{align*}
	\eta _{[ q]}M &= \oint_{\mathbb{T}}\eta_{[ q]}Z_{[q]}(z)\,\mathrm{d}\mu (z) Z_{[p]}^{\top }(z^{-1}) \\
	&= \oint_{\mathbb{T}}Z_{[q]}(z^{-1})\,\mathrm{d}\mu (z) Z_{[p]}^{\top }(z^{-1}).
\end{align*}
Considering \eqref{prop13}, we conclude
\begin{align*}
	\eta _{[ q]}M &= \oint_{\mathbb{T}}Z_{[q]}(z^{-1})\,\mathrm{d}\mu (z) Z_{[p]}^{\top }(z)\eta_{[ p]} = \mathcal M\eta_{[ p]}.
\end{align*}
Proceeding similarly to compute \(M\eta_{[ p]}\) and considering \eqref{prop12} and \eqref{prop15}, we obtain \eqref{simeta02}. Obtaining \eqref{simnu01} is straightforward by considering \eqref{prop22} and \eqref{prop25}. Thus, from \eqref{prop11} and \eqref{prop21} we have
\[\begin{aligned}
	M\eta_{[ p]} &= \eta_{[ q]}\mathcal M, & \nu_{[ q]}M &= \mathcal M\nu_{[ p]}.
\end{aligned}\]
Finally, combining \eqref{simeta02} with \eqref{factorMain} and \eqref{YMMY3} gives \eqref{simnu02}.
\end{proof}



In order to obtain recurrence relations for the mixed MOLPUC $B(z)$ and $A(z^{-1})$, we first consider the
following relations and definitions. From \eqref{YMMY1} and \eqref{LUMMleft} we have
\begin{align*}
\Upsilon_{[q]}L^{-1}\bar U^{-1}=L^{-1}\bar U^{-1}\Upsilon_{[p]}.
\end{align*}%
We then right-multiply by $\bar U^{-1}$ and left-multiply by $L^{-1}$ the above
expression, wich gives%
\begin{align*}
L\Upsilon_{[q]}L^{-1}=\bar U^{-1}\Upsilon_{[p]}\bar U.
\end{align*}%
We follow a similar approach from \eqref{YMMY2} and \eqref{LUMMright} to get%
\begin{align*}
\bar{\mathcal{L}}\Upsilon_{[q]}^{\top }\bar{\mathcal{L}}^{-1}=\mathcal{U}%
^{-1}\Upsilon_{[p]}^{\top }\mathcal{U}.
\end{align*}%
This leads to our next

\begin{Definition}
Let us define the matrices%
\begin{align}
T\coloneq L\Upsilon_{[q]}L^{-1} =&\bar U^{-1}\Upsilon_{[p]}\bar U,
\label{Tmatrix} \\
\Tscr \coloneq\bar{\mathcal{L}}\Upsilon_{[q]}^{\top }\bar{\L}^{-1} =&\mathcal{U}^{-1}\Upsilon_{[p]}^{\top }\mathcal{U}.
\label{Tmatrixwave}
\end{align}%
The matrices $T$ and $\Tscr $, have respectively the structures given in \eqref{bandedmatrixv01}.
\end{Definition}

\begin{pro}[Banded structure of the recursion matrix]\label{bandedmatrixv01}
\begin{enumerate}
	\item 	The matrix $T$ is a banded matrix with $2(p+q)+1$ diagonals, $2q$ superdiagonals and $2p$ subdiagonals.
\item 	The matrix $\Tscr$ is a banded matrix with $2(p+q)+1$ diagonals, $2p$ superdiagonals and $2q$ subdiagonals.	
\end{enumerate}
\end{pro}

\begin{proof}
 \eqref{Tmatrix}  and the band structure of $\Upsilon_{[q]}$ we deduce that $T=L\Upsilon_{[q]}L^{-1}$ has all its superdiagonals above the $2q$ superdiagonal with zero entries. We also deduce that $T=\bar U^{-1}\Upsilon_{[p]}\bar U$ has all its subdiagonals below the $2q$ subdiagonal with zero entries.  For $\Tscr$'s band structure we have a similar argument starting from  \eqref{Tmatrixwave}.
\end{proof}
Next, let's illustrate the band structure of \( T \):
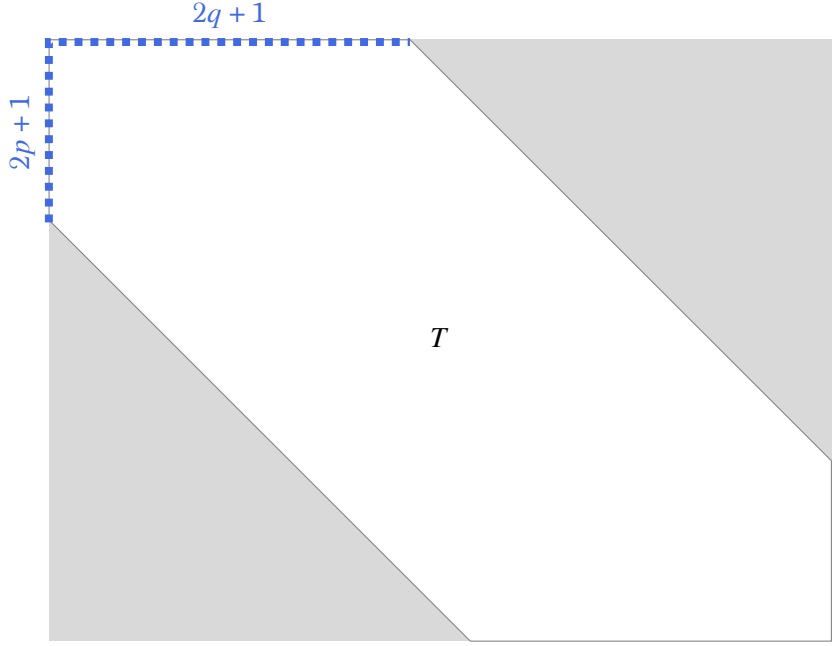
	\begin{center}
	\vspace*{.7cm}
	\drawmatrixset{bbox style={fill=Gray!30}}
	$\drawmatrix[bbox/.append, ,scale=8,banded, width=1.3,
	]T$

	\vspace*{-8.66cm}
	\hspace*{-6.2cm}	
	\begin{tikzpicture}
		\draw[ line width=3pt,dashed,RoyalBlue] (0,6)--(0,8.4) node [midway, above, sloped]  {$2p+1$}--(4.8,8.4) node [midway, above, sloped]  {$2q+1$} ;
	\end{tikzpicture}
	\vspace*{6cm}
\end{center}

With the above definitions in place, we can establish the following recurrence relations for \( B(z) \) and \( A(z^{-1}) \):
\begin{pro}[Recurrence Relations]
The mixed MOLPUC matrices \( B(z) \)
and \( A(z) \) satisfy the following recurrence relation, taking into
account the aforementioned matrix \( T \):
\begin{align}
	TB(z)&=z\,B(z),  \label{TBz} \\
	\bar A(z^{-1})T&=z\bar A(z^{-1}).  \label{AzT}
\end{align}%
Concerning the right mixed type multiple orthogonal polynomials on the step line \( \mathcal{B}(z) \) and \( \A(z) \), and the matrix \( \Tscr \), we find the analogous relations:
\begin{align}
	\Tscr  \bar{\A}(z^{-1}) =&{z}\, \bar{\A}(z^{-1}),  \label{TBz-wave}
	\\
	\mathcal{B}(z)\Tscr  =&z\,\mathcal{B}(z).
	\label{AzT-wave}
\end{align}
\end{pro}

\begin{proof}
From \eqref{Tmatrix}, we know that the action of \( T \) on the CMV left-matrix Laurent polynomials \( B(z) \) is as follows:
\begin{align*}
	TB(z) =&L\Upsilon_{[q]}L^{-1}LZ_{[q]}(z)=L\Upsilon _{[
		q]}Z_{[q]}(z) \\
	=&zLZ_{[q]}(z)=z\,B(z),
\end{align*}
which provides a matrix form recurrence relation for the biorthogonal polynomials \( B(z) \). The action of \( T \) on the left-hand side of the
biorthogonal polynomials \( \bar A(z^{-1}) \), taking into account \eqref{Tmatrix}, \eqref{prop01}, and \eqref{prop03}, is:
\begin{align*}
	\bar A(z^{-1})T =&Z_{[p]}^{\top }(z^{-1})\bar U\bar U^{-1}\Upsilon _{[
		p]}\bar U=Z_{[p]}^{\top }(z^{-1})\Upsilon_{[p]}\bar U \\
	=&\left( \Upsilon_{[p]}^{\top }Z_{[p]}(z^{-1})\right) ^{\top
	}\bar U=z\,Z_{[p]}^{\top }(z^{-1})\bar U=z\bar A(z^{-1}).
\end{align*}

Concerning mixed MOLPUCs \( \mathcal{A}(z^{-1}) \) and \( \mathcal{B}(z) \), we conclude:
\begin{align*}
	\Tscr \bar{\A}(z^{-1}) =& \bar{\L}\Upsilon_{[q]}^{\top }%
	\bar{\L}^{-1}\bar{\L}Z_{[q]}(z^{-1})=\bar{\L}\Upsilon _{[
		q]}^{\top }Z_{[q]}(z^{-1}) \\
	=& z\bar{\L}Z_{[q]}(z^{-1})=z\bar{\A}(z^{-1}),
\end{align*}
and
\begin{align*}
	\mathcal{B}(z)\Tscr  =& Z_{[p]}^{\top }(z)\mathcal{U}%
	\mathcal{U}^{-1}\Upsilon_{[p]}^{\top }\mathcal{U}=Z_{[p]}^{\top
	}(z)\Upsilon_{[p]}^{\top }\mathcal{U} \\
	=& \left( \Upsilon_{[p]}Z_{[p]}(z)\right) ^{\top }\mathcal{U}%
	=z Z_{[r]}^{\top }(z)\mathcal{U}={z}\,\mathcal{B}(z).
\end{align*}

\end{proof}



\subsection{Szeg\H{o} recurrence matrices}

Associated to the intertwining and untangling matrices,  see \eqref{Eta(r)} and \eqref{Nu(r)}, and the symmetries of the moment matrices as described in \eqref{YMMY3} and \eqref{YMMY4} we have:
\begin{pro}
The following identities hold:
\[	\begin{aligned}
	\bar{\L}\eta _{[ q]}L^{-1} =&{\U}^{-1}\eta
		_{[ p]}\bar U,&
		 L\eta _{[ q]}\bar{\L}%
		^{-1}=\bar U^{-1}\eta _{[ p]}\U,   \\
		 L\nu _{[ q]}\bar{\L}^{-1} =&\bar U^{-1}\nu _{[ p]}%
		\U, & \bar{\L}\nu _{[q]}L^{-1}=\U^{-1}\nu _{[ p]}\bar U. 
	\end{aligned}\]
\end{pro}

\begin{proof}
From equations \eqref{simeta01} and \eqref{simnu01}, we deduce:
\[\begin{aligned}
	\eta_{[q]}L^{-1}\bar U^{-1} &= \bar{\L}^{-1}\U^{-1}\eta_{[p]}, &
	\nu_{[q]}L^{-1}U^{-1} &= \bar{\L}^{-1}\U^{-1}\nu_{[p]},
\end{aligned}\]
which leads to:
\[\begin{aligned}
	\bar{\L}\eta_{[q]}L^{-1} &= \U^{-1}\eta_{[p]}\bar U, &
	\bar{\L}\nu_{[q]}L^{-1} &= \U^{-1}\nu_{[p]}\bar U.
\end{aligned}\]

Similarly, from equations \eqref{simeta02} and \eqref{simnu02}, we derive:
\[\begin{aligned}
	L^{-1}\bar U^{-1}\eta_{[p]} &= \eta_{[q]}\bar{\L}^{-1}\U^{-1}, &
	L^{-1}U^{-1}\nu_{[p]} &= \nu_{[q]}\bar{\L}^{-1}\U^{-1},
\end{aligned}\]
resulting in:
\[\begin{aligned}
	\bar U^{-1}\eta_{[p]}\U &= L\eta_{[q]}\bar{\L}^{-1}, &
	\bar U^{-1}\nu_{[p]}\U &= L\nu_{[q]}\bar{\L}^{-1}.
\end{aligned}\]
\end{proof}%
Motivated by these results, we define the Szeg\H{o} recurrence matrices as follows:

\begin{Definition}[The Szeg\H{o} recurrence matrices]\label{Definition:eta_nu}
the Szeg\H{o}  recurrence matrices are
\begin{align}
R\coloneq\bar{\L}\eta _{[ q]}L^{-1} =&{\U}^{-1}\eta
_{[ p]}\bar U,& \Rscr&\coloneq L\eta _{[ q]}\bar{\L}%
^{-1}=\bar U^{-1}\eta _{[ p]}\U,  \label{defSs} \\
S\coloneq L\nu _{[ q]}\bar{\L}^{-1} =&\bar U^{-1}\nu _{[ p]}%
 \U, & \Sscr&\coloneq\bar{\L}\nu _{[q]}L^{-1}=\U^{-1}\nu _{[ p]}\bar U.  \label{defRs}
\end{align}
\end{Definition}

\begin{pro}[Banded structure of the  Szeg\H{o}  recurrence matrices]\label{bandedmatrixv01eta}
		 The Szeg\H{o} matrices are banded matrices, each possessing $p+q+1$ diagonals, with $q$ superdiagonals and $p$ subdiagonals.
\end{pro}

\begin{proof}
From the definition of the matrices $\nu_{[r]}$ and $\eta_{[r]}$, it's evident that they are banded matrices, each featuring $r$ superdiagonals and $r$ subdiagonals. Then, by examining Definition \ref{Definition:eta_nu} and employing a similar argument as for the banded structure of $T$, we arrive at the asserted conclusion.
\end{proof}

\begin{pro}\label{pro:SR}
The banded matrices $S$ and $\Sscr$ are inverses:
\begin{align*}
	S\Sscr = \Sscr S = I,
\end{align*}
and likewise, $R$ and $\Rscr$ are inverses:
\begin{align*}
	R\Rscr = \Rscr R = I,
\end{align*}
meaning $\Sscr = S^{-1}$ and $\Rscr = R^{-1}$.
\end{pro}
\begin{proof}
	It follows from Definition \ref{Definition:eta_nu}  and the fact that $\nu_{[r]}^2=I$ and $\eta_{[r]}^2=I$.
\end{proof}

With these new matrices in hand, demonstrating the following is straightforward:
\begin{pro}[Factorizations of the recurrence matrices]\label{pro:TSR}
The following factorizations hold true:
\begin{align*}
	\begin{array}{lll}
		T = SR, & & T^{-1} = \Rscr\Sscr, \\ 
		\Tscr^{-1} = \Sscr\Rscr, & & \Tscr = RS.%
	\end{array}
\end{align*}
\end{pro}

\begin{proof}
Taking into account \eqref{prop01}, and based on all the above definitions, we can readily derive the following expressions:
\begin{align*}
	SR &= \left( \bar U^{-1}\nu_{[p]}\mathcal{U}\right) \left( \mathcal{U}^{-1}\eta_{[p]}\bar U\right) = \bar U^{-1}\nu_{[p]}\eta_{[p]}\bar U = \bar U^{-1}\Upsilon_{[p]}\bar U = T, \\
	RS &= \left( \mathcal{U}^{-1}\eta_{[p]}\bar U\right) \left( \bar U^{-1}\nu_{[p]}\mathcal{U}\right) = \mathcal{U}^{-1}\eta_{[p]}\nu_{[p]}\mathcal{U} = \mathcal{U}^{-1}\Upsilon_{[p]}^{-1}\mathcal{U} = \mathcal{R}\mathcal{S} = \left( \bar U^{-1}\eta_{[p]}\mathcal{U}\right) \left( \mathcal{U}^{-1}\nu_{[p]}\bar U\right) \\
	&= \bar U^{-1}\Upsilon_{[p]}^{\top}\bar U = \left( \bar U^{-1}\Upsilon_{[p]}\bar U\right)^{-1} = T^{-1}, \\
	\mathcal{S}\mathcal{R} &= \left( \mathcal{U}^{-1}\nu_{[p]}\bar U\right) \left( \bar U^{-1}\eta_{[p]}\mathcal{U}\right) = \mathcal{U}^{-1}\nu_{[p]}\eta_{[p]}\mathcal{U} = \mathcal{U}^{-1}\Upsilon_{[p]}\mathcal{U} = \Tscr^{-1}.
\end{align*}
\end{proof}

In the context of symmetries, Corollary \ref{cor:LU-sym} implies the following:
\begin{pro}[Symmetry properties of the Szeg\H{o} recurrence matrices]\label{pro:RS-sym}
	The following holds:
	\begin{enumerate}
		\item  If the matrix of measures is real, i.e., $\mu = \bar{\mu}$, then:
		\begin{align*}
			R \bar{R} = S \bar{S} = I,
		\end{align*}
		with $\Rscr = \bar{R}$ and $\Sscr = \bar{S}$. Specifically,
		\begin{align*}
			R &= \bar{L}\eta_{[q]}L^{-1} = U^{-1}\eta_{[p]}\bar{U}, \\
			S &= L\nu_{[q]}\bar{L}^{-1} = \bar{U}^{-1}\nu_{[p]}U.
		\end{align*}
		Moreover, $\Tscr = \bar{T}^{-1}$ and $T = L\Upsilon_{[q]}L^{-1} = \bar{U}^{-1}\Upsilon_{[p]}\bar{U}$.
	\item  For $p=q$, if the matrix of measures is symmetric, i.e., $\mu = \mu^\top$, then the matrices $R$, $S$, $\Rscr$, and $\Sscr$ are symmetric complex matrices. Specifically,
	\begin{align*}
		R &= \bar{U}^\top \eta L^{-1} = L^{-\top}\eta\bar{U}, \\
		S &= L\nu\bar{U}^{-\top} = \bar{U}^{-1}\nu L^\top.
	\end{align*}
	Moreover, $\Tscr = T^\top$, where $T = L\Upsilon L^{-1} = \bar{U}^{-1}\Upsilon\bar{U}$ and $\Tscr = \bar{U}^\top\Upsilon^\top\bar{U}^{-\top} = L^{-\top}\Upsilon^\top L^\top$.
	\item For $p=q$, whenever the matrix of measures is Hermitian, i.e. $\mu=\mu^\dagger$, the matrices $R,S,T,\Rscr,\Sscr$ and $\Tscr$ are  orthogonal complex matrices 
	\begin{align*}
		RR^\top=R^\top R=SS^\top=S^\top S=	TT^\top=T^\top T=&=I,\\
		\Rscr\Rscr^\top=\Rscr^\top \Rscr=\Sscr\Sscr^\top=\Sscr^\top \Sscr=\Tscr\Tscr^\top=\Tscr^\top \Tscr&=I.
	\end{align*}
 Specifically,
\[\begin{aligned}
	R&=\bar \L \eta L^{-1}=\L^{-\dagger}\eta L^\top,&
	S&=L \nu \bar{\L}^{-1}=L^{-\top}\nu \L^\dagger.
\end{aligned}\]
	\item If $p=q$ and the matrix of measures is both real and symmetric, i.e., $\mu = \bar{\mu}$ and $\mu = \mu^\top$, then the matrices $R = \bar{\Rscr}$ and $S = \bar{\Sscr}$ satisfy:
\[	\begin{aligned}
		R &= R^\top, & S &= S^\top, & R\bar{R} = S\bar{S} &= I.
	\end{aligned}\]
	Specifically,
\[	\begin{aligned}
		R &= L\eta L^{-1} = L^{-\top}\eta L^\top, &
		S &= L\nu L^{-1} = L^{-\top}\nu L^\top.
	\end{aligned}\]
	Notice that $R$, $S$, and $T$ are unitary matrices:
	\begin{align*}
		RR^\dagger = R^\dagger R = SS^\dagger = S^\dagger S = TT^\dagger = T^\dagger T = I.
	\end{align*}
	Moreover, $\Tscr = \bar{T}^{-1} = T^\top$ are also unitary matrices.
	\end{enumerate}
\end{pro}

To recover information about the specific entries  of the matrices $R, \Rscr,S$ and $\Sscr^{-1}$ 

\begin{Definition}
\begin{enumerate}
	\item Let us consider $q\times q$ blocks to describe  the lower triangular matrices $L,\bar \L$  and   $p\times p$ blocks to describe  the upper triangular matrices $U,\bar \U$: 
\begin{align*}
	L&=\left[\begin{NiceMatrix}
		L_0 & 0_q&0_q&0_q &\Cdots[shorten-end=-3pt]&\phantom{t}  \\
	0_q& L_{1}& 0_q & 0_q&\Cdots[shorten-end=-3pt]&\phantom{t}\\
		0_q&	0_q& L_{2}& 0_q &\Cdots[shorten-end=-3pt]&\phantom{t}\\[-7pt]
		0_q&	0_q&	0_q& L_{3}& \Ddots[shorten-end=-3pt]&\phantom{t}\\
	\Vdots[shorten-end=-7pt]&\Vdots[shorten-end=-7pt]&\Ddots[shorten-end=-7pt]&\Ddots[shorten-end=0pt]&\Ddots[shorten-end=0pt]&\phantom{t}\\&&&&&\\
	\phantom{t}&\phantom{t}&\phantom{t}&\phantom{t}&\phantom{t}
	\end{NiceMatrix}\right]\left[\begin{NiceMatrix}
	I_q & 0_q&0_q&0_q &\Cdots[shorten-end=-3pt]&\phantom{t}  \\
l_0& 	I_q & 0_q & 0_q&\Cdots[shorten-end=-3pt]&\phantom{t}\\
	*&		l_{1}& 	I_q & 0_q &\Cdots[shorten-end=-3pt]&\phantom{t}\\[-7pt]
	*&	*&		l_{2}& 	I_q & \Ddots[shorten-end=-3pt]&\phantom{t}\\
	\Vdots[shorten-end=-7pt]&\Vdots[shorten-end=-7pt]&\Ddots[shorten-end=-3pt]&\Ddots[shorten-end=0pt]&\Ddots[shorten-end=0pt]&\phantom{t}\\&&&&&\\
	\phantom{t}&\phantom{t}&\phantom{t}&\phantom{t}&\phantom{t}
\end{NiceMatrix}\right],\\
	 \L&=\left[\begin{NiceMatrix}
	\L_0 & 0_q&0_q&0_q &\Cdots[shorten-end=-3pt]&\phantom{t}  \\
	0_q& \L_{1}& 0_q & 0_q&\Cdots[shorten-end=-3pt]&\phantom{t}\\
	0_q&	0_q& \L_{2}& 0_q &\Cdots[shorten-end=-3pt]&\phantom{t}\\[-7pt]
	0_q&	0_q&	0_q& \L_{3}& \Ddots[shorten-end=-3pt]&\phantom{t}\\
	\Vdots[shorten-end=-7pt]&\Vdots[shorten-end=-7pt]&\Ddots[shorten-end=-7pt]&\Ddots[shorten-end=0pt]&\Ddots[shorten-end=0pt]&\phantom{t}\\&&&&&\\
	\phantom{t}&\phantom{t}&\phantom{t}&\phantom{t}&\phantom{t}
\end{NiceMatrix}\right]\left[\begin{NiceMatrix}
	I_q & 0_q&0_q&0_q &\Cdots[shorten-end=-3pt]&\phantom{t}  \\
	\l_{0}& 	I_q & 0_q & 0_q&\Cdots[shorten-end=-3pt]&\phantom{t}\\
	*&		\l_{1}& 	I_q & 0_q &\Cdots[shorten-end=-3pt]&\phantom{t}\\[-7pt]
	*&	*&		\l_{2}& 	I_q & \Ddots[shorten-end=-3pt]&\phantom{t}\\
	\Vdots[shorten-end=-7pt]&\Vdots[shorten-end=-7pt]&\Ddots[shorten-end=0pt]&\Ddots[shorten-end=0pt]&\Ddots[shorten-end=0pt]&\phantom{t}\\&&&&&\\
	\phantom{t}&\phantom{t}&\phantom{t}&\phantom{t}&\phantom{t}
\end{NiceMatrix}\right],\\
	U&=\left[\begin{NiceMatrix}
	I_p & u_{0} &*&* &\Cdots[shorten-end=-3pt]&\phantom{t}  \\
	0_p& 	I_p & u_{1} & *&\Cdots[shorten-end=-3pt]&\phantom{t}\\
	0_p&	0_p	& 	I_p & u_{2}&\Cdots[shorten-end=-3pt]&\phantom{t}\\[-7pt]
	0_p&	0_p&0_p& 	I_p & \Ddots[shorten-end=-3pt]&\phantom{t}\\
	\Vdots[shorten-end=-7pt]&\Vdots[shorten-end=-7pt]&\Ddots[shorten-end=12pt]&\Ddots[shorten-end=7pt]&\Ddots[shorten-end=0pt]&\phantom{t}\\&&&&&\\
	\phantom{t}&\phantom{t}&\phantom{t}&\phantom{t}&\phantom{t}
\end{NiceMatrix}\right]\left[\begin{NiceMatrix}
U_0 & 0_p&0_p&0_p &\Cdots[shorten-end=-3pt]&\phantom{t}  \\
0_p& U_{1}& 0_p & 0_p&\Cdots[shorten-end=-3pt]&\phantom{t}\\
0_p&	0_p& U_{2}& 0_p &\Cdots[shorten-end=-3pt]&\phantom{t}\\[-7pt]
0_p&	0_p&	0_p& U_{3}& \Ddots[shorten-end=-3pt]&\phantom{t}\\
\Vdots[shorten-end=-7pt]&\Vdots[shorten-end=-7pt]&\Ddots[shorten-end=-7pt]&\Ddots[shorten-end=0pt]&\Ddots[shorten-end=0pt]&\phantom{t}\\&&&&&\\
\phantom{t}&\phantom{t}&\phantom{t}&\phantom{t}&\phantom{t}
\end{NiceMatrix}\right],\\
\U&=\left[\begin{NiceMatrix}
	I_p & \u_{0} &*&* &\Cdots[shorten-end=-3pt]&\phantom{t}  \\
	0_p& 	I_p & \u_{1} & *&\Cdots[shorten-end=-3pt]&\phantom{t}\\
	0_p&	0_p	& 	I_p & \u_{2}&\Cdots[shorten-end=-3pt]&\phantom{t}\\[-7pt]
	0_p&	0_p&0_p& 	I_p & \Ddots[shorten-end=-3pt]&\phantom{t}\\
	\Vdots[shorten-end=-7pt]&\Vdots[shorten-end=-7pt]&\Ddots[shorten-end=20pt]&\Ddots[shorten-end=15pt]&\Ddots[shorten-end=0pt]&\phantom{t}\\&&&&&\\
	\phantom{t}&\phantom{t}&\phantom{t}&\phantom{t}&\phantom{t}
\end{NiceMatrix}\right]\left[\begin{NiceMatrix}
\U_0 & 0_p&0_p&0_p &\Cdots[shorten-end=-3pt]&\phantom{t}  \\
0_p& \U_{1}& 0_p & 0_p&\Cdots[shorten-end=-3pt]&\phantom{t}\\
0_p&	0_p& \U_{2}& 0_p &\Cdots[shorten-end=-3pt]&\phantom{t}\\[-7pt]
0_p&	0_p&	0_p& \U_{3}& \Ddots[shorten-end=-3pt]&\phantom{t}\\
\Vdots[shorten-end=-7pt]&\Vdots[shorten-end=-7pt]&\Ddots[shorten-end=-7pt]&\Ddots[shorten-end=0pt]&\Ddots[shorten-end=0pt]&\phantom{t}\\&&&&&\\
\phantom{t}&\phantom{t}&\phantom{t}&\phantom{t}&\phantom{t}
\end{NiceMatrix}\right].
\end{align*}

Here, for $n\in\N_0$,  $L_n,\L_n$ are nonsingular lower triangular $q\times q$ matrices and $U_n,\U_n$ are nonsingular upper triangular $p\times p$ matrices. For $n\in\N_0$, we also have the blocks $l _{n},\l _{n}\in\C^{q\times q}$ and 
 $u _{n},\u_{n}\in\C^{p\times p}$. Finally, when $p=q$ for $n\in\N_0$,  the notation
 \begin{align*}
 \begin{aligned}
 		H_n&\coloneq L_n^{-1}\bar U_n^{-1}, & \mathcal H_n&\coloneq\bar \L_n^{-1}\U_n^{-1}
 \end{aligned}
 \end{align*}
 will be used.

\item 	Let us introduce, for $n\in\N_0$,  the nonsingular matrices $\mathbb L_n\in\C^{2q\times 2q}$ and 
	$\mathbb U_n\in\C^{2p\times 2p}$ given by 
\[\begin{aligned}
	\mathbb L_n&\coloneq \begin{bNiceMatrix}
		-\bar \L_n l_{n}L_n^{-1} & \bar \L_n L_{n+1}^{-1}\\
		\bar \L_{n+1}(I_q-\bar \l_{n}l_{n})L_n^{-1} & \bar \L_{n+1}\bar \l_{n}L_{n+1}^{-1}
	\end{bNiceMatrix}, 	&\mathbb L_n^{-1}&\coloneq \begin{bNiceMatrix}
	-L_n \bar \l_{n}\bar \L_n^{-1} & L_n \bar \L_{n+1}^{-1}\\[2pt]
L_{n+1}(I_q-l_{n}\bar \l_{n})\bar \L_n^{-1} & L_{n+1}l_{n+1,n}\bar \L_{n+1}^{-1}
	\end{bNiceMatrix},  \\
	\mathbb U_n &\coloneq \begin{bNiceMatrix}
			-\U_n^{-1} \u_{n}\bar U_n & \U^{-1}_{n}(I_p-\u_{n}\bar u_{n})\bar U_{n+1}\\
	\U_{n+1} ^{-1}\bar U_n &  \U^{-1}_{n+1}\bar U_{n+1}\bar U_{n+1}
	\end{bNiceMatrix},&
\	\mathbb U_n ^{-1}&\coloneq \begin{bNiceMatrix}
	-\bar U_n^{-1} \bar u_{n}\U_n & \bar U^{-1}_{n}(I_p-\bar u_{n}\u_{n})\U_{n+1}\\
	\bar U_{n+1} ^{-1}\U_n &  \bar U^{-1}_{n+1}\u_{n}\U_{n+1}
\end{bNiceMatrix} .
\end{aligned}\]
and the lower triangular matrices $\mathbb L_{-1}\coloneq \bar \L_0L_0^{-1}\in\C^{q\times q}$ and the upper triangular matrices  $\mathbb U_{-1}\coloneq \U_0^{-1}\bar U_0\in\C^{p\times p}$.
\end{enumerate}
\end{Definition}

\begin{pro}
	For real matrices of measures, i.e. $\mu=\bar \mu$ 
	we have
\[	\begin{aligned}
		\mathbb L_n&= \overline{\mathbb L^{-1}_n}= \begin{bNiceMatrix}
			-\bar L_n l_{n}L_n^{-1} & \bar L_n L_{n+1}^{-1}\\
			\bar L_{n+1}(I_q-\bar l_{n}l_{n})L_n^{-1} & \bar L_{n+1}\bar l_{n}L_{n+1}^{-1}
		\end{bNiceMatrix}, &
		\mathbb U_n &=	\overline{\mathbb U^{-1}_n }= \begin{bNiceMatrix}
			-U_n^{-1} u_{n}\bar U_n & U^{-1}_{n}(I_p-u_{n}\bar u_{n})\bar U_{n+1}\\
			U_{n+1} ^{-1}\bar U_n &  U^{-1}_{n+1}\bar u_{n}\bar U_{n+1}
		\end{bNiceMatrix},
	\end{aligned}\]
	and $\mathbb L_{-1}=\overline{\mathbb L^{-1}_{-1}}= \bar L_0L_0^{-1}$,   $\mathbb U_{-1}= \overline{\mathbb U_{-1}^{-1}}= U_0^{-1}\bar U_0$.
\end{pro}

Then, we can prove that
\begin{pro}\label{pro:stairs}
	Matrices $R,S,\Rscr$ and $\Sscr$ have the following banded structure
	\begin{align*}
R&=\left[\begin{NiceMatrix}
	\mathbb L_{-1} & 0_{q\times 2q}&0_{q\times 2q}&0_{q\times 2q}&\Cdots[shorten-end=-3pt]&\phantom{t}  \\
*& \mathbb L_{1}& 0_{2q} & 0_{2q}&\Cdots[shorten-end=-3pt]&\phantom{t}\\
		*&	*& \mathbb L_{3}& 0_{2q} &\Cdots[shorten-end=-3pt]&\phantom{t}\\[-5pt]
		*&	*&	*& \mathbb L_{5}& \Ddots[shorten-end=-3pt]&\phantom{t}\\
	\Vdots[shorten-end=-7pt]&\Vdots[shorten-end=-7pt]&\Ddots[shorten-end=20pt]&\Ddots[shorten-end=8pt]&\Ddots[shorten-end=0pt]&\phantom{t}\\&&&&&\\
	\phantom{t}&\phantom{t}&\phantom{t}&\phantom{t}&\phantom{t}\\
	\CodeAfter
	\tikz \draw[blue,line width=1.2pt] (2|-1) |- (3|-2) |- (4|-3) |-(5|-4);
\end{NiceMatrix}\right]=\left[\begin{NiceMatrix}
\mathbb U_{-1} & *&*&*&\Cdots[shorten-end=-3pt]&\phantom{t}  \\
0_{2p\times p}& \mathbb U_{1}&  *& *&\Cdots[shorten-end=-3pt]&\phantom{t}\\
0_{2p\times p}&	0_{2p}& \mathbb U_{3}& *&\Cdots[shorten-end=-3pt]&\phantom{t}\\[-5pt]
0_{2p\times p}&	0_{2p}&	0_{2p}& \mathbb U_{5 }& \Ddots[shorten-end=-3pt]&\phantom{t}\\
\Vdots[shorten-end=-9pt]&\Vdots[shorten-end=-7pt]&\Ddots[shorten-end=7pt]&\Ddots[shorten-end=5pt]&\Ddots[shorten-end=0pt]&\phantom{t}\\&&&&&\\
\phantom{t}&\phantom{t}&\phantom{t}&\phantom{t}&\phantom{t}\\
\CodeAfter
\tikz \draw[blue,line width=1.2pt] (2-|1) -| (3-|2) -| (4-|3) -|(5-|4) --(5-|5);
\end{NiceMatrix}\right],\\
S&=\left[\begin{NiceMatrix}
	\mathbb L_{0}^{-1} & 0_{ 2q}&0_{2q}&0_{2q}&\Cdots[shorten-end=-3pt]&\phantom{t}  \\
	*& \mathbb L_{2}^{-1}& 0_{2q} & 0_{2q}&\Cdots[shorten-end=-3pt]&\phantom{t}\\
	*&	*&  \mathbb L_{4}^{-1}& 0_{2q} &\Cdots[shorten-end=-3pt]&\phantom{t}\\[-5pt]
	*&	*&	*&  \mathbb L_{6}^{-1}& \Ddots[shorten-end=-3pt]&\phantom{t}\\
	\Vdots[shorten-end=-7pt]&\Vdots[shorten-end=-7pt]&\Ddots[shorten-end=12pt]&\Ddots[shorten-end=5pt]&\Ddots[shorten-end=0pt]&\phantom{t}\\&&&&&\\
	\phantom{t}&\phantom{t}&\phantom{t}&\phantom{t}&\phantom{t}\\
	\CodeAfter
	\tikz \draw[blue,line width=1.2pt] (2|-1) |- (3|-2) |- (4|-3) |-(5|-4);
\end{NiceMatrix}\right]=\left[\begin{NiceMatrix}
	\mathbb U_0^{-1} & *&*&*&\Cdots[shorten-end=-3pt]&\phantom{t}  \\
	0_{2p}& \mathbb U^{-1}_{2}&  *& *&\Cdots[shorten-end=-3pt]&\phantom{t}\\
	0_{2p}&	0_{2p}& \mathbb U^{-1}_{4}& *&\Cdots[shorten-end=-3pt]&\phantom{t}\\[-5pt]
	0_{2p}&	0_{2p}&	0_{2p}& \mathbb U^{-1}_{6 }& \Ddots[shorten-end=-3pt]&\phantom{t}\\
	\Vdots[shorten-end=-7pt]&\Vdots[shorten-end=-5pt]&\Ddots[shorten-end=7pt]&\Ddots[shorten-end=5pt]&\Ddots[shorten-end=0pt]&\phantom{t}\\&&&&&\\
	\phantom{t}&\phantom{t}&\phantom{t}&\phantom{t}&\phantom{t}\\
	\CodeAfter
	\tikz \draw[blue,line width=1.2pt] (2-|1) -| (3-|2) -| (4-|3) -|(5-|4) --(5-|5);
\end{NiceMatrix}\right],\\
\Rscr&=\left[\begin{NiceMatrix}
	\mathbb L_{-1}^{-1} & 0_{q\times 2q}&0_{q\times 2q}&0_{q\times 2q}&\Cdots[shorten-end=-3pt]&\phantom{t}  \\
	*&\mathbb L_{1}^{-1}& 0_{2q} & 0_{2q}&\Cdots[shorten-end=-3pt]&\phantom{t}\\
	*&	*& \mathbb L_{3}^{-1}& 0_{2q} &\Cdots[shorten-end=-3pt]&\phantom{t}\\[-5pt]
	*&	*&	*& \mathbb L_{5}^{-1}& \Ddots[shorten-end=-3pt]&\phantom{t}\\
	\Vdots[shorten-end=-7pt]&\Vdots[shorten-end=-7pt]&\Ddots[shorten-end=22pt]&\Ddots[shorten-end=8pt]&\Ddots[shorten-end=0pt]&\phantom{t}\\&&&&&\\
	\phantom{t}&\phantom{t}&\phantom{t}&\phantom{t}&\phantom{t}\\
	\CodeAfter
	\tikz \draw[blue,line width=1.2pt] (2|-1) |- (3|-2) |- (4|-3) |-(5|-4);
\end{NiceMatrix}\right]=\left[\begin{NiceMatrix}
	\mathbb U_{-1}^{-1} & *&*&*&\Cdots[shorten-end=-3pt]&\phantom{t}  \\
	0_{2p\times p}& \mathbb U_{1}^{-1}&  *& *&\Cdots[shorten-end=-3pt]&\phantom{t}\\
	0_{2p\times p}&	0_{2p}& \mathbb U_{3}^{-1}& *&\Cdots[shorten-end=-3pt]&\phantom{t}\\[-5pt]
	0_{2p\times p}&	0_{2p}&	0_{2p}& \mathbb U_{5 }^{-1}& \Ddots[shorten-end=-3pt]&\phantom{t}\\
	\Vdots[shorten-end=-7pt]&\Vdots[shorten-end=-5pt]&\Ddots[shorten-end=7pt]&\Ddots[shorten-end=5pt]&\Ddots[shorten-end=0pt]&\phantom{t}\\&&&&&\\
	\phantom{t}&\phantom{t}&\phantom{t}&\phantom{t}&\phantom{t}\\
	\CodeAfter
	\tikz \draw[blue,line width=1.2pt] (2-|1) -| (3-|2) -| (4-|3) -|(5-|4) --(5-|5);
\end{NiceMatrix}\right],\\
\Sscr&=\left[\begin{NiceMatrix}
	\mathbb L_{0} & 0_{2q}&0_{2q}&0_{2q}&\Cdots[shorten-end=-3pt]&\phantom{t}  \\
	*& \mathbb L_{2}& 0_{2q} & 0_{2q}&\Cdots[shorten-end=-3pt]&\phantom{t}\\
	*&	*& \mathbb L_{4}& 0_{2q} &\Cdots[shorten-end=-3pt]&\phantom{t}\\[-5pt]
	*&	*&	*& \mathbb L_{6}& \Ddots[shorten-end=-3pt]&\phantom{t}\\
	\Vdots[shorten-end=-7pt]&\Vdots[shorten-end=-7pt]&\Ddots[shorten-end=12pt]&\Ddots[shorten-end=5pt]&\Ddots[shorten-end=0pt]&\phantom{t}\\&&&&&\\
	\phantom{t}&\phantom{t}&\phantom{t}&\phantom{t}&\phantom{t}\\
	\CodeAfter
	\tikz \draw[blue,line width=1.2pt] (2|-1) |- (3|-2) |- (4|-3) |-(5|-4);
\end{NiceMatrix}\right]=\left[\begin{NiceMatrix}
	\mathbb U_0 & *&*&*&\Cdots[shorten-end=-3pt]&\phantom{t}  \\
	0_{2p}& \mathbb U_{2}&  *& *&\Cdots[shorten-end=-3pt]&\phantom{t}\\
	0_{2p}&	0_{2p}& \mathbb U_{4}& *&\Cdots[shorten-end=-3pt]&\phantom{t}\\[-5pt]
	0_{2p}&	0_{2p}&	0_{2p}& \mathbb U_{6 }& \Ddots[shorten-end=-3pt]&\phantom{t}\\
	\Vdots[shorten-end=-7pt]&\Vdots[shorten-end=-5pt]&\Ddots[shorten-end=7pt]&\Ddots[shorten-end=5pt]&\Ddots[shorten-end=0pt]&\phantom{t}\\&&&&&\\
	\phantom{t}&\phantom{t}&\phantom{t}&\phantom{t}&\phantom{t}\\
	\CodeAfter
	\tikz \draw[blue,line width=1.2pt] (2-|1) -| (3-|2) -| (4-|3) -|(5-|4) --(5-|5);
\end{NiceMatrix}\right].
\end{align*}
\end{pro}

\begin{proof}
A block analysis of the relations within Definition \ref{Definition:eta_nu}, coupled with the banded structure described in Proposition \ref{bandedmatrixv01eta}, confirms the stated results.
\end{proof}

For example, for  $p=3$ and $q=2$  the leading truncated matrices are depicted below
\setcounter{MaxMatrixCols}{30} 
\[\begin{aligned}
	S,\Sscr&=	
	\left[\begin{NiceMatrix}[small]
		\CodeBefore
		\rectanglecolor{blue!15}{4-1}{6-3}	\rectanglecolor{blue!15}{7-4}{9-6}\rectanglecolor{blue!15}{10-7}{12-9}\rectanglecolor{blue!15}{13-10}{15-12}
		\rectanglecolor{red!15}{1-3}{2-4}	\rectanglecolor{red!15}{3-5}{4-6}\rectanglecolor{red!15}{5-7}{6-8}\rectanglecolor{red!15}{7-9}{8-10}
		\rectanglecolor{red!15}{9-11}{10-12}\rectanglecolor{red!15}{11-13}{12-14}
		\Body
		* &*&*&0&0&0&0&0&0&0&0&0&0&0\\
		* &*&*&*&0&0&0&0&0&0&0&0&0&0\\
		* &*&*&*&0&0&0&0&0&0&0&0&0&0\\
		* &*&*&*&0&0&0&0&0&0&0&0&0&0\\
		0&*&*&*&*&*&*&0&0&0&0&0&0&0\\
		0 &0&*&*&*&*&*&*&0&0&0&0&0&0\\
		0&0&0&0&0&0&*&*&0&0&0&0&0&0\\
		0 &0&0&0&0&0&*&*&0&0&0&0&0&0\\
		0 &0&0&0&0&0&*&*&*&*&*&0&0&0\\
		0 &0&0&0&0&0&*&*&*&*&*&*&0&0\\
		0 &0&0&0&0&0&0&*&*&*&*&*&0&0\\
		0 &0&0&0&0&0&0&0&*&*&*&*&0&0\\
		0 &0&0&0&0&0&0&0&0&0&0&0&*&*\\
		0 &0&0&0&0&0&0&0&0&0&0&0&*&*\\
		0 &0&0&0&0&0&0&0&0&0&0&0&*&*\\
	\end{NiceMatrix}\right], &
	R,\Rscr&=	
	\left[\begin{NiceMatrix}[small]
		\CodeBefore
		\rectanglecolor{blue!15}{4-1}{6-3}	
		\rectanglecolor{blue!15}{7-4}{9-6}\rectanglecolor{blue!15}{10-7}{12-9}\rectanglecolor{blue!15}{13-10}{15-12}
		\rectanglecolor{red!15}{1-3}{2-4}
		\rectanglecolor{red!15}{3-5}{4-6}	\rectanglecolor{red!15}{3-5}{4-6}\rectanglecolor{red!15}{5-7}{6-8}\rectanglecolor{red!15}{7-9}{8-10}
		\rectanglecolor{red!15}{9-11}{10-12}\rectanglecolor{red!15}{11-13}{12-14}
		\Body
		* &0&0&0&0&0&0&0&0&0&0&0&0&0\\
		0&*&0&0&0&0&0&0&0&0&0&0&0&0\\
		0 &0&*&*&*&0&0&0&0&0&0&0&0&0\\
		0 &0&0&*&*&*&0&0&0&0&0&0&0&0\\
		0&0&0&*&*&*&0&0&0&0&0&0&0&0\\
		0 &0&0&*&*&*&0&0&0&0&0&0&0&0\\
		0&0&0&*&*&*&*&*&*&0&0&0&0&0\\
		0 &0&0&0&*&*&*&*&*&*&0&0&0&0\\
		0 &0&0&0&0&*&*&*&*&*&0&0&0&0\\
		0 &0&0&0&0&0&0&0&0&*&0&0&0&0\\
		0 &0&0&0&0&0&0&0&0&*&*&*&*&0\\
		0 &0&0&0&0&0&0&0&0&*&*&*&*&*\\
		0 &0&0&0&0&0&0&0&0&*&*&*&*&*\\
		0 &0&0&0&0&0&0&0&0&0&*&*&*&*\\
		0 &0&0&0&0&0&0&0&0&0&0&*&*&*\\
	\end{NiceMatrix}\right].
\end{aligned}
\]

We now analyze the action of the four matrices defined above on the CMV
biorthogonal Laurent polynomials, and we collect the results as the following

\begin{pro}[Szeg\H{o} type recurrence relations]
	The left (correspondingly right) mixed type multiple orthogonal polynomials
	on the step line $B(z)$ and $A(z^{-1})$ (correspondingly $\mathcal{B}(z)$
	and $\mathcal{A}(z^{-1})$) satisfy as well the following recurrence-type
	relations%
\[	\begin{aligned}
		RB(z)& =\bar{\A}(z^{-1}), &
		\mathcal{B}(z) R&=\bar A(z^{-1}), &
		\Rscr\bar{\A}(z)	&=B(z^{-1}), &\bar A(z)\Rscr&=\mathcal{B}(z^{-1}) , \\
		S\bar{\A}(z^{-1})& =z B(z), &
		\bar A(z^{-1})S &=z\mathcal{B}(z),&
		\Sscr B(z^{-1})&=z\bar{\A}(z), &
		\mathcal{B}(z^{-1})\Sscr&=z	\bar A(z).
	\end{aligned}\]
\end{pro}

\begin{proof}
	The action of $R$ on the matrix biorthogonal polynomials $B(z)$ is as follows%
	\begin{align*}
		RB(z)=\bar{\L}\eta _{[ q]}L^{-1}LZ_{[q]}(z)=\bar{\L}\eta
		_{[ q]}Z_{[q]}(z)=\bar{\L}Z_{[q]}(z^{-1})=\bar{\A}(z^{-1}).
	\end{align*}%
	The action of $R$ on the matrix biorthogonal polynomials $\mathcal{B}(z)$ is%
	\begin{align*}
		\mathcal{B}(z)R=Z_{[p]}^{\top }(z)\mathcal{U}\mathcal{U}
		^{-1}\eta _{[ p]}\bar U=Z_{[p]}^{\top }(z)\eta _{[p]}\bar U=Z_{[p]}^{\top }(z^{-1})\bar U=\bar A(z^{-1}).
	\end{align*}%
	Correspondingly, the action of $S$ on $\mathcal{B}(z)$ yields%
	\begin{align*}
		S\bar{\A}(z^{-1})=L\nu _{[ q]}\bar{\L}^{-1}
		\bar {\L}Z_{[q]}(z^{-1})=L\nu_{[q]}Z_{[q]}(z^{-1})=z LZ_{[q]}(z)=z B(z)
	\end{align*}%
	and finally%
	\begin{align*}
		\bar A(z^{-1})S=Z_{[p]}^{\top }(z^{-1})\bar U\bar U^{-1}\nu _{[ p]}\mathcal{U%
		} =Z_{[p]}^{\top }(z^{-1})\nu _{[ p]}\mathcal{U}%
		=zZ_{[p]}^{\top }(z)\mathcal{U}=z\mathcal{B}(z).
	\end{align*}
	
\end{proof}

\begin{pro}[Reduced Szeg\H{o} type recurrence relations]
	\begin{enumerate}
		\item If  the matrix of measures is real, i.e. $\bar{\mu}=\mu$, then there exits  banded matrices, with $q$ superdiagonals and $p$ subdiagonals, $S$ and $S$ with $R\bar R=S\bar S=I$,  such that
		\begin{align*}
			\begin{aligned}
				R B(z)&= \bar B(z^{-1}), &  A(z)R&=\bar A(z^{-1}),&
				S\bar B(z^{-1})&=zB(z), & \bar A(z^{-1})S&=z  A(z).
			\end{aligned}
		\end{align*}
		\item When $p=q$, if the matrix of measures is symmetric, i.e. $\mu^\top=\mu$, then
		\begin{align*}
			\begin{aligned}
				R B(z)&= A^\dagger(z^{-1}), &  B^\top(z)R&=\bar A(z^{-1}),&
				SA^\dagger(z^{-1})&=zB(z), & \bar A(z^{-1})S&= z B^\top(z).
			\end{aligned}
		\end{align*}
		\item  When $p=q$, if the matrix of measures is Hermitian, i.e. $\mu^\dagger=\mu$, then
		\begin{align*}
			\begin{aligned}
				R B(z)&= \bar \A(z^{-1}), &  \A^\dagger(z^{-1})R&=B^\top(z),&
				S\bar \A(z^{-1})&=zB(z), &B^\top(z^{-1})S&=\A^\dagger(z)
			\end{aligned}
		\end{align*}
		\item  When $p=q$, for real symmetric matrix of measures, i.e. $\bar \mu=\mu$ and $\mu^\top=\mu$, there exits  unitary symmetric banded (with $q$ superdiagonals) matrices, $R=R^\top$, $S=S^\top$, $R R^\dagger=SS^\dagger=I$,  such that 
		\begin{align*}
			\begin{aligned}
				R B(z)&= \bar B(z^{-1}), &  S\bar B(z^{-1})&=zB(z).
			\end{aligned}
		\end{align*}
	\end{enumerate}
\end{pro}
\begin{proof}
	It follows from Propositions \ref{pro:sym-poly}, \ref{bandedmatrixv01eta} and \ref{pro:RS-sym}.
\end{proof}

Given a matrix polynomial $P(z)\in \C^{r\times r}$, $\deg P=n$,  its reverse (or reciprocal) polynomial
is  $P^*=z^n \bar P(z^{-1})$. Let us introduce 
\begin{Definition}
We define the 	polynomials matrices $Q_n,\mathcal Q_n\in \C^{q\times q}[z]$ and $P_n,\mathcal P_n\in \C^{p\times p}[z]$ as follows:
\[\begin{aligned}
	Q_{2n}(z)&\coloneq z^n L_{2n}^{-1}B_{2n}^{[q]}(z), & 
	Q^*_{2n+1}(z)&\coloneq z^{n+1} \bar \L_{2n+1}^{-1}\A^{[q]}_{2n+1}(z),\\
	P_{2n}(z)&\coloneq z^n A_{2n}^{[p]}(z)\bar U_{2n}^{-1}, & 
	P^*_{2n+1}(z)&\coloneq z^{n+1}\B^{[q]}_{2n+1}(z) \U_{2n+1}^{-1},\\
\mathcal Q_{2n}(z)&\coloneq z^n \B_{2n}^{[q]}(z)\U_{2n}^{-1}, & 
	\mathcal Q^*_{2n+1}(z)&\coloneq z^{n+1} A^{[q]}_{2n+1}(z) \bar U_{2n+1}^{-1},\\
	\mathcal P_{2n}(z)&\coloneq z^n \bar \L_{2n}^{-1}\A_{2n}^{[p]}(z), & 
	\mathcal P^*_{2n+1}(z)&\coloneq z^{n+1}  L_{2n+1}^{-1}B^{[q]}_{2n+1}(z),
\end{aligned}\]
are called the Szeg\H{o} block polynomials.
\end{Definition}
The polynomials are monic, of the the degree bounded as indicated in the subindex.


\subsection{Szeg\H{o} recurrence for  the square matrix case}
For $p=q$, the band in Proposition \ref{pro:stairs} we are lead to
\begin{pro}
	For $p=q$, the following relations are satisfied
	\begin{align*}
		L_{n+1,n}H_n&=\mathcal H_n \u_{n},\\
		\mathcal H_n^{-1} H_{n+1} &=I-\u_{n} \bar u_{n},\\
		I-\bar \l_{n}l_{n}&=\mathcal H_{n+1} H_n^{-1},\\
		\bar \l_{n}H_{n+1}&=\mathcal H_{n+1} \bar u_{n},\\
			\bar\l_{n}\mathcal H_n&= H_n \bar u_{n},\\
		 H_n^{-1} \mathcal H_{n+1} &=I-\bar u_{n} \u_{n},\\
		I-l_{n}\bar\l_{n}&= H_{n+1} \mathcal H_n^{-1},\\
		L_{n+1,n}\mathcal H_{n+1}&= H_{n+1} \u_{n}.
 	\end{align*}
\end{pro}

\begin{pro}
	Let us assume  $p=q$. Then, 
	\begin{enumerate}
		\item If $\mu=\mu^\top$ we have:
\begin{align*}
	\mathbb L_n&= 	\mathbb U_n^\top=\begin{bNiceMatrix}
		-\bar U^\top_n l_{n}L_n^{-1} & \bar U_n L_{n+1}^{-1}\\
		\bar U^\top_{n+1}(I-\bar u^\top_{n}l_{n})L_n^{-1} & \bar U^\top_{n+1}\bar u^\top_{n}L_{n+1}^{-1}
	\end{bNiceMatrix}, 	
		\\
		\mathbb L_n^{-1}&=\mathbb U_n^{\top}= \begin{bNiceMatrix}
				-L_n \bar u^\top_{n}\bar U_n^{-\top} & L_n \bar U_{n+1}^{-\top}\\[2pt]
				L_{n+1}(I-l_{n}\bar u^\top_{n})\bar U_n^{-\top} & L_{n+1}l_{n}\bar U_{n+1}^{-\top}
			\end{bNiceMatrix}, 
\end{align*}
and  $\mathbb L_{-1}=\mathbb U_{-1}^\top= \bar U^\top_0L_0^{-1}$; and $H_n^\top=\mathcal H_n$.

\item  If  $\mu=\mu^\dagger$, we have: 
\begin{align*}
	\mathbb L_n&= 	\mathbb U_n^{-\top}=\begin{bNiceMatrix}
		-\bar \L_n l_{n}L_n^{-1} & \bar \L_n L_{n+1}^{-1}\\
		\bar \L_{n+1}(I-\bar \l_{n}l_{n})L_n^{-1} & \bar \L_{n+1}\bar \l_{n}L_{n+1}^{-1}
	\end{bNiceMatrix}, 	\\
	\mathbb L_n^{-1}&= 	\mathbb U_n^{-\top}= \begin{bNiceMatrix}
		-L_n \bar \l_{n}\bar \L_n^{-1} & L_n \bar \L_{n+1}^{-1}\\[2pt]
		L_{n+1}(I-l_{n}\bar \l_{n})\bar \L_n^{-1} & L_{n+1}l_{n}\bar \L_{n+1}^{-1}
	\end{bNiceMatrix},  
\end{align*}
and $\mathbb L_{-1}=\mathbb U_{-1}^{-\top} =\bar \L_0L_0^{-1}$; and $H_n=\bar L_n^{-1}L_n^{-\top}$ and $\mathcal H_n=\bar \L_n^{-1}L_n^{-\dagger}$ .

\item If  $\mu$=$\bar\mu$ and $\mu=\mu^\top$, we have:
\begin{align*}
	\mathbb L_n&= \overline{\mathbb L_n^{-1}}	=\mathbb U_n^\top=\overline{\mathbb U_n^{-\top}}=\begin{bNiceMatrix}
		-\bar L_n l_{n}L_n^{-1} & \bar L_n^\top L_{n+1}^{-1}\\
		\bar L_{n+1}(I-\bar l_{n}l_{n})L_n^{-1} & \bar L_{n+1}\bar l_{n}L_{n+1}^{-1}
	\end{bNiceMatrix}, 	
\end{align*}
and $H_n=H_n^\dagger=\bar{ \mathcal H}_n=\bar{\mathcal H}_n^\dagger$.
	\end{enumerate}
\end{pro}
\begin{proof}
	Just recall that for i) $\U=L^\top$ and $\L=U^\top$ , for ii) $U=L^\dag$ and $\U=\L^\dag$ and finally for iii) that $U=L^\top$, $\L=L$ and $\U=L^\top$.
\end{proof}

\begin{pro}[The square matrix case]
	For $p=q$,
	\begin{enumerate}
\item 	For the  real case, $\bar\mu=\mu$, we have  $L=\L$ and $U=\U$ the following relations are satisfied
\begin{align*}
	l_{n}H_n&=\bar H_n u_{n},\\
	\bar H_n^{-1} H_{n+1} &=I-u_{n} \bar u_{n},\\
	I-\bar l_{n}l_{n}&=\bar H_{n+1} H_n^{-1},\\
	l_{n}\bar H_{n+1}&= H_{n+1} u_{n}.
\end{align*}

\item 	For the symmetric case, $\mu=\mu^\top$, we have  $\U=L^\top$ and $\L=U^\top$ the following relations are satisfied
\begin{align*}
	l_{n}H_n&= H^\top_n l^\top_{n+1},\\
	 H_n^{-\top} H_{n+1} &=I-l^\top_{n} \bar u_{n},\\
	\bar u^\top_{n}H_{n+1}&= H^\top_{n+1} \bar u_{n},\\
	\bar u^\top_{n} H^\top_n&= H_n \bar u_{n},\\
	H_n^{-1}  H^\top_{n+1} &=I-\bar u_{n} l^\top_{n},\\
	l_{n} H^\top_{n+1}&= H_{n+1} l^\top_{n}.
\end{align*}

\item For the Hermitian case $\mu=\mu^\dagger$, we have $U=L^\dag$ and $\U=\L^\dag$  the following relations are satisfied
\begin{align*}
	l_{n}H_n&=\mathcal H_n \l^\dagger_{n},\\
	\mathcal H_n^{-1} H_{n+1} &=I-\l^\dagger_{n} l^\top_{n},\\
	I-\bar \l_{n}l_{n}&=\mathcal H_{n+1} H_n^{-1},\\
	\bar \l_{n}H_{n+1}&=\mathcal H_{n+1} l^\top_{n},\\
	\bar\l_{n}\mathcal H_n&= H_n l^\top_{n},\\
	H_n^{-1} \mathcal H_{n+1} &=I-l^\top_{n} \l^\dagger_{n},\\
	I-l_{n}\bar\l_{n}&= H_{n+1} \mathcal H_n^{-1},\\
	l_{n}\mathcal H_{n+1}&= H_{n+1} \l^\dagger_{n}.
\end{align*}

\item For the real symmetric case, $\bar{\mu}=\mu$ and $\mu=\mu^\top$,  we have  $U=L^\top$, $\L=L$ and $\U=L^\top$, and the following relations are satisfied
\begin{align*}
	l_{n}H_n&=\bar H_n l^\top_{n},\\
I-\bar l_{n}l_{n}&=\bar H_{n+1} H_n^{-1},\\
	l_{n}\bar H_{n+1}&= H_{n+1} l^\top_{n}.
\end{align*}
	 
\end{enumerate}
\end{pro}
\subsection{On the scalar case: Szeg\H{o} polynomials and recurrence relations}
	For $p=q=1$ and $\bar\mu=\mu$, the relations
\[	\begin{aligned}
	RB(z)&=\bar B(z^{-1}), &	zB(z)&=S \bar B(z^{-1}),
	\end{aligned}\]
	can be written in terms of 
	\begin{align*}
		\mathbb L_n&= \overline{\mathbb L_n^{-1}}	=\mathbb U_n^\top=\overline{\mathbb U_n^{-\top}}=\begin{bNiceMatrix}
		-\bar L_n l_{n}L_n^{-1} & \bar L_n L_{n+1}^{-1}\\
		\bar L_{n+1}(1- |l_{n}|^2)L_n^{-1} & \bar L_{n+1}\bar l_{n}L_{n+1}^{-1}
	\end{bNiceMatrix}, 	
	\end{align*}
	as 
	\begin{align}\label{eq:S_system1}
		\begin{bNiceMatrix}
			\bar B_{2n-1}(z^{-1})\\
			\bar B_{2n}(z^{-1})
		\end{bNiceMatrix}&=\mathbb L_{2n-1} 	\begin{bNiceMatrix}
		 B_{2n-1}(z)\\
		 B_{2n}(z)
		\end{bNiceMatrix}, \\\label{eq:S_system2}
			z\begin{bNiceMatrix}
		B_{2n}(z)\\
		B_{2n+1}(z))
		\end{bNiceMatrix} &=\bar{\mathbb L}_{2n} 	\begin{bNiceMatrix}
		\bar B_{2n}(z^{-1})\\
		\bar B_{2n+1}(z^{-1})
		\end{bNiceMatrix}.
	\end{align}
	Note that from the first equation of the linear system \eqref{eq:S_system1}
	\begin{align}\label{eq:S1}
\bar B_{2n-1}(z^{-1})=-\bar L_{2n-1} l_{2n}L_{2n-1}^{-1}	 B_{2n-1}(z)+\bar L_{2n-1} L_{2n}^{-1}B_{2n}(z)
	\end{align}
	we get 
		\begin{align*}
		 B_{2n-1}(z)=- L_{2n-1} \bar l_{2n}\bar L_{2n-1}^{-1}	 \bar B_{2n-1}(z^{-1})+ L_{2n-1} \bar L_{2n}^{-1}\bar B_{2n}(z^{-1})
	\end{align*}
	that using \eqref{eq:S1} transforms into
			\begin{align*}
		B_{2n-1}(z)&=- L_{2n-1} \bar l_{2n}\bar L_{2n-1}^{-1}\left(-\bar L_{2n-1} l_{2n}L_{2n-1}^{-1}	 B_{2n-1}(z)+\bar L_{2n-1} L_{2n}^{-1}B_{2n}(z)\right)+ L_{2n-1} \bar L_{2n}^{-1}\bar B_{2n}(z^{-1})\\
		&=| L_{2n}|^2 B_{2n-1}(z)-L_{2n-1} \bar l_{2n} L_{2n}^{-1}B_{2n}(z)+ L_{2n-1} \bar L_{2n}^{-1}\bar B_{2n}(z^{-1})
	\end{align*}
	that can be written  as
			\begin{align*}
 \bar L_{2n}	(1- |l_{2n}|^2 )	L_{2n-1}^{-1}B_{2n-1}(z)+\bar L_{2n} \bar l_{2n} L_{2n}^{-1}B_{2n}(z)
		&= \bar B_{2n}(z^{-1})
	\end{align*}
	which is second equation of the linear system \eqref{eq:S_system1}. Hence, \eqref{eq:S1} imply \eqref{eq:S_system1}. For \eqref{eq:S_system2} we proceed similarly, the equation
	\begin{align}\label{eq:S2}
	z	B_{2n}(z)=- L_{2n} \bar l_{2n+1}\bar L_{2n}^{-1}	 \bar B_{2n}(z^{-1})+ L_{2n} \bar L_{2n+1}^{-1}\bar B_{2n+1}(z^{-1})
	\end{align}
	implies the second equation in \eqref{eq:S_system2}. In terms of the polynomials
	\begin{align*}
	\Phi_{2n}(z)&\coloneq z^n L_{2n}^{-1}B_{2n}(z),\\
	\Phi_{2n+1}(z)&\coloneq z^n \bar L_{2n+1}^{-1}\bar B_{2n+1}(z^{-1}),\\
	\Phi^*_{2n}(z)&=z^n\bar L_{2n}^{-1}\bar B_{2n}(z^{-1}),\\
	\Phi^*_{2n+1}(z)&= z^{n+1} L_{2n+1}^{-1}B_{2n+1}(z),
	\end{align*}
	the system \eqref{eq:S1} and \eqref{eq:S2} can be written
	\begin{align*}
		\Phi_n&=z\Phi_{n-1}+\alpha_{n} \Phi^*_{n-1},\\
		\Phi^*_n&=z\bar \alpha_{n} \Phi_{n-1}+ \Phi^*_{n-1},
	\end{align*}
	where we use the reciprocal or reverse polynomials $P^*(z)\coloneq z^{\deg P}\bar P(z^{-1})$ and the Verblunsky coefficients given by $\alpha_{2n}= l_{2n}$ and $\alpha_{2n+1}=\bar l_{2n+1}$.
	
	These equations are  the Szeg\H{o}  recurrence relations for the for the Szeg\H{o} orthogonal polynomials in the unit circle $\Phi_n(z)$. Hence we refer to  $R$ and $S$ as the Szeg\H{o} matrices and refer to the corresponding  recurrence relations for the $B$'s as Szeg\H{o} recurrence relations. From the previous discussion is clear that the matrices $l_{n}$ plays a role similar to the one played by the Verblunsky coefficients $\alpha_n$ in the standard Szeg\H{o} case.

\subsection{Recurrences for the multiple orthogonal Laurent polynomials}

Let us discuss the case with only two real measures, $q=1$ and $p=2$,  $\mu_1=\bar \mu_1$ and $\mu_2=\bar \mu_2$, then 
\begin{align*}
	\mathbb L_n&= \overline{\mathbb L^{-1}_n}= \begin{bNiceMatrix}
		-\bar L_n l_{n}L_n^{-1} & \bar L_n L_{n+1}^{-1}\\
		\bar L_{n+1}(I_q- |l_{n}|^2)L_n^{-1} & \bar L_{n+1}\bar l_{n}L_{n+1}^{-1}
	\end{bNiceMatrix}\in\C^{2\times 2}, \\
	\mathbb U_n &=	\overline{\mathbb U^{-1}_n }= \begin{bNiceMatrix}
		-U_n^{-1} u_{n}\bar U_n & U^{-1}_{n}(I_p-u_{n}\bar u_{n})\bar U_{n+1}\\
		U_{n+1} ^{-1}\bar U_n &  U^{-1}_{n+1}\bar u_{n}\bar U_{n+1}
	\end{bNiceMatrix}\in\C^{4\times 4},
\end{align*}
were $L_n,l_{n}\in \C$ and $U_n,u_{n}\in\C^{2\times 2}$ with $L_n\neq 0$ and 
\begin{align*}
	U_n&=\begin{bNiceMatrix}
		U_{n,1,1}& U_{n, 1,2}\\
		0& U_{n,2,2}
	\end{bNiceMatrix}, & U_{n,1,1},U_{n,2,2}&\neq 0.
\end{align*}
For the matrix $S$ providing  the Szeg\H{o} type recursion relations $S\bar B(z^{-1})=zB(z)$ and $\bar A(z^{-1})S=zA(z)$ we have a diagonal block structure 
\begin{align*}
	S=\diag(\bar{\mathbb U}_0, \bar{\mathbb U}_2,\bar{\mathbb U}_4,\dots)
\end{align*}
with 
\begin{align*}
	\bar{\mathbb U}_{2n}
	=\begin{bNiceMatrix}
		\bar {\mathbb L}_{2n}&0_2\\
		U^{-1}_{2n+1}\bar U_{2n} & \bar{ \mathbb L}_{2n+2}
	\end{bNiceMatrix}.
\end{align*}
Hence, we have 
\begin{align*}
I_2-U_{2n,2n+1}\bar U_{2n,2n+1}&=0_2\\
	-U_{2n}^{-1} U_{2n,2n+1}\bar U_{2n} &=\bar {\mathbb L}_{2n},\\
	U^{-1}_{2n+1}\bar U_{2n,2n+1}\bar U_{2n+1}&=\bar{\mathbb L}_{2n+2}.
\end{align*}
Thus, we get
\begin{align*}
	U_{2n,2n+1}=-U_{2n}\bar{\mathbb L}_{2n}\bar U_{2n}^{-1}=\bar U_{2n+1}{ \mathbb L}_{2n+2} U_{2n+1}^{-1}.
\end{align*}
If we write 
\[\begin{aligned}
	\mathbb B_{2n}&\coloneq\begin{bNiceMatrix}
		B_{2n}\\
		B_{2n+1}
	\end{bNiceMatrix}, &	\mathbb A_{2n}&\coloneq\begin{bNiceMatrix}
	A^{(1)}_{2n}&A^{(2)}_{2n+1}\\[5pt]
	A^{(2)}_{2n}&A^{(2)}_{2n+1}
	\end{bNiceMatrix}, 
\end{aligned}\]
we can write the following Szeg\H{o} type recurrence
\begin{align*}
	\bar{\mathbb L}_{4n}\bar{\mathbb B}_{4n}(z^{-1})&=z\mathbb B_{4n}(z),\\
	-U_{2n+1}^{-1}\bar U_{2n}\bar{\mathbb B}_{4n}(z^{-1})+	\bar{\mathbb L}_{4n+2}\bar{\mathbb B}_{4n+2}(z^{-1})&=z\mathbb B_{4n+2}(z),\\
	\bar{\mathbb A}_{4n}(z^{-1})	\bar{\mathbb L}_{4n} +\bar{\mathbb A}_{4+2n}(z^{-1})	U^{-1}_{2n+1}\bar U_{2n} &=z{\mathbb A}_{4n}(z),\\
	\bar{\mathbb A}_{4n+2}(z^{-1})\bar{\mathbb L}_{4n+2}&=z{\mathbb A}_{4n+2}(z).
\end{align*}

%
\subsection{Christoffel--Darboux kernels}

\label{S03-CDkernels}



In this section, we aim to present and analyze the Christoffel--Darboux kernels (CD kernels for short) for the CMV biorthogonal Laurent polynomials. We will define them within the framework of mixed-type multiple orthogonality on the unit circle. We also present the fundamental projection and reproducing properties of these mathematical structures, and additionally we provide an Aitken--Berg--Collar (ABC) theorem in the same fashion as it appears in \cite{S-PSPM08}.

Here and subsequently, we will deal with truncated matrices instead of semi-infinite matrices. Thus, for any semi-infinite matrix $M$, the notation $M^{[n]}$ will denote that we get just the $n$ first rows and columns of $M$. Concerning the CMV Laurent matrix polynomials, $B^{[n]}(z)$ means that we get $n$ rows of $B(z)$, so $B^{[n]}(z)$ will be a $n\times q$ matrix with Laurent polynomials at their entries. Respectively, $A^{[n]}(z^{-1})$ means that we get $n$ columns of $A(z^{-1})$, so $A^{[n]}(z^{-1})$ will be a $p \times n$ matrix with Laurent polynomials at their entries.

\begin{align}
\raisebox{28pt}{$T^{[n]}=\quad$} \renewcommand{\arraystretch}{1.0} %
\begin{NiceMatrix} * & &\Cdots^{2q+1}& & * & & & & & & \\ \Vdots^{2p+1}& & &
& &\Ddots & & & & & \\ & & & & & & * & & & & \\ * & & & & & & & * & & & \\
&\Ddots& & & & & & & & & \\ & & & & & &\Vdots_{2q+1}&\Vdots^{2q}& &\Ddots&
\\ & & & * & &\Cdots^{2p+1}& * & * &\Cdots^{2q}& & * \\[0.4mm] & & & & *
&\Cdots_{2p} & * & & & & \\ & & & & &\Ddots &\Vdots_{2p} & & & & \\ & & & &
& & * & & & & \\ \CodeAfter 
\tikz 
	\draw [arrows = {Latex}-{Latex}] (-7,-3) -- node[above=0.4mm] {$n$} (-2.6,-3);
\tikz	\draw [arrows = {Latex}-{Latex}] (0.1,-1) -- node[above=0.4mm, rotate=270] {$n$} (0.1,3);
\tikz	\draw [draw=white] (-1,1/2)  node[above=0.01mm] {$T^{[n,+]}$};
\tikz	\draw [draw=white] (-4.1,-2.5)  node[above=0.01mm] {$ T^{[n,-]}$ };
\SubMatrix\lceil{1-1}{7-7}. \tikz \draw (-2.55,-3) --
(-2.55,3.0); \tikz \draw (8-|1) -- (8-|12); 
\end{NiceMatrix}
\label{bandedmatrixv02}
\end{align}



\vspace*{5pt}
\begin{pro}[Christoffel--Darboux formula]
The following Christoffel--Darboux formula%
\begin{align}
K^{[n]}(x,y)\coloneq A^{[n]}(x^{-1})B^{[n]}(y)=\frac{1}{x^{-1}-y}\left(
A^{[p,n]}(x^{-1})T^{[n,-]
}B^{[n]}(y)-A^{[n]}(x^{-1})T^{[n,+]}B^{[n,q]}(y)\right)
\label{CDKernel(x,y)}
\end{align}%
holds. The $p\times q$ matrix $K^{[n]}(x,y)$ is the truncated Christoffel--Darboux kernel. Entrywise, for $a\in \{1,\ldots ,p\}$, and $b\in \{1,\ldots ,q\}$, the above expression can be read as
\begin{align}
	K_{a,b}^{[n]}(x,y)\coloneq\left( A^{(a)}(x^{-1})\right) ^{[n]}\left(
	B^{(b)}(y)\right) ^{[n]}=\sum_{k=0}^{n-1}A_{k}^{(a)}(x^{-1})B_{k}^{(b)}(y).
	\label{CDKernel(x,y)ab}
\end{align}
\end{pro}



\begin{proof}
Coming back to the matrix $T$ defined in \eqref{Tmatrix}, let us consider
the triangular matrices $T^{[n,+]}$ and $T^{[n,-]}$
showed in \eqref{bandedmatrixv02}. Observe that $T^{[n,+]}$
is a lower triangular $(2q-1)\times (2q-1)$ matrix, and 
$T^{[n,-]}$ is an upper triangular $(2p-1)\times (2p-1)$ matrix.
From \eqref{TBz} we know $TB(z)=zB(z)$, respectively $%
A(z^{-1})T=z^{-1}A(z^{-1})$, so based in these expressions, we consider the
truncations%
\[\begin{aligned}
\left( TB(z)\right) ^{[n]} =&zB^{[n]}(z), \\
\left( A(z^{-1})T\right) ^{[n]} =&z^{-1}A^{[n]}(z^{-1}).
\end{aligned}\]%
Analyzing the action of $T^{[n]}$ in this scenario, we see%
\begin{align*}
\left( TB(z)\right) ^{[n]} =&T^{[n]}B^{[n]}(z)+
T^{[n,-]}B^{[n,q]}(z)=zB^{[n]}(z), \\
\left( A(z^{-1})T\right) ^{[n]}
=&A^{[n]}(z^{-1})T^{[n]}+A^{[p,n]}(z^{-1})
T^{[n,-]}
=z^{-1}A^{[n]}(z^{-1}).
\end{align*}%

Thus, considering the independent complex variables $y,x\in \mathbb{T}$, we
rewrite the above equations as%
\begin{align*}
T^{[n]}B^{[n]}(y) =&yB^{[n]}(y)-T^{[n,+]}B^{[n,q]}(y), \\
A^{[n]}(x^{-1})T^{[n]}
=&x^{-1}A^{[n]}(x^{-1})-A^{[p,n]}(x^{-1})
T^{[n,-]},
\end{align*}%
and multiplying the first one by $A^{[n]}(x^{-1})$ on the l.h.s., and the
second one by $B^{[n]}(y)$ on the r.h.s., we finally deduce%

\begin{align*}
\left( x^{-1}-y\right)
A^{[n]}(x^{-1})B^{[n]}(y)=A^{[p,n]}(x^{-1})
T^{[n,-]}
B^{[n]}(y)-A^{[n]}(x^{-1})T^{[n,+]}B^{[n,q]}(y).
\end{align*}%
Hence, with the notation%
\begin{align}
K^{[n]}(x,y)\coloneq A^{[n]}(x^{-1})B^{[n]}(y),  \label{KernelAB[n]}
\end{align}%
where $K^{[n]}(x,y)$ being a $p\times q$ matrix, the statement follows. 
\end{proof}



Our aim now is to stablish a significant expression concerning the reproducing property of the truncated Christoffel--Darboux kernel $K^{[n]}(x,y)$. In order to achieve this outcome, let us define the following semi-infinite operator as the semi-infinite matrix
\begin{align*}
\pi ^{[n]}\coloneq
\left[\begin{NiceArray}{c|c|cc}
	I_n & 0_n &0_n&\Cdots\\
	\hline
		0_n & 0_n &0_n&\Cdots\\
		\hline
		0_n & 0_n &0_n&\Cdots\\
		\Vdots& \Vdots&\Vdots&
\end{NiceArray}\right]
\end{align*}

such that 
\begin{align}
\left( \pi ^{[ n]}\right) ^{2}=\pi ^{[ n]},\quad \pi ^{[n]}I\pi ^{[ n]}=\pi ^{[ n]},  \label{propPIn}
\end{align}%
with $I$ being the semi-infinite identity matrix. We use this block
projection operator to prove our next



\begin{Theorem}[reproducing property]
For the definition of the Christoffel--Darboux kernel \eqref{CDKernel(x,y)}, the following reproducing property holds
\begin{align}
\oint_{\mathbb{T}}K^{[n]}(x,y)\d\mu (y)K^{[n]}(y,z)=K^{[n]}(x,z),
\label{reproKn}
\end{align}%
where $\d\mu (y)$ is a $q\times p$ matrix with Borel measures in their entries, and all the $K^{[n]}$'s involved are $p\times q$ matrices with Laurent polynomials in their entries.

In terms of the components, for $a,c\in \{1,\ldots ,p\}$, and $b,d\in \{1,\ldots ,q\}$, the above expression can be read as
\begin{align}
\sum_{c=1}^{p}\sum_{b=1}^{q}\oint_{\mathbb{T}}K_{a,b}^{[n]}(x,y)%
\mu _{b,c}(y)K_{c,d}^{[n]}(y,z)dy=K_{a,d}^{[n]}(x,z)  \label{reproKnab}
\end{align}%
Observe that, as they where defined in \eqref{Apol(a)} and \eqref{Bpol(b)} are row (respectively column) vector with polynomial entries.
\end{Theorem}



\begin{proof}
To prove \eqref{reproKn} we simply observe that
\begin{align*}
K^{[n]}(x,y)=A^{[n]}(x^{-1})B^{[n]}(y)=A(x^{-1})\,\pi ^{[ n]}\,B(y).
\end{align*}%
Thus, 
\begin{align*}
\oint_{\mathbb{T}}K^{[n]}(x,y)\d\mu (y)K^{[n]}(y,z)
=&\oint_{\mathbb{T}}A(x^{-1})\,\pi ^{[ n]}\,B(y)\d\mu (y) A(y^{-1})\,\pi ^{[ n]}\,B(z) \\
=&A(x^{-1})\,\pi ^{[ n]}\left( \oint_{\mathbb{T}}B(y)\d\mu
(y)A(y^{-1})\right) \pi ^{[ n]}\,B(z).
\end{align*}%
By biorthogonality  the term in parentheses is just the
semi-infinite identity matrix $I$, and hence by \eqref{propPIn} we get%
\begin{align*}
\oint_{\mathbb{T}}K^{[n]}(x,y)\d\mu (y)K^{[n]}(y,z)=A(x^{-1})\,\pi^{[ n]}\,B(z)=K^{[n]}(x,z),
\end{align*}%
which proves \eqref{reproKn}. We recall that $\d\mu (y)$ is a $q\times p$ matrix whose entries are measures, so therefore $K^{[n]}$ are in turn $p\times q$ matrices.

For the entrywise version \eqref{reproKnab}, we begin by considering \eqref{CDKernel(x,y)ab}. Thus, for $a\in \{1,\ldots ,p\}$, and $b\in \{1,\ldots
,q\}$, we have%
\begin{align*}
K_{a,b}^{[n]}(x,y)=\left( A^{(a)}(x^{-1})\right) ^{[n]}\left(
B^{(b)}(y)\right) ^{[n]}=\sum_{k=0}^{n-1}A_{k}^{(a)}(x^{-1})B_{k}^{(b)}(y).
\end{align*}%
Notice that we also have%
\begin{align*}
K_{a,b}^{[n]}(x,y)=\left( A^{(a)}(x^{-1})\right) ^{[n]}\left(
B^{(b)}(y)\right) ^{[n]}=\left( \left( Z_{[p]}^{(a)}(z^{-1})\right) ^{\top
}\right) ^{[n]}U^{[n]}L^{[n]}\left( Z_{[q]}^{(b)}(y)\right) ^{[n]}.
\end{align*}%
Taking into account \eqref{CDKernel(x,y)ab}, entrywise we have%
\begin{align*}
\sum_{c=1}^{p}\sum_{b=1}^{q}\oint_{\mathbb{T}}K_{a,b}^{[n]}(x,y)\,%
\mu _{b,c}(y)K_{c,d}^{[n]}(y,z)dy=
\end{align*}%
\begin{align}
\sum_{c=1}^{p}\sum_{b=1}^{q}\sum_{k=0}^{n-1}\sum_{\ell
=0}^{n-1}\oint_{\mathbb{T}}A_{k}^{(a)}(x^{-1})B_{k}^{(b)}(y)\,\mu
_{b,c}(y)A_{\ell }^{(c)}(y^{-1})B_{\ell }^{(d)}(z)dy.  \label{sumasSigma}
\end{align}%
Next, by the biorthogonality property, it is a simple
matter to see that%
\begin{align*}
\sum_{c=1}^{p}\sum_{b=1}^{q}\oint_{\mathbb{T}}B_{k}^{(b)}(y)\,\mu
_{b,c}(y)\,A_{\ell }^{(c)}(y^{-1})dy=\delta _{k,\ell }
\end{align*}%
and therefore \eqref{sumasSigma} becomes%
\begin{align*}
\sum_{k=0}^{n-1}\sum_{\ell =0}^{n-1}\delta _{k,\ell
}\,A_{k}^{(a)}(x^{-1})B_{\ell
}^{(d)}(z)=%
\sum_{k=0}^{n-1}A_{k}^{(a)}(x^{-1})B_{k}^{(d)}(z)=K_{a,d}^{[n]}(x,z).
\end{align*}

\end{proof}



Moving forward with the above property at hand, and in like manner as in the scalar case, we find that the Christoffel--Darboux kernel \eqref{CDKernel(x,y)} can be interpreted as the integral kernel of a projection, and we next discuss this particular property and the linear spaces involved. Note that, due to the matrix nature of the situation under study, the CD kernel can carry out the projection in two different ways. To conduct a more comprehensive analysis, let us consider two different types of matrices, whose entries are integrable functions. Let us consider the $r\times q$ matrix $f_{B}^{r\times q}:\mathbb{C}\rightarrow \mathbb{C}^{r\times q}$, (respectively the $p\times r$ matrix $f_{A}^{p\times r}:\mathbb{C}\rightarrow \mathbb{C}^{p\times r}$) whose entries are integrable functions. To keep things simple, and without loss of generality, we can set $r=1$ in the analysis that follows. Thus let us consider the $1\times q$ row vector $f_{B}:\mathbb{C}\rightarrow \mathbb{C}^{q}$, such that $f_{B}(z)=\begin{bmatrix}
f_{B}^{(1)}(z) & \cdots & f_{B}^{(q)}(z)%
\end{bmatrix}%
$ (respectively, the $p\times 1$ column vector $f_{A}:\mathbb{C}\rightarrow 
\mathbb{C}^{p}$, such that $f_{A}(z)=%
\begin{bmatrix}
f_{A}^{(1)}(z) & \cdots & f_{A}^{(p)}(z)%
\end{bmatrix}%
^{\top }$), with integrable functions on the complex plane at their respective entries. Let $\mathcal{F}_{q}=\left\{ f_{B}(z)\text{ such that } f_{B}:\mathbb{C}\rightarrow \mathbb{C}^{q}\right\} $ (respectively $\mathcal{F}_{p}=\left\{ f_{A}(z)\text{ such that }f_{A}:\mathbb{C}\rightarrow \mathbb{C}^{p}\right\} $). Let us define now the linear space of $1\times q$ row vectors, whose entries are Laurent polynomials of complex coefficients $\mathcal{H}_{B}$, which is in turn the linear span $\mathcal{H}_{B}=\,$span$\{B_{k}^{(b^{\prime })}\}_{k=0}^{n-1}$ for $b,b^{\prime }\in \{1,\ldots ,q\}$ (respectively the linear space of $p\times 1$ column vectors, whose entries are Laurent polynomials of complex coefficients $\mathcal{H}_{A}$, which is in turn the linear span $\mathcal{H}_{A}=\,$span$\{A_{k}^{(a^{\prime })}\}_{k=0}^{n-1}$ for $a,a^{\prime }\in \{1,\ldots ,p\}$) of the CMV matrix Laurent polynomial $B(z)$ (respectively $A(z^{-1})$).

Hence, we need to consider two different kinds of projection operators $\pi _{B}^{[n]}:  \mathcal{F}_{q}  \rightarrow  \mathcal{H}_{B} $ and $\pi _{A}^{[n]}:  \mathcal{F}_{p}  \rightarrow  \mathcal{H}_{A} $
given by
\[
\begin{aligned}
 f_{B}(x) & \longmapsto & \left( \pi _{B}^{[n]}f_{B}\right) (y)%
\coloneq\oint_{\mathbb{T}}f_{B}(x)\d\mu, &  f_{A}(y) & \longmapsto & \left( \pi _{A}^{[n]}f_{A}\right) (x)%
	\coloneq\displaystyle\oint_{\mathbb{T}}K^{[n]}(x,y)\d\mu
		(x)f_{A}(y).
\end{aligned}
\]

Thus, the truncated kernel polynomial ${K^{[n]}(x,y)}$ is the integral
kernel of the above projection operators.

To finish this Section, we prove the following noteworthy finding



\begin{Theorem}[ABC Theorem]
For the definition of the Christoffel--Darboux kernel \eqref{CDKernel(x,y)},
the following identity holds%
\begin{align*}
K^{[n]}(x,y)=\left( Z_{[p]}^{\top }(x^{-1})\right) ^{[n]}\left( M%
^{[n]}\right) ^{-1}\left( Z_{[q]}(y)\right) ^{[n]}.
\end{align*}%
Componentwise, with $a\in \{1,\ldots ,p\}$ and $b\in \{1,\ldots ,q\}$, for
the definition of the Christoffel--Darboux kernel \eqref{CDKernel(x,y)ab},
the following alternative identity holds%
\begin{align*}
K_{a,b}^{[n]}(x,y)=\left( \left( Z_{[p]}^{(a)}(x^{-1})\right) ^{[n]}\right)
^{\top }\left( M^{[n]}\right) ^{-1}\left( Z_{[q]}^{(b)}(y)\right)
^{[n]}.
\end{align*}
\end{Theorem}



\begin{proof}
First, from \eqref{LUMMleft} we know that a truncation of size $n$ of the
matrix of moments can be given by%
\begin{align*}
M^{[n]}=\left( L^{[n]}\right) ^{-1}\left( U^{[n]}\right) ^{-1},
\end{align*}%
so therefore%
\begin{align}
\left( M^{[n]}\right) ^{-1}=U^{[n]}L^{[n]}.  \label{MUL-1}
\end{align}%
Next, replacing the definitions \eqref{Bpol} and \eqref{Apol}, namely 
\begin{align*}
B(z)=LZ_{[q]}(z)\text{, and }A(z^{-1})=Z_{[p]}^{\top }(z^{-1})U.
\end{align*}%
Thus, from \eqref{KernelAB[n]} we deduce%
\begin{align*}
K^{[n]}(x,y) =&A^{[n]}(x^{-1})B^{[n]}(y)=\left( Z_{[p]}^{\top
}(x^{-1})U\right) ^{[n]}\left( LZ_{[q]}(y)\right) ^{[n]} \\
=&\left( Z_{[p]}^{\top }(z^{-1})\right) ^{[n]}U^{[n]}L^{[n]}\left(
Z_{[q]}(y)\right) ^{[n]}.
\end{align*}%
Replacing \eqref{MUL-1} in the above expression, the statement follows.

Concerning the componentwise expression, combining formulas \eqref{Bpol(b)} and \eqref{Apol(a)} with \eqref{CDKernel(x,y)ab}, and making the above
computations in the same manner, the expression follows.
\end{proof}

\section{Christoffel perturbations}

\label{S04-ChristPerturb}

In this section, we will examine the theory of Christoffel transformations of the $q\times p$ matrix of measures $\d\mu (z)$ defined in \eqref{MatrixMeasures}. It means that we will discuss the effect of the multiplication of $\d\mu (z)$ by certain appropriate $r\times r$, $r\in \mathbb{N}$ diagonal matrices $W_{[r]}(z)$ with Laurent polynomials at the entries of the main diagonal, which we will define in detail a little later. Since we are working with matrices, let us observe that these transformations can come from the left hand side (LHS for short in the sequel) multiplication of $\d\mu (z)$, namely $W_{[q]}(z)\d\mu (z)$, and also in a dual manner they can come from the right hand side (RHS for short in the sequel) multiplication of $\d\mu (z)$, namely $\d\mu (z)W_{[p]}(z) $. We will begin with a thorough development of the LHS Christoffel perturbations of $\d\mu (z)$, and towards the last part of the section, as the results are expected to be similar, we will provide a less detailed presentation of the results for the RHS Christoffel perturbations of $\d\mu (z)$.

First of all, we will start by introducing some structures that will be
useful to us in the development of the main ideas. Thus, let us define the
following matrices based on the identity matrix. For $a,r\in \mathbb{N}$, $%
a\in \{1,\ldots ,r\}$, $\nu =0,1,2,\ldots $, the matrix $I_{[r]}^{(a)}$ will
be a semi-infinite diagonal matrix with ones just at the positions $(r\nu
+a,r\nu +a)$ in the main diagonal, and zeros everywhere else. For a given $r$%
, these matrices satisfy%
\begin{align}
I_{[r]}^{(a)}I_{[r]}^{(b)}=\delta _{a,b}I_{[r]}^{(a)},  \label{IIdI}
\end{align}%
for $\delta _{a,b}$ being the Kronecker delta, and also we find%
\begin{align*}
\sum_{a=1}^{r}I_{[r]}^{(a)}=I
\end{align*}%
for $I$ being the semi-infinite identity matrix. For example, for $r=2$ we
have%
\[\begin{aligned}
I_{[2]}^{(1)}&\coloneq\diag(1,0|1,0|1,0,\dots),%
& I_{[2]}^{(2)}&\coloneq\diag(0,1|0,1|0,1|\dots)
.
\end{aligned}\]%
From \eqref{IIdI} we know that 
\begin{align}
I_{[r]}^{(a)}I_{[r]}^{(a)}=\left( I_{[r]}^{(a)}\right) ^{2}=I_{[r]}^{(a)},
\label{powersI}
\end{align}%
and also observe that%
\begin{align}
I_{[r]}^{(a)}Z_{[r]}(z)=Z_{[r]}^{(a)}.  \label{IZaZa}
\end{align}%
In this framework, from \eqref{prop02}, we can define the following $r$ projector matrices $\Upsilon_{[r]}^{(a)}$, $a\in\{1,\ldots ,r\}$, such that%
\begin{align}
\Upsilon_{[r]}^{(a)}=\Upsilon _{[r]}I_{[r]}^{(a)}=I_{[r]}^{(a)}\Upsilon_{[r]},  \label{UpsilonID}
\end{align}%
satisfying%
\begin{align}
\sum_{a=1}^{r}\Upsilon_{[r]}^{(a)}=\Upsilon_{[r]}\sum_{a=1}^{r}I_{[r]}^{(a)}=\Upsilon_{[r]}I_{[r]}=\Upsilon_{[r]}.  \label{sumUpsilona}
\end{align}%
The action of these projector matrices $\Upsilon_{[r]}^{(a)}$ on the
monomial matrices \eqref{ZnmonCMV} is such that $\Upsilon _{[
r]}^{(a)} $ multiplying $Z_{[r]}(z)$ yields $z$ by a monomial matrix $%
Z_{[r]}^{(a)}(z)$ of size $\infty \times r$, which are all zeros except the $%
a$-th colum, satisfying%
\begin{align*}
\sum_{a=1}^{r}Z_{[r]}^{(a)}(z)=Z_{[r]}(z).
\end{align*}%
Moreover, the action of these matrices on a semi-infinite column vector of the monomial matrix $Z_{[r]}(z)$ can be seen as follows
\begin{align}
\Upsilon_{[r]}^{(a)}Z_{[r]}^{(b)}(z)=z\delta _{a,b}Z_{[r]}^{(b)}(z).
\label{prop02ab}
\end{align}%
Notice that these columns $Z_{[r]}^{(a)}$ are previously defined in \eqref{ZnmonCMVj}, and note also the following subtle detail: with this notation,
and taking \eqref{powersI} into account, we find that%
\begin{align}
\Upsilon_{[r]}Z_{[r]}^{(a)}(z)=\Upsilon _{[r]}I_{[r]}^{(a)}Z_{[r]}(z)=\Upsilon_{[r]}\left( I_{[r]}^{(a)}\right)
^{2}Z_{[r]}(z)=\Upsilon_{[r]}^{(a)}Z_{[r]}^{(a)}(z).
\label{subtleDetail}
\end{align}%
Although these semi-infinite projection matrices $\Upsilon _{[r]}^{(a)}$ have a countable infinity number of columns identically zero, it
is a simple matter to check that if we restrict their action just to their own subspace, the transpose of $\Upsilon _{[r]}^{(a)}$ is also its inverse, that is, $\left( \Upsilon_{[r]}^{(a)}\right) ^{\top}=\left( \Upsilon_{[r]}^{(a)}\right) ^{-1}$.

For reasons that will become clear a little further ahead, here we are going
to analyze perturbations of a very specific form. Similar perturbations were
described for the first time in \cite{amt}, and there they were called
\textquotedblleft \textit{prepared Laurent polynomials}\textquotedblright .
To describe the particular kind of modifications in which we are interested
here, we first define a \textit{balanced Laurent polynomial} as follows%
\begin{align}
W(z)=c(z-z_{1})(z-z_{2})\cdots (z-z_{2d})z^{-d}=\sum_{i=-d}^{d}\omega
_{i}z^{i},  \label{polyPerturbadorW}
\end{align}%
for which $\deg^{\pm }W(z)= d$, $d\in 
\mathbb{N}$, with $W(z)$ having exactly $2d$ different complex zeros. Let us
define now a $r\times r$ diagonal matrix%
\begin{align}
W_{[r]}(z)\coloneq\diag\left( W_{1}(z),\ldots ,W_{r}(z)\right) ,
\label{mbLp[r]}
\end{align}%
with all the entries in its main diagonal being well-poised Laurent polynomials of the same degrees $\pm d$, and different zeros. By extension, we will refer to $W_{[r]}(z)$ as \textit{matrix balanced Laurent polynomial}. Under this framework, $W_{[r]}(z)$ has a total of $2dr$\ different and simple zeros, and given a matrix of measures \eqref{MatrixMeasures}, in what follows we will consider the new perturbed matrix of measures%
\begin{align}
\d\hat{\mu}(z)\coloneq W_{[q]}(z)\d\mu (z),  \label{ChristPertLHS}
\end{align}%
where $W_{[q]}(z)$ is a matrix prepared Laurent polynomial. The
corresponding Christoffel $W$-modified CMV left ($\mathcal{\hat{M}}$) (resp.
right ($\tilde{\mathcal{\hat{M}}}$)) moment matrices are%
\begin{align}
\mathcal{\hat{M}}=\oint_{\mathbb{T}}Z_{[q]}(z)\d\hat{\mu}%
(z)\,Z_{[p]}^{\top }(z^{-1})=\oint_{\mathbb{T}}Z_{[q]}(z)%
\,W_{[q]}(z)\d\mu (z)\,Z_{[p]}^{\top }(z^{-1})  \label{Mhat}
\end{align}%
and%
\begin{align}
\tilde{\mathcal{\hat{M}}}=\oint_{\mathbb{T}}Z_{[q]}(z^{-1})\d\hat{\mu}(z) Z_{[p]}^{\top }(z)=\oint_{\mathbb{T}}Z_{[q]}(z^{-1})\,W_{[q]}(z)\d\mu (z) Z_{[p]}^{\top }(z).
\label{MhatWave}
\end{align}

Next, with the notation%
\begin{align}
\mathcal{W}=\sum_{b=1}^{q}W_{b}\left( \Upsilon_{[q]}^{(b)}\right) ,
\label{Wsuma}
\end{align}%
we prove the following result concerning the connection between the Gram
matrices $M$ and $\mathcal{\hat{M}}$.



\begin{pro}
The Christoffel $W$ perturbed  CMV moment matrix $\mathcal{\hat{M}}$ defined
in \eqref{Mhat}, can be rewritten as%
\begin{align*}
\mathcal{\hat{M}}=\oint_{\mathbb{T}}Z_{[q]}(z)\,W_{[q]}(z)\d\mu (z) Z_{[p]}^{\top }(z^{-1})=\mathcal{W}\oint_{\mathbb{T}}Z_{[q]}(z)\d\mu (z) Z_{[p]}^{\top }(z^{-1})=\mathcal{WM}.
\end{align*}
\end{pro}



\begin{proof}
From, \eqref{prop01}-\eqref{prop03} we have%
\begin{align}
\left( \Upsilon_{[r]}\right)
^{k}Z_{[r]}(z)=z^{k}\,Z_{[r]}(z)=Z_{[r]}(z)\,z^{k},\quad k\in \mathbb{Z}.
\label{powersUpsilon}
\end{align}%
This in turn implies that, for any Laurent polynomial $p(z)$ we have%
\begin{align*}
p\left( \Upsilon_{[r]}\right) Z_{[r]}(z)=Z_{[r]}(z)\,p(z).
\end{align*}%
In particular, for any given matrix prepared Laurent polynomial $W_{[q]}(z)$
we get%
\begin{align}
W_{[q]}\left( \Upsilon_{[q]}\right)
Z_{[q]}(z)=Z_{[q]}(z)W_{[q]}\left( z\right) .  \label{dosLados}
\end{align}%
The right hand side above can be written as%
\begin{align*}
Z_{[q]}(z)W_{[q]}(z) =&%
\begin{bmatrix}
Z_{[q]}^{(1)} & Z_{[q]}^{(2)} & \cdots & Z_{[q]}^{(q)}%
\end{bmatrix}%
\begin{bmatrix}
W_{1}(z) &  &  \\ 
& \ddots &  \\ 
&  & W_{q}(z)%
\end{bmatrix}
\\
=&%
\begin{bmatrix}
Z_{[q]}^{(1)}W_{1}(z) & Z_{[q]}^{(2)}W_{2}(z) & \cdots & 
Z_{[q]}^{(q)}W_{q}(z)%
\end{bmatrix}%
.
\end{align*}%
Every entry in the above vector is, in fact, a semi-infinite column vector with Laurent polynomials in their entries, and taking \eqref{IZaZa} into
account, they can be rewritten as
\begin{align*}
Z_{[q]}^{(a)}W_{a}(z)=I_{[q]}^{(a)}Z_{[q]}(z)W_{a}(z),
\end{align*}%
which yields%
\begin{align}
Z_{[q]}(z)W_{[q]}(z)=%
\begin{bmatrix}
I_{[q]}^{(1)}W_{1}(z) & I_{[q]}^{(2)}W_{2}(z) & \cdots & 
I_{[q]}^{(q)}W_{q}(z)%
\end{bmatrix}%
Z_{[q]}(z).  \label{ladoD}
\end{align}%
On the other hand,%
\begin{align*}
W_{[q]}\left( \Upsilon_{[q]}\right) Z_{[q]}(z) =&%
\begin{bmatrix}
W_{1}(\Upsilon_{[q]}) &  &  \\ 
& \ddots &  \\ 
&  & W_{q}(\Upsilon_{[q]})%
\end{bmatrix}%
\begin{bmatrix}
Z_{[q]}^{(1)} & Z_{[q]}^{(2)} & \cdots & Z_{[q]}^{(q)}%
\end{bmatrix}
\\
=&%
\begin{bmatrix}
W_{1}(\Upsilon_{[q]})Z_{[q]}^{(1)} & W_{2}(\Upsilon_{[q]})Z_{[q]}^{(2)} & \cdots & W_{q}(\Upsilon_{[q]})Z_{[q]}^{(q)}%
\end{bmatrix}%
\end{align*}%
Now, every entry in the above vector is, in fact, a semi-infinite column vector with Laurent polynomials in their entries, and taking \eqref{IZaZa} into account, they can be rewritten as
\begin{align*}
W_{a}(\Upsilon_{[q]})Z_{[q]}^{(a)}=W_{a}(\Upsilon _{[q]})I_{[q]}^{(a)}Z_{[q]}(z)=W_{a}(\Upsilon_{[q]}^{(a)})Z_{[q]}(z),
\end{align*}%
which leads to%
\begin{align}
W_{[q]}\left( \Upsilon_{[q]}\right) Z_{[q]}(z)=%
\begin{bmatrix}
W_{1}(\Upsilon_{[q]}^{(1)}) & W_{2}(\Upsilon_{[q]}^{(2)}) & 
\cdots & W_{q}(\Upsilon_{[q]}^{(q)})%
\end{bmatrix}%
Z_{[q]}(z).  \label{ladoI}
\end{align}%
Hence, from \eqref{dosLados}, \eqref{IIdI}, \eqref{prop02ab} and (\eqref{UpsilonID} we finally obtain
\begin{align*}
Z_{[q]}(z)W_{[q]}(z) =&%
\begin{bmatrix}
W_{1}(\Upsilon_{[q]}^{(1)})Z_{[q]}^{(1)} & W_{2}(\Upsilon _{[
q]}^{(2)})Z_{[q]}^{(2)} & \cdots & W_{2}(\Upsilon _{[
q]}^{(q)})Z_{[q]}^{(q)}%
\end{bmatrix}
\\
=&\left( W_{1}\left( \Upsilon_{[q]}^{(1)}\right) +W_{2}\left(
\Upsilon_{[q]}^{(2)}\right) +\cdots +W_{q}\left( \Upsilon _{[
q]}^{(q)}\right) \right) Z_{[q]}(z) \\
=&\left( \sum_{b=1}^{q}W_{b}\left( \Upsilon_{[q]}^{(b)}\right)
\right) Z_{[q]}(z).
\end{align*}%

\end{proof}


From now on, by analogy, we will denote by%
\begin{align*}
\hat{B}(z)\coloneq\hat{L}Z_{[q]}(z)\text{, \ and \ \ }\hat{A}(z^{-1})%
\coloneq Z_{[p]}^{\top }(z^{-1})\hat{U}
\end{align*}%
the CMV matrix Laurent polynomials, biorthogonal with respect to the
Christoffel $W$-modified matrix of measures $\d\hat{\mu}$.


\subsection{Connector matrix and connection formulas}



Next we consider in turn the Gauss factorization of the moment matrix $%
\mathcal{\hat{M}}$, which is%
\begin{align*}
\mathcal{\hat{M}}=\hat{L}^{-1}\hat{U}^{-1}=\mathcal{W\,}L^{-1}U^{-1}.
\end{align*}%
Multiplying the last equation by $U$ on the r.h.s., and by $\hat{L}$ on the
l.h.s., we call the resulting semi-infinite upper triangular connection
matrix%
\begin{align}
N_{C}\coloneq\hat{U}^{-1}U=\hat{L}\mathcal{W}L^{-1}.
\label{defmatrixN}
\end{align}

The matrix $N_{C}$\ is know as a \textit{connector}, or \textit{%
connection matrix} (see \cite[Sec. 6.3]{M-EIBPOA21}), because it links the
original with the modified sequences of orthogonal polynomials. We next
prove the following important result concerning the connection properties of $N_{C}$



\begin{pro}[Connection formulas]
The following connection formulas, between CMV matrix Laurent polynomials
biorthogonals with respect to $\d\mu (z)$ and $\d\hat{\mu}(z)$, hold:%
\begin{align}
\begin{array}{lllll}
\mathcal{W}\,\hat{B}(z)=\mathcal{N\,}B(z), &  &  &  & \hat{B}^{(b)}(z)={%
\displaystyle\frac{1}{W_{b}(z)}}\,\mathcal{N\,}B^{(b)}(z), \\ 
&  &  &  &  \\ 
\hat{A}(z^{-1})\,N_{C}=A(z^{-1}), &  &  &  & \hat{A}^{(a)}(z^{-1})\,%
N_{C}=A^{(a)}(z^{-1}).%
\end{array}
\label{relatAB}
\end{align}
\end{pro}

\begin{proof}
From \eqref{ZnmonCMVj} and \eqref{Bpol} we can therefore study the action of the matrix $N_{C}$ on $B(z)$, namely
\begin{align*}
\mathcal{N\,}B(z)=\hat{L}\mathcal{W}L^{-1}LZ_{[q]}(z)=\hat{L}\mathcal{W}%
Z_{[q]}(z)=\mathcal{W}\,\hat{B}(z),
\end{align*}%
and therefore, on the columns of $B(z)$, we have%
\begin{align}
\mathcal{N\,}B^{(b)}=\hat{L}\mathcal{W}Z_{[q]}^{(b)}=W_{b}(z)\hat{B}^{(b)},
\label{conformB0}
\end{align}%
which provides the following connection formula for the polynomials $\hat{B}%
^{(b)}(z)$ and $B^{(b)}(z)$%
\begin{align}
\hat{B}^{(b)}(z)=\frac{1}{W_{b}(z)}\mathcal{N\,}B^{(b)}(z).
\label{connformB}
\end{align}%
Proceeding in a similar fashion, for $a\in \{1,\ldots ,p\}$ and taking into account \eqref{Apol} we find
\begin{align*}
A(z^{-1})=Z_{[p]}^{\top }(z^{-1})U=%
\begin{bmatrix}
A^{(1)}(z^{-1}) \\ 
A^{(2)}(z^{-1}) \\ 
\vdots \\ 
A^{(p)}(z^{-1})%
\end{bmatrix}%
=%
\begin{bmatrix}
\left( Z_{[p]}^{(1)}(z^{-1})\right) ^{\top } \\ 
\left( Z_{[p]}^{(2)}(z^{-1})\right) ^{\top } \\ 
\vdots \\ 
\left( Z_{[p]}^{(p)}(z^{-1})\right) ^{\top }%
\end{bmatrix}%
U
\end{align*}%
and therefore, the inverse of the semi-infinite upper triangular matrix $N_{C}$, acting on the r.h.s. $A(z^{-1})$, yields%
\begin{align*}
A(z^{-1})\,N_{C}^{-1}=Z_{[p]}^{\top }(z^{-1})UU^{-1}\hat{U}%
=Z_{[p]}^{\top }(z^{-1})\hat{U}=\hat{A}(z^{-1}),
\end{align*}%
which provides the following connection formula for polynomials $\hat{A}%
^{(a)}(z^{-1})$ and $A^{(a)}(z^{-1})$%
\begin{align}
\hat{A}^{(a)}(z^{-1})N_{C}=A^{(a)}(z^{-1}).  \label{connformA}
\end{align}%

\end{proof}



From \eqref{defmatrixN}, it is clear that the structure of the connection matrix $N_{C}$ depends on the polynomials at the entries of the perturbation matrix $W_{[q]}(z)$. Indeed, we observe that evaluating \eqref{conformB0} at the zeros of every prepared Laurent polynomial at the main diagonal of $W_{b}(z)$ (see \eqref{polyPerturbadorW}, one gets
\begin{align*}
\mathcal{N\,}B^{(b)}(z_{b,\ell })=0,\text{ for every }\ell \in \{1,\ldots
,2d\}.
\end{align*}

Moreover, denoting by $N_{C,i,j}$ the elements of $N_{C}$, and
taking into account that $N_{C}$ is an upper triangular matrix (i.e., $%
N_{C,i,j}=0$ if $i>j$), thus for $i=0,1,2,\ldots $ we can assert that%
\begin{align}
\sum_{j=i}^{i+2dq}N_{C,i,j}\,B_{j}^{(b)}(z_{b,\ell })=0,\quad
b=1,\ldots ,q,\text{ and }\ell \in \{1,\ldots ,2d\}.  \label{sumaNB}
\end{align}%
We are interested in solving the above linear system of equations for the
entries $N_{C,i,j}$, $j\geq i$, and in order to do so, we restrict
ourselves to the situation where all the zeros $z_{b,\ell }$ are simple, and
different for every $b=1,\ldots ,q$.

Next we discuss the reason why we need perturbations as prepared Laurent
polynomials of the same degree in $W_{[q]}(z)$, and this relies on the
particular structure of the matrix $\Upsilon_{[q]}$ and its powers.
From \eqref{Wsuma} and \eqref{defmatrixN} we know%
\begin{align*}
N_{C}=\hat{U}^{-1}U=\hat{L}\left( \sum_{b=1}^{q}W_{b}\left( \Upsilon
_{[ q]}^{(b)}\right) \right) L^{-1}.
\end{align*}%
On one side $\hat{U}^{-1}U$ means that $N_{C}$ must be an
upper triangular matrix, and on the other side one encounters that $\hat{L}%
\mathcal{W}L^{-1}$ must be a lower Hessenberg matrix depending on
the particular structure of the matrix $\sum_{b=1}^{q}W_{b}\left( \Upsilon
_{[ q]}^{(b)}\right) $, that in turn depends on the degrees $\pm d$,
and the number $q$ of the prepared Laurent polynomials at the diagonal of $%
W_{[q]}\left( \Upsilon_{[q]}\right) $. From \eqref{powersUpsilon},
the matrix $W_{[q]}\left( \Upsilon_{[q]}\right) $ will be a $\left(
4dq+1\right) $-diagonal matrix, which in turn implies that its highest
nonzero upper-diagonal will be at position $(2dq+1)$. Putting together the
two bolded requirements mentioned above, it follows that $N_{C}$ will
be a semi-infinite upper $(2dq+1)$-banded matrix, so in order to determine
its entries we need $2dq$ different zeros for the prepared Laurent
polynomials in $W_{[q]}\left( \Upsilon_{[q]}\right) $. It means that
all the zeros in the diagonal entries of the perturbation $W_{[q]}$ must be
a total of $2dq$ different zeros. The shape of $N_{C}$ is shown in \eqref{bandedmatrixv03}.



\begin{align}
N_{C}=\, \renewcommand{\arraystretch}{1.0} 
\begin{bNiceMatrix} * &
\Cdots_{2dq+1}& & * & & & \\ & \Ddots & & &\Ddots & & \\ & & \Ddots &
&\phantom{*}&\Ddots & \\ & & &\Ddots &\phantom{*}&\Ddots &\phantom{*}\\ & &
& &\Ddots &\Ddots &\phantom{*}\\ & & & & &\Ddots &\phantom{*}\\ & & & & &
&\phantom{*}\\ 
\end{bNiceMatrix}
\label{bandedmatrixv03}
\end{align}



\subsubsection{Connection formula between the CD kernels}



Next, we are going to connect the above kernel \eqref{CDKernel(x,y)} with
the corresponding kernel of the modified matrix measure $\d\hat{\mu}(z)$,
that is, we find a connection formula between $K^{[n]}(x,y)$ and $\hat{K}%
^{[n]}(x,y)$.



From \eqref{relatAB}, we define the matrix $N_{C}^{[n,2dq ]}$, showed in \eqref{bandedmatrixv04}, where
the truncation of the product $N_{C}B(z)$ is also shown graphically.
The matrix $N_{C}^{[n,2dq ]}$ satisfies
\begin{align}
\left( N_{C}B(z)\right) ^{[n]}=N_{C}^{[n]}B^{[n]}(z)+\left( 
N_{C}^{[n,2dq ]}B^{[n,2dq]}(z)\right)
_{n\times q}=W(z)\,\hat{B}^{[n]}(z).  \label{truncaNBn}
\end{align}


\enlargethispage{1cm}
\begin{align}
\raisebox{32pt}{$
\renewcommand{\arraystretch}{1.0}
\begin{bNiceMatrix}
\ast          &              &\Cdots_{2dq+1}&     &\ast          &       &              &               &      &     &     \\ 
              &\Ddots        &              &     &              &\Ddots &              &               &      &     &     \\
              &              &              &     &              &       & \ast         &               &      &     &     \\ 
              &              &              &     &              &       &              &\blacktriangle &      &     &     \\
              &              &              &     &              &       &              &               &      &     &     \\
							&              &              &     &              &       &\Vdots_{2dq+1}&\Vdots^{2dq}   &      &\Ddots& \\
							&              &              &     &              &       & \ast &\blacktriangle &\Cdots_{2dq+1}&     &\blacktriangle \\
\CodeAfter \tikz \draw (1-|8) -- (8-|8);
\end{bNiceMatrix}
$} \renewcommand{\arraystretch}{1.0} \begin{bNiceMatrix} \ast &\Cdots &\ast
\\ \Vdots & &\Vdots \\ & & \\ & & \\ & & \\ & & \\ & & \\ \ast &\Cdots &\ast
\\ \blacktriangle &\Cdots &\blacktriangle \\ \Vdots^{2dq} & &\Vdots \\ & &
\\ \blacktriangle &\Cdots^{q} &\blacktriangle \\ \CodeAfter \tikz \draw
[shorten > = 0.5em, shorten < = 0.5em](9-|1) -- (9-|last) ;
\end{bNiceMatrix} 
\raisebox{28pt}{$
\mbox{\Large $ = $}
\renewcommand{\arraystretch}{1.0}
\begin{bNiceMatrix}
\ast       &\Cdots     &\ast       \\
\Vdots     &           &\Vdots     \\
           &           &           \\
           &           &           \\
           &           &           \\
					 &           &           \\
           &           &           \\
\ast       &\Cdots^{q} &\ast       \\
\end{bNiceMatrix}
\mbox{\Large $+$}
\renewcommand{\arraystretch}{1.0}
\begin{bNiceMatrix}
              &            &              \\
              &            &              \\
              &            &              \\
					    &            &              \\
\blacktriangle  &\Cdots      &\blacktriangle  \\
\Vdots^{2dq}  &            &\Vdots       \\
              &            &              \\
\blacktriangle  &\Cdots^{q}  &\blacktriangle  \\
\end{bNiceMatrix}
$}  \label{bandedmatrixv04}
\end{align}
\begin{tikzpicture}[remember picture, overlay]
 \draw [draw=white] (4.2,4.3) -- (4.21,4.3)  node[above=0.01mm] {\mbox{\normalsize $N_{C}^{[n]}$} };
 \draw [draw=white] (7.8,4.3) -- (7.81,4.3)  node[above=0.01mm] {\mbox{\normalsize $N_{C}^{[n,2dq]}$}};
 \draw [draw=white] (10.2,4.30) -- (10.21,4.30)  node[above=0.01mm] {\mbox{\normalsize $B^{[n]}(z)$ }};
 \draw [draw=white] (8.2,1.0) -- (8.21,1.0)  node[above=0.01mm] {\mbox{\normalsize $B^{[n,2dq]}(z)$ }};
 \draw [draw=white] (12.2,1.7) -- (12.21,1.7)  node[above=0.01mm] {\mbox{\normalsize $N_{C}^{[n]}B^{[n]}(z)$} };
 \draw [draw=white] (15.00,1.7) -- (15.01,1.7)  node[above=0.01mm] {\mbox{\scriptsize $N_{C}^{[n,2dq]}B^{[n,2dq]}(z)$} };
\end{tikzpicture}



Observe that the above expression takes a different form for polynomials $%
A(z^{-1})$, because the matrix $N_{C}$ is a banded upper triangular
matrix, so we only have zeros under the main diagonal, and therefore we
actually have 
\begin{align}
\hat{A}^{[n]}(z^{-1})N_{C}^{[n]}=A^{[n]}(z^{-1}).  \label{truncaANn}
\end{align}


Taking into account column by column the above expressions, and the
definition \eqref{CDKernel(x,y)ab}, the modified kernel is given by%
\begin{align*}
W_{b}(y)\hat{K}_{a,b}^{[n]}(x,y) =&W_{b}(y)\left( \hat{A}%
^{(a)}(x^{-1})\right) ^{[n]}\left( \hat{B}^{(b)}(y)\right) ^{[n]} \\
=&\left( \hat{A}^{(a)}(x^{-1})\right) ^{[n]}\left( N_{C}^{[n]}\left(
B^{(b)}(y)\right) ^{[n]}+N_{C}^{[n,2dq
]}\left( B^{(b)}(y)\right) ^{[n,2dq]}\right) \\
=&\left( \hat{A}^{(a)}(x^{-1})\right) ^{[n]}N_{C}^{[n]}\left(
B^{(b)}(y)\right) ^{[n]}+\left( \hat{A}^{(a)}(x^{-1})\right) ^{[n]}N_{C}^{[n,2dq]}\left( B^{(b)}(y)\right) ^{[n,2dq]} \\
=&\left( A^{(a)}(x^{-1})\right) ^{[n]}\left( B^{(b)}(y)\right)
^{[n]}+\left( \hat{A}^{(a)}(x^{-1})\right) ^{[n]}N_{C}^{[n,2dq ]}\left( B^{(b)}(y)\right) ^{[n,2dq]}
\end{align*}%
so, we thus conclude%
\begin{align}
W_{b}(y)\hat{K}_{a,b}^{[n]}(x,y)=K_{a,b}^{[n]}(x,y)+\left( \hat{A}%
^{(a)}(x^{-1})\right) ^{[n]}N_{C}^{[n,2dq]}\,\left( B^{(b)}(y)\right) ^{[n,2dq]}.  \label{connKernels}
\end{align}



\subsection{Christoffel formula for $\hat{B}(z)$}



Next, we provide the corresponding Christoffel formula for the CMV Laurent
polynomials $\hat{B}(z)$. With the same notation $\tilde{N_{C}}%
_{i} $ as above, for \eqref{sumaNB} we have%
\begin{align*}
\sum_{m=n}^{n+2dq-1} N_{C,n,m}\,B_{m}^{(b)}(z_{b,\ell })&=-\tilde{N}_{C,n}\,B_{n+2dq}^{(b)}(z_{b,\ell }),& \begin{aligned}
	b&\in\{1,\ldots ,q\},&\ell &\in \{1,\ldots ,2d\}.
\end{aligned}
\end{align*}%
In the sequel we use the following notation
\begin{align} \label{Bgotica}
	\mathcal B_n\coloneq 			\begin{bmatrix}
		B_{n}^{(1)}(z_{1,1}) & \cdots & B_{n}^{(1)}(z_{1,2d}) & \cdots & 
		B_{n}^{(q)}(z_{q,1}) & \cdots & B_{n}^{(q)}(z_{q,2d}) \\[3pt]
		B_{n+1}^{(1)}(z_{1,1}) & \cdots & B_{n+1}^{(1)}(z_{1,2d}) & \cdots & 
		B_{n+1}^{(q)}(z_{q,1}) & \cdots & B_{n+1}^{(q)}(z_{q,2d}) \\ 
		\vdots &  & \vdots &  & \vdots &  & \vdots \\ \\
		B_{n+2dq-1}^{(1)}(z_{1,1}) & \cdots & B_{n+2dq-1}^{(1)}(z_{1,2d}) & \cdots & 
		B_{n+2dq-1}^{(q)}(z_{q,1}) & \cdots & B_{n+2dq-1}^{(q)}(z_{q,2d})%
	\end{bmatrix}.%
\end{align}%
This squared $2dq\times 2dq$ matrix comes from the evaluation of the
matrix $\left( B^{(b)}(z_{b,\ell })\right) ^{[n,2dq]}$ at the $2dq$
different zeros $z_{b,\ell }$ , $b\in\{1,\ldots ,q\}$, $\ell \in \{1,\ldots ,2d\}$
of the perturbation matrix $W_{[q]}$ defined in \eqref{mbLp[r]}.
Then, in matrix form, for every $n\in\N_0 $ we get%
\begin{align*}
\begin{bmatrix}
N_{C,n,n} & N_{C,n,n+1} & \cdots & N_{C,n,n+2dq-1}
\end{bmatrix} \mathcal B_n%
=-\tilde{N}_{C,n}%
\begin{bmatrix}
B_{n+2dq}^{(1)}(z_{1,1}) & B_{i+2dq}^{(2)}(z_{2,2}) & \cdots & 
B_{n+2dq}^{(q)}(z_{q,2d})%
\end{bmatrix}.%
\end{align*}



The solution of this equation provides all the nonzero entries of the connection matrix $N_{C}$, which are our unknowns. Thus, replacing all these values in \eqref{connformB} leads to
\begin{align*}
W_{b}(z)\frac{1}{\tilde{N}_{C,n}}\hat{B}_{i}^{(b)}(z)=B_{n+2dq}^{(b)}(z)-
\begin{bmatrix}
B_{n+2dq}^{(1)}(z_{1,1}) & \cdots & B_{n+2dq}^{(q)}(z_{q,2d})%
\end{bmatrix}%
{%
\mathcal B_n
}^{-1}%
\begin{bmatrix}
B_{n}^{(b)}(z) \\ 
\vdots \\ 
B_{n+2dq-1}^{(b)}(z)%
\end{bmatrix}.%
\end{align*}%
Now, let us define the quasi-determinant $\Theta _{\ast }$\ of a given block
matrix as follows (see, for example \cite[Sec. 3.1]{M-EIBPOA21}, \cite%
{GGRW-AM05})%
\begin{align}
\Theta _{\ast }\begin{bNiceMatrix}[margin,hvlines] A & b \\ \hline c & d \\
\end{bNiceMatrix}\coloneq d-cA^{-1}b=\frac{%
\begin{vNiceMatrix}[margin,hvlines] A & b \\ \hline c & d \\
\end{vNiceMatrix}}{\left\vert A\right\vert }.  \label{QuasiDet}
\end{align}%
From this, we found that the CMV mixed multiple Laurent OPUC on the step
line $\hat{B}(z)$ associated with the Christoffel $W$-modified matrix of
measures supported on the unit circle $\d\hat{\mu}(z)=W_{[q]}\d\mu (z)$,
can be represented in terms of the polynomials $B(z)$ as the following
quasi-determinant%
\begin{align*}
\frac{1}{\tilde{N}_{C,i}}\hat{B}_{n}^{(b)}(z)=\frac{1}{W_{b}(z)}%
\Theta _{\ast }{%
\begin{bmatrix}
B_{n}^{(1)}(z_{1,1}) & \cdots & B_{n}^{(q)}(z_{q,2d}) & B_{n}^{(b)}(z) \\ 
\vdots &  & \vdots & \vdots \\ 
B_{n+2dq-1}^{(1)}(z_{1,1}) & \cdots & B_{n+2dq-1}^{(q)}(z_{q,2d}) & 
B_{n+2dq-1}^{(b)}(z) \\[3pt]
B_{n+2dq}^{(1)}(z_{1,1}) & \cdots & B_{n+2dq-1}^{(q)}(z_{q,2d}) & 
B_{n+2dq}^{(b)}(z)%
\end{bmatrix}%
.}
\end{align*}

Finally, taking again into account \eqref{QuasiDet}, we have proved the
following



\begin{pro}
The Christoffel formula for the perturbed CMV mixed multiple Laurent OPUC on
the step-line $\hat{B}(z)$ can be written as follows%
\begin{align*}
\frac{1}{\tilde{N}_{C,n}}\hat{B}_{n}^{(b)}(z)=\frac{1}{W_{b}(z)}{%
\frac{{%
\begin{vmatrix}
B_{n}^{(1)}(z_{1,1}) & \cdots & B_{n}^{(q)}(z_{q,2d}) & B_{n}^{(b)}(z) \\ 
\vdots &  & \vdots & \vdots \\ 
B_{n+2dq-1}^{(1)}(z_{1,1}) & \cdots & B_{n+2dq-1}^{(q)}(z_{q,2d}) & 
B_{n+2dq-1}^{(b)}(z) \\ 
B_{n+2dq}^{(1)}(z_{1,1}) & \cdots & B_{n+2dq-1}^{(q)}(z_{q,2d}) & 
B_{n+2dq}^{(b)}(z)%
\end{vmatrix}%
}}{{%
\begin{vmatrix}
B_{n}^{(1)}(z_{1,1}) & \cdots & B_{n}^{(q)}(z_{q,2d}) \\ 
\vdots &  & \vdots \\ 
B_{n+2dq-1}^{(1)}(z_{1,1}) & \cdots & B_{n+2dq-1}^{(q)}(z_{q,2d})%
\end{vmatrix}%
}}.}
\end{align*}
\end{pro}



\begin{Remark}
\label{remBgotic copy(1)} It is important to note that the entries of the
determinant at the denominator of the above expression depend on the $2dq$
different zeros of the matrix balanced Laurent polynomials $W_{[q]}$, so in
the sequel we only consider those set of $2dq$ zeros such that the
aforementioned determinant is not zero.
\end{Remark}



\subsection{Christoffel formula for $\hat{A}(z^{-1})$}



We next consider the Christoffel formula for the CMV Laurent polynomials $A(z^{-1})$, which come in terms of the CD kernels of the above section \eqref{S03-CDkernels}. Evaluating formula \eqref{connKernels} for a fixed $b\in\{1,\ldots ,q\}$, at the zeros of the polynomials $W_{b}(y)$, for $y=z_{b,\ell}$, $\ell \in \{1,\ldots ,2d\}$, yields $W_{b}(z_{b,\ell })=0$.
We will use the notation
\begin{align}
	\kappa _{a,b,\ell}^{[n]}(x)&\coloneq K_{a,b}^{[n]}(x,z_{b,\ell }), & \kappa _{a}^{[n]}(x)\coloneq%
	\begin{bmatrix}
		\kappa _{a,1,1}^{[n]}(x)&\cdots & 	\kappa _{a,1,2d}^{[n]}(x) &\cdots & \kappa _{a,q,1}^{[n]}(x) & \cdots & \kappa
		_{a,q,2d}^{[n]}(x)%
	\end{bmatrix}%
\end{align}

%
%
%
%
%
\begin{Remark}
\label{remBgotic} It is important to note that the entries in $\Pi $ depend
on the $2dq$ different zeros of the matrix balanced Laurent polynomials $%
W_{[q]}$, so in the sequel we only consider those set of $2dq$ zeros such
that $\det \Pi \neq 0$. That is, we only deal with those cases in which $\Pi 
$ is nonsingular, and therefore it is an invertible matrix.
\end{Remark}



Thus,we get 

\begin{align*}
\kappa^{[n]}_a(x)%
=-\left( \hat{A}^{(a)}(x^{-1})\right) ^{[n]}N_{C}^{[n,2dq ]}\,%
\mathcal B_n,
\end{align*}
Next, let $\boldsymbol{e}_{s}^{[r]}$, $s,r\in \mathbb{N}$, $1\leq s\leq r$ be the $r\times 1$ column vector
\begin{align}
	\boldsymbol{e}_{s}^{[r]}=%
\begin{bNiceMatrix}[last-row,code-for-last-row = \scriptsize]
\Cdots[shorten=6pt] & 0 & \Cdots[shorten=5pt] \, 1 \, \Cdots[shorten=5pt] & 0 & \Cdots[shorten=6pt]	 \\
		&   & \overset{\uparrow }{\text{s-th position}} & &        \\  
	\CodeAfter
	\UnderBrace[shorten,yshift=18pt]{1-1}{1-5}{\text{\scriptsize $r$}}
\end{bNiceMatrix}^{\top }
\label{ecolumnvec}
\end{align}

\vspace*{1cm}
with just one at position $s$ and zeros everywhere else. Also we use $	\boldsymbol{e}_{s}=	\boldsymbol{e}_{s}^{[s]}$.

 According to remark \ref{remBgotic}, we only consider the cases in which the square matrix $\Pi $%
\ is invertible, so we can write%
\begin{align*}
\kappa _{a}^{[n]}(x)\mathcal B_n^{-1}	\boldsymbol{e}_{2dq}=-\left( \hat{A}^{(a)}(x^{-1})\right) ^{[n]}\,N_{C}^{[n,2dq ]}\,,
\end{align*}%
and taking into account \eqref{ecolumnvec}, and the structure of the matrix $N_{C}^{[n,2dq ]}$ we can finally state
\begin{align*}
\kappa _{a}^{[n]}(x)\mathcal B_n ^{-1}\boldsymbol{e}^{[2dq]}_{2dq}=-\left( \hat{A}^{(a)}(x^{-1})\right) ^{[n]}\,\tilde{N}_{C,n-1}\,,
\end{align*}%
where we denoted by $\tilde{N}_{C,n}=N_{C,n,n+2dq+1}$.

Thus, we have proved the following



\begin{Theorem}
The mixed multiple Laurent OPUC on the step line $\hat{A}(z^{-1})$
associated with the Christoffel $W$-modified matrix of measures supported on
the unit circle $\d\hat{\mu}(z)=W_{[q]}\d\mu (z)$, can be represented in
terms of the polynomials $A(z^{-1})$ as follows%
\begin{align*}
\hat{A}^{(a)}(x^{-1})=\frac{-1}{\tilde{N}_{C,n-1}}\kappa
_{a}^{[n]}(x)\mathcal B_n ^{-1}.
\end{align*}
\end{Theorem}



Hence, with definition \eqref{QuasiDet} at hand, we can rewrite $\hat{A}%
^{(a)}(x^{-1})$ in terms of quasi-determinants as follows 
\begin{align*}
\hat{A}^{(a)}(x^{-1})=\frac{-1}{\tilde{N}_{C,n-1}}\kappa
_{a}^{[n]}(x)\,)\mathcal B_n^{-1}\boldsymbol{e}_{2dq}=\frac{1}{\tilde{N}_{C,n-1}}\Theta_{\ast }\begin{bNiceMatrix}[margin,hvlines] \mathcal B_n &
\boldsymbol{e}_{2dq} \\ \hline \kappa _{a}^{[n]}(x) & 0 \\ \end{bNiceMatrix}=%
\frac{1}{\tilde{N}_{C,n-1}}\frac{%
\begin{vNiceMatrix}[margin,hvlines]\mathcal B_n& \boldsymbol{e}_{2dq} \\ \hline \kappa
_{a}^{[n]}(x) & 0 \\ \end{vNiceMatrix}}{\left\vert \mathcal B_n\right\vert }
\end{align*}%
or, more explicitly, expanding the determinant of the numerator by the last
column%
\begin{align*}
\hat{A}^{(a)}(x^{-1})=\frac{-1}{\tilde{N}_{C,n-1}}{\scriptsize {%
\frac{{%
\begin{vmatrix}
B_{i}^{(1)}(z_{1,1}) & \cdots & B_{i}^{(1)}(z_{1,2d}) & \cdots & 
B_{i}^{(q)}(z_{q,1}) & \cdots & B_{i}^{(q)}(z_{q,2d}) & 0 \\ 
B_{i+1}^{(1)}(z_{1,1}) & \cdots & B_{i+1}^{(1)}(z_{1,2d}) & \cdots & 
B_{i+1}^{(q)}(z_{q,1}) & \cdots & B_{i+1}^{(q)}(z_{q,2d}) & 0 \\ 
\vdots &  & \vdots &  & \vdots &  & \vdots & \vdots \\ 
B_{i+2dq-1}^{(1)}(z_{1,1}) & \cdots & B_{i+2dq-1}^{(1)}(z_{1,2d}) & \cdots & 
B_{i+2dq-1}^{(q)}(z_{q,1}) &  & B_{i+2dq-1}^{(q)}(z_{q,2d}) & 1 \\ 
\kappa _{a,1}^{[n]}(x) & \cdots & \kappa _{a,1}^{[n]}(x) & \cdots & \kappa
_{a,q}^{[n]}(x) &  & \kappa _{a,q}^{[n]}(x) & 0%
\end{vmatrix}%
}}{{%
\begin{vmatrix}
B_{i}^{(1)}(z_{1,1}) & \cdots & B_{i}^{(1)}(z_{1,2d}) & \cdots & 
B_{i}^{(q)}(z_{q,1}) & \cdots & B_{i}^{(q)}(z_{q,2d}) \\ 
B_{i+1}^{(1)}(z_{1,1}) & \cdots & B_{i+1}^{(1)}(z_{1,2d}) & \cdots & 
B_{i+1}^{(q)}(z_{q,1}) & \cdots & B_{i+1}^{(q)}(z_{q,2d}) \\ 
\vdots &  & \vdots &  & \vdots &  & \vdots \\ 
B_{i+2dq-1}^{(1)}(z_{1,1}) & \cdots & B_{i+2dq-1}^{(1)}(z_{1,2d}) & \cdots & 
B_{i+2dq-1}^{(q)}(z_{q,1}) & \cdots & B_{i+2dq-1}^{(q)}(z_{q,2d})%
\end{vmatrix}%
}}.}}
\end{align*}

\section{Geronimus perturbation}

\label{S05-GeronPerturb}

In this section we are going to study the so called Geronimus perturbations of the matrix measure, corresponding to divisions of $\d\mu (z)$ by certain appropiated Laurent polynomials. As in the previous section, given the matrix nature of the problem we will find again two types of Geronimus modifications, namely the LHS and the RHS Geronimus perturbations of $\d\mu (z)$. In a similar way, and since the theoretical development of the RHS case is the dual of the LHS case, we will perform in detail just the LHS Geronimus perturbations of $\d\mu (z)$, and for the dual case we will restrict ourselves to presenting the main results in the last part of this section.

In the sequel, we follow the ideas in \cite{AGMM-BMS19}, \cite{amt} and \cite{AAM}, and as in those papers, the so called second kind functions will play an important role in the subsequent discussion. Hence, let us define certain $q\times p$, with $q,p\in \mathbb{N}$\ matrix of complex measures $\d\check{\mu}(z)$, supported on the unit circle $\mathbb{T}$, and such that the following identity holds%
\begin{align}
\d\mu (z)=W_{[q]}(z)\d\check{\mu}(z)=%
\begin{bmatrix}
W_{1}(z) &  &  \\ 
& \ddots &  \\ 
&  & W_{q}(z)%
\end{bmatrix}%
\d\check{\mu}(z).  \label{measuGeron}
\end{align}%
$W_{[q]}(z)$ is a matrix prepared Laurent polynomial, as that defined in section \ref{S04-ChristPerturb}. Concerning the related matrices of moments, $M$ was define in \eqref{MMleft}, namely $M=L^{-1}U^{-1}$, and we define LHS Geronimus $W$-modified moment matrix $\mathcal{\check{M}}$ as follows
\begin{align}
\mathcal{\check{M}}\coloneq\oint_{\mathbb{T}}Z_{[q]}(z)\d\check{\mu}(z)\,Z_{[p]}^{\top }(z^{-1})  \label{Mcheck}
\end{align}

\begin{pro}
With the notation \eqref{Wsuma}, the LHS Geronimus perturbed CMV moment
matrix $\mathcal{\check{M}}$, and the moment matrix $M$, can be
related as follows%
\begin{align*}
\mathcal{W\check{M}}=M
\end{align*}
\end{pro}

\begin{proof}
The proof is quite obvious from the definition of the matrices $\mathcal{%
\check{M}}$ and $M$. Thus%
\begin{align*}
M =&\oint_{\mathbb{T}}Z_{[q]}(z)\d\mu
(z)\,Z_{[p]}^{\top }(z^{-1})=\oint_{\mathbb{T}}Z_{[q]}(z)\,W_{[q]}(z)\d\check{\mu}(z)\,Z_{[p]}^{\top }(z^{-1}) \\
=&\left( \sum_{b=1}^{q}W_{b}\left( \Upsilon_{[q]}^{(b)}\right)
\right) \oint_{\mathbb{T}}Z_{[q]}(z)\d\check{\mu}%
(z)\,Z_{[p]}^{\top }(z^{-1})=\mathcal{W\check{M}}.
\end{align*}
\end{proof}

By analogy, in the sequel we will denote by 
\begin{align}
\check{B}(z)\coloneq\check{L}Z_{[q]}(z)\text{, \ and \ \ }\check{A}(z^{-1})%
\coloneq Z_{[p]}^{\top }(z^{-1})\check{U}  \label{BApolGer}
\end{align}%
the CMV matrix Laurent polynomials, biorthogonal with respect to the
Geronimus $W$-modified matrix of measures $\d\check{\mu}$.

\subsection{Connector matrix and connection formulas}

Next we consider the Gauss factorization of the moment matrix $%
\mathcal{\check{M}}=\check{L}^{-1}\check{U}^{-1}$, and taking into account
the last proposition, we have%
\begin{align*}
L^{-1}U^{-1}=\mathcal{W\check{M}}=\mathcal{W}\check{L}^{-1}\check{U}^{-1}.
\end{align*}

Multiplying the last equation by $\check{U}$ on the r.h.s., and by $L$ on
the l.h.s., we call the resulting semi-infinite upper triangular \textit{%
connection matrix}%
\begin{align}
N_{G}\coloneq U^{-1}\check{U}=L\sum_{b=1}^{q}W_{b}(\Upsilon ^{(b)})%
\check{L}^{-1}=L\,\mathcal{W\,}\check{L}^{-1}.  \label{defmatrixNGer}
\end{align}

In this case, the connector, or connection matrix is given by $N_{G}$,
and it links the original with the Geronimus $W$-modified sequences of
orthogonal polynomials. As in the Christoffel modifiying case, we can state
as follows the connection properties of $N_{G}$

\begin{pro}[Connection formulas]
The following connection formulas, between CMV matrix Laurent polynomials
biorthogonals with respect to $\d\mu (z)$ and $\d\check{\mu}(z)$, hold:%
\begin{align}
\begin{array}{lllll}
N_{G}\,\check{B}(z)=B(z)W_{[q]}(z), &  &  &  & N_{G}\,\check{B}%
^{(b)}(z)=B^{(b)}(z)W_{b}(z), \\ 
&  &  &  &  \\ 
A(z^{-1})\,N_{G}=\check{A}(z^{-1}), &  &  &  & A^{(a)}(z^{-1})N^G=\check{A}^{(a)}(z^{-1}).%
\end{array}
\label{relatABGer}
\end{align}
\end{pro}

\begin{proof}
From \eqref{ZnmonCMVj} and \eqref{BApolGer} we can therefore study the
action of the matrix $N_{G}$ on $\check{B}(z)$, namely%
\begin{align*}
N_{G}\mathcal{\,}\check{B}(z)=L\,\mathcal{W\,}\check{L}^{-1}\check{L}%
\,Z_{[q]}(z)=L\mathcal{W}Z_{[q]}(z)=\mathcal{W}\,B(z),
\end{align*}%
and therefore, on the columns of $B(z)$, we have%
\begin{align}
N_{G}\mathcal{\,}\check{B}^{(b)}=L\mathcal{W}%
Z_{[q]}^{(b)}=W_{b}(z)B^{(b)},  \label{conformB0Ger}
\end{align}%
which provides the following connection formula for the polynomials $\check{B%
}^{(b)}(z)$ and $B^{(b)}(z)$%
\begin{align}
W_{b}(z)\check{B}^{(b)}(z)=N_{G}\mathcal{\,}B^{(b)}(z).
\label{connformBGer}
\end{align}%
Proceeding in a similar manner, for $a\in \{1,\ldots ,p\}$ we find that%
\begin{align*}
\check{A}(z^{-1})=Z_{[p]}^{\top }(z^{-1})\check{U}=%
\begin{bmatrix}
\check{A}^{(1)}(z^{-1}) \\ 
\check{A}^{(2)}(z^{-1}) \\ 
\vdots \\ 
\check{A}^{(p)}(z^{-1})%
\end{bmatrix}%
=%
\begin{bmatrix}
\left( Z_{[p]}^{(1)}(z^{-1})\right) ^{\top } \\ 
\left( Z_{[p]}^{(2)}(z^{-1})\right) ^{\top } \\ 
\vdots \\ 
\left( Z_{[p]}^{(p)}(z^{-1})\right) ^{\top }%
\end{bmatrix}%
\check{U}
\end{align*}%
and therefore, the action of the upper triangular matrix $N_{G}$ on the
r.h.s. $A(z^{-1})$, yields%
\begin{align*}
A(z^{-1})\,N_{G}=Z_{[p]}^{\top }(z^{-1})UU^{-1}\check{U}=Z_{[p]}^{\top
}(z^{-1})\check{U}=\check{A}(z^{-1}),
\end{align*}%
which provides a connection formula between $A(z^{-1})$ and $\check{A}%
(z^{-1})$%
\begin{align}
A(z^{-1})\,N_{G}=\check{A}(z^{-1}).  \label{connformAGer}
\end{align}%
Component-wise for $\check{A}^{(a)}(z^{-1})$ and $A^{(a)}(z^{-1})$, the above
expression easily follows.
\end{proof}

As in section \ref{S04-ChristPerturb}, we briefly discuss the shape of the connection matrix reason why we need perturbations as prepared Laurent polynomials of the same degree in $W_{[q]}(z)$, and this relies on the particular structure of the matrix $\Upsilon_{[q]}$\ and its powers. From \eqref{Wsuma} and \eqref{defmatrixN} we know
\begin{align*}
N_{G}=U^{-1}\check{U}=L\left(\sum_{b=1}^{q}W_{b}(\Upsilon^{(b)})\right) \check{L}^{-1}=L\mathcal{W}\check{L}^{-1}.
\end{align*}
On one side $U^{-1}\check{U}$ means that $N_{G}$ must be an upper triangular matrix, and on the other side one encounters that $L\,\mathcal{W}\check{L}^{-1}$ must be a lower Hessenberg matrix depending on the particular structure of the matrix $\sum_{b=1}^{q}W_{b} \left( \Upsilon_{[q]}^{(b)}\right) $, that in turn depends on the degrees $\pm d$, and the number $q$ of the prepared Laurent polynomials at the diagonal of $W_{[q]}\left( \Upsilon_{[q]}\right) $. 
From \eqref{powersUpsilon}, the matrix $W_{[q]}\left( \Upsilon_{[q]}\right)$ will be a $\left( 4dq+1\right) $-diagonal matrix, which in turn implies that its highest nonzero upper-diagonal will be at position $(2dq+1)$. 
Putting together the two  requirements mentioned above, it follows that $N_{G}$ will be, as in the Christoffel case, a semi-infinite upper $(2dq+1)$-banded matrix. Thus, the shape of $N_{G}$ is
\vspace*{.5cm}
\begin{align}
\renewcommand{\arraystretch}{1.0}N_{G}=\begin{bNiceMatrix} \ast &
&\Cdots_{2dq+1}& &\ast & & & & & & \\ &\Ddots & & & &\Ddots & & & & & \\ & &
& & & & \ast & & & & \\ & & & & & & &\blacktriangle & & & \\ & & & & & & & &
& & \\ & & & & & &\Vdots_{2dq+1}&\Vdots^{2dq} & &\Ddots& \\ & & & & & & \ast
&\blacktriangle &\Cdots_{2dq+1}& &\blacktriangle \\ \CodeAfter \tikz \draw
(1-|8) -- (8-|8); \end{bNiceMatrix}  \label{bandedmatrixv06}
\end{align}%
\begin{tikzpicture}[remember picture, overlay]
  \draw [draw=white] (6.1,3.9) -- (6.15,3.9)  node[above=0.15mm] {${\scriptstyle N_{G,n,\ell}}$};
  \draw [draw=white] (9.2,3.9) -- (9.25,3.9)  node[above=0.15mm] {${\scriptstyle \tilde N_{G,\ell} = N_{G,n,n+2dq}}$};
\end{tikzpicture}

We next prove the following result concerning the connection between $C(z)$, $D(z)$, $\check{C}(z)$ and $\check{D}(z)$:
\begin{pro}
For $C(z)$, $D(z)$, $\check{C}(z)$ and $\check{D}(z)$, the following connection formulas follow
\begin{align}
W_{[q]}(z)\check{C}(z)-C(z)\,N_{G}=&\oint_{\mathbb{T}}\frac{%
W_{[q]}(z)-W_{[q]}(x)}{z-x}\d\check{\mu}(x)\check{A}(x^{-1}),
\label{skfCconn} \\
\,N_{G}\,\check{D}(z) =&D(z).  \label{skfDconn}
\end{align}
\end{pro}

\begin{proof}
Multipliying $\check{C}(z)$ in \eqref{skfC} by $W_{[q]}(z)$ on the left hand
side, we have%
\begin{align}
W_{[q]}(z)\check{C}(z)=\oint_{\mathbb{T}}W_{[q]}(z)\frac{1}{z-x}\d\check{\mu}(x)\check{A}(x^{-1}).  \label{Ger1}
\end{align}%
On the other hand, multiplying $C(z)$ in \eqref{skfC} by $N_{G}$ on the right hand side, and combining with \eqref{measuGeron} and \eqref{relatABGer}, yields
\begin{align}
C(z)\,N_{G}=\oint_{\mathbb{T}}\frac{1}{z-x}\d\mu (x)A(x^{-1})\,%
N_{G}=\oint_{\mathbb{T}}\frac{1}{z-x}W_{[q]}(x)\d\check{\mu}(x)%
\check{A}(x^{-1}).  \label{Ger2}
\end{align}%
Next, subtracting \eqref{Ger1} minus \eqref{Ger2}, implies%
\begin{align*}
W_{[q]}(z)\check{C}(z)-C(z)N_{G}
\end{align*}%
\begin{align*}
=\oint_{\mathbb{T}}W_{[q]}(z)\frac{1}{z-x}\d\check{\mu}(x)\check{A}%
(x^{-1})-\oint_{\mathbb{T}}\frac{1}{z-x}W_{[q]}(x)\d\check{\mu}(x)%
\check{A}(x^{-1})
\end{align*}%
which proves \eqref{skfCconn}. To prove \eqref{skfDconn} it is enough to multiply $\check{D}(z)$  by $N_{G}$ on the left hand
side, and combining with \eqref{measuGeron} and \eqref{relatABGer}, yields
\begin{align*}
N_{G}\,\check{D}(z)=\oint_{\mathbb{T}}N_{G}\,\frac{1}{z-x}%
\check{B}(x)\d\check{\mu}(x)=\oint_{\mathbb{T}}\frac{1}{z-x}%
\mathcal{W}B(x)\d\check{\mu}(x)=\oint_{\mathbb{T}}\frac{1}{z-x}%
B(x)\d\mu (x)=D(z).
\end{align*}%
\end{proof}

\subsubsection{Connection formula between the CD kernels}

In this framework of Geronimus $W$-modifications, note that from above formulas it is a simple matter to obtain the connection formula between the corresponding truncated kernels $K^{[n]}(x,y)$ and $\check{K}^{[n]}(x,y)$.
Concerning the Christoffel--Darboux kernels, coming back to their definition \eqref{KernelAB[n]} in section \ref{S03-CDkernels}, we have
\begin{align}
K^{[n]}(x,y)=A^{[n]}(x^{-1})B^{[n]}(y)=A(x^{-1})\,\pi ^{[ n]}\,B(y).
\label{defKer}
\end{align}%
Thus, considering the truncations of $\left( N_{G}\,\check{B}(y)\right) ^{[n]}$ and $\left( A(x)\,N_{G}\right) ^{[n]}$ we have a very similar
situation as in \eqref{bandedmatrixv04}, but having the connection matrix $N_{G}$ instead of $N_{C}$. This leads in a straightforward way to the following expressions
\begin{align*}
\left( N_{G}\,\check{B}(y)\right) ^{[n]} =&N_{G}^{[n]}\check{B}%
^{[n]}(y)+N_{G}^{[n,2dq]}\check{B}^{[n,2dq]}(y), \\
\left( A(x^{-1})\,N_{G}\right) ^{[n]} =&A^{[n]}(x^{-1})N_{G}%
^{[n]}=\check{A}^{[n]}(x^{-1}),
\end{align*}%
respectively analog to \eqref{truncaNBn} and \eqref{truncaANn}. From the above expressions we have
\begin{align}
\left( N_{G}\,\check{B}(y)\right) ^{[n]}=N_{G}^{[n]}\,\check{B}^{[n]}(y)+N_{G}^{[n,2dq]}\check{B}^{[n,2dq]}(y)=B^{[n]}(y)W_{[q]}(z).
\label{NBcheckn}
\end{align}%
Taking into account column by column the above expressions, defining
\begin{align}
\check{K}^{[n]}(x,y)=\check{A}^{[n]}(x^{-1})\check{B}^{[n]}(y)=\check{A}
(x^{-1})\,\pi ^{[ n]}\,\check{B}(y)  \label{defKerGen}
\end{align}%
and the definition \eqref{CDKernel(x,y)ab}, the Geronimus $W$-modified kernel is given by \eqref{conformB0Ger}
\begin{align*}
W_{b}(y)K_{a,b}^{[n]}(x,y) =&\left( A^{(a)}(x^{-1})\right) ^{[n]}\left( 
N_{G}\mathcal{\,}\check{B}^{(b)}(y)\right) ^{[n]} \\
=&\left( \check{A}^{(a)}(x^{-1})\right) ^{[n]}(\check{B}^{(b)}(y))^{[n]}+
\left( A^{(a)}(x^{-1})\right) ^{[n]}N_{G}^{[n,2dq]}(\check{B}%
^{(b)}(y))^{[n,2dq]} \\
=&\check{K}_{a,b}^{[n]}(x,y)+\left( A^{(a)}(x^{-1})\right) ^{[n]}N_{G}%
^{[n,2dq]}(\check{B}^{(b)}(y))^{[n,2dq]}.
\end{align*}%
Thus, we conclude%
\begin{align}
\check{K}^{[n]}(x,y)-K^{[n]}(x,y)W(y)=\left( A(x^{-1})\right) ^{[n]}\mathsf{N}^{[n,2dq]}\check{B}^{[n]}(y).  \label{connKerGer}
\end{align}

\subsection{Geronimus perturbation with singular part}

We next consider a singular part $\d\mu _{s}(z)$ in the definition of the Geronimus $W$-mofication of the matrix of measures $\d\mu (z)$. Thus, we will have
\begin{align}
\d\check{\mu}(z)\coloneq\left( W_{[q]}(z)\right) ^{-1} \d\mu (z)+ \d\mu_{s}(z).  \label{MeasuGeronSing}
\end{align}%
Taking into account that the entries in the diagonal matrix%
\begin{align*}
W_{[q]}(z)=\text{diag }\left( W_{1}(z),\ldots ,W_{r}(z)\right) ,
\end{align*}
are well-poised Laurent polynomials of the same degrees $\pm d$ as we defined in section \ref{S04-ChristPerturb}, $W_{[q]}(z)$ has a total of $2dq$ different and simple zeros. It is importan to note that the singular part of the measure must fulfilled
\begin{align*}
W_{[q]}(z)\d\mu _{s}(z)=\boldsymbol{0}_{q\times p}
\end{align*}%
in order to verify%
\begin{align*}
W_{[q]}(z)\d\check{\mu}(z)=W_{[q]}(z)\left( W_{[q]}(z)\right) ^{-1}\d\mu
(z)+W_{[q]}(z)\d\mu _{s}(z)=\d\mu (z)
\end{align*}%
and thus recover \eqref{measuGeron}. Under these conditions, we choose $\d\mu _{s}(z)$ to be a $q\times p$ matrix whose entries are $2dq$ Dirac
deltas located exactly at the $2dq$ simple zeros of the perturbation matrix $%
W_{[q]}(z)$. We denote these zeros as $z_{b,j}$ for every $j\in \{1,\ldots
,2d\}$, and $b\in \{1,\ldots ,q\}$. Being $E_{b,a}$ a $q\times p$ matrix
with one at position $(b,a)$ and zeros everywhere else, the singular part of
the $q\times p$ matrix measure $\d\mu _{s}(z)$, can be written as%
\begin{align}
\d\mu
_{s}(z)=\sum_{a=1}^{p}\sum_{b=1}^{q}\sum_{j=1}^{2d}E_{b,a}\,m_{b,a,j}\,%
\delta (z-z_{b,j})=\sum_{b=1}^{q}\boldsymbol{e}_{b}^{[q]}\boldsymbol{\upsilon }%
_{b}^{[p]}(z)  \label{SigmaSingular01}
\end{align}%
where%
\begin{align*}
\boldsymbol{\upsilon }_{b}^{[p]}(z)\coloneq\sum_{a=1}^{p}\sum_{j=1}^{2d}(%
\boldsymbol{e}_{a}^{[p]})^{\top }m_{b,a,j}\,\delta (z-z_{b,j}).
\end{align*}

\begin{Remark}
\label{remSupportSing} The $b$-th row, with $b\in \{1,\ldots ,q\}$, of the
singular part $\d\mu _{s}(z)$ of the matrix measure $\d\check{\mu}(z)$,
is supported at the $2d$ different zeros of the Laurent polynomial $W_{b}(z)$
at the $(b,b)$ diagonal entry of the matrix $W_{[q]}(z)$.
\end{Remark}

Next, replacing \eqref{MeasuGeronSing} in the expression for \eqref{skfC}, and taking into account $A(z^{-1})\,N_{G}=\check{A}(z^{-1})$ in \eqref{relatABGer}, we know
\begin{align*}
\check{C}(z) =&\oint_{\mathbb{T}}\frac{1}{z-x}\d\check{\mu}(x)%
\check{A}(x^{-1}) \\
=&\oint_{\mathbb{T}}\frac{1}{z-x}\left( \left( W_{[q]}(x)\right)
^{-1}\d\mu (x)+\d\mu _{s}(x)\right) \check{A}(x^{-1}) \\
=&\oint_{\mathbb{T}}\frac{1}{z-x}\left( \left( W_{[q]}(x)\right)
^{-1}\d\mu (x)+\d\mu _{s}(x)\right) A(x^{-1})\,N_{G}\\
=&\mathcal{C}(z)\,\mathsf{N},
\end{align*}%
where we denote%
\begin{align}
\mathcal{C}(z)\coloneq\oint_{\mathbb{T}}\frac{1}{z-x}\left(
W_{[q]}(x)\right) ^{-1}\d\mu (x)A(x^{-1})+\left\langle \mu _{s},\frac{1%
}{z-x}A(x^{-1})\right\rangle ,  \label{CCurlycheck}
\end{align}%
being this $\mathcal{C}(z)$ a singular function with simple poles at the $2dq $ simple zeros of the perturbation matrix $W_{[q]}(z)$. Notice that we express the second term of the above expression \eqref{CCurlycheck} as a functional. This is because the set of discrete points where $\d\mu _{s}$ is supported, namely the singular part of the measure $\d\check{\mu}$, is not located on the unit circle $\mathbb{T}$. That is, $\d\mu _{s}$ is not a matrix measure supported on the unit circle.

\begin{align*}
W_{[q]}(z)\check{C}(z)-C(z)\,N_{G}=\oint_{\mathbb{T}}\frac{%
W_{[q]}(z)-W_{[q]}(x)}{z-x}\d\check{\mu}(x)\check{A}(x^{-1})
\end{align*}%
Hence, in terms of $\mathcal{C}(z)$, formula \eqref{skfCconn} reads%
\begin{align}
\left( W_{[q]}(z)\mathcal{C}(z)-C(z)\right) \,N_{G}=\oint\limits_{%
\mathbb{T}}\frac{W_{[q]}(z)-W_{[q]}(x)}{z-x}\d\check{\mu}(x)\check{A}%
(x^{-1}).  \label{newCcheck}
\end{align}

\begin{Remark}
\label{RemfactorPol}We need to briefly analyze the quotient on the right hand side of the above equation. The matrix $W_{[q]}$ is defined in \eqref{mbLp[r]} as a diagonal matrix, so we define
\begin{align}
\delta W_{[q]}(z,x)\coloneq\frac{W_{[q]}(z)-W_{[q]}(x)}{z-x},
\label{defdeltaW}
\end{align}
and we have at the $(b,b)$, $b\in \{1,\ldots ,q\}$ diagonal entry, the term
\begin{align}
\delta W_{b}(z,x)=\frac{W_{b}(z)-W_{b}(x)}{z-x}.  \label{factorPol}
\end{align}
If $W_{b}$ is a Laurent polynomial described in \eqref{polyPerturbadorW},
their degrees are exactly $\deg ^{-}W_{b}=-d$ and $\deg ^{+}W_{b}=d$, so
therefore \eqref{factorPol} it is in turn a Laurent polynomial of precise
degrees%
\begin{align}
\deg ^{-}\left( \delta W_{b}(z,x)\right) =-d,\qquad \deg ^{+}\left( \delta
W_{b}(z,x)\right) =d-1.  \label{degFactor}
\end{align}%
This trivially follows from the fact that, when $k\in \mathbb{N}$, the term $%
(z^{k}-x^{k})$ is always divisible by $(z-x)$.
\end{Remark}

Next, our aim is to determine the entries of $N_{G}$. In order to do
that, we first consider the next significant finding concerning the
expression \eqref{newCcheck}

\begin{pro}
For every $b\in \{1,\ldots ,q\}$, with the definition \eqref{CCurlycheck},
the expression \eqref{newCcheck} becomes%
\begin{align}
\left( W_{b}(z)\mathcal{C}^{(b)}(z)-C^{(b)}(z)\right) \,N_{G}=0
\label{SuperCigual0}
\end{align}%
if the degree $n$ satisfies%
\begin{align*}
n\geq 2dq-2+b.
\end{align*}
\end{pro}

\begin{proof}
We first come back to the entrywise expression of \eqref{newCcheck}, which is%
\begin{align}
\sum_{n=\ell }^{\ell +2dq}\left( W_{b}(z)\mathcal{C}%
_{n}^{(b)}(z)-C_{n}^{(b)}(z)\right) \,N_{G,n,\ell }=\oint\limits_{%
\mathbb{T}}\delta W_{b}(z,x)\d\check{\mu}_{b,a}(x)\check{A}%
_{n}^{(a)}(x^{-1}),  \label{newCcheckDelta}
\end{align}%
where $\delta W_{b}(z,x)$ is the Laurent polynomial with degrees given in \eqref{degFactor}, to observe that its right hand side will not always be different from zero, and that this fact depends directly on the degrees of the Laurent polynomial $\delta W_{b}(z,x)$ (considered as a polynomials in the $x$ variable). Next, from the orthogonality of the $A$ polynomials we easily obtain the component-wise orthogonality for the Geronimus $W$-modified measure, namely
\begin{align*}
\begin{aligned}
	\sum_{a=1}^{p}\oint_{\mathbb{T}}\check{A}_{n}^{(a)}(x^{-1})\,x^{k}\d \check{\mu}_{b,a}(x) =&0, &
&\text{for $-{\left\lceil \frac{n+2-b-q}{2q}\right\rceil } \leq  k\leq \left\lceil \frac{n+2-b}{2q}\right\rceil -1$},& b&\in \{1,\ldots ,q\}.
\end{aligned}
\end{align*}%
Therefore, replacing the degrees \eqref{degFactor} in the above expression, and solving the resulting inequalities, we can easily derive the statement of the proposition.
\end{proof}

On the other hand, taking \eqref{CCurlycheck} and \eqref{ecolumnvec} into
account, we need to deeply analyze the product $$(\boldsymbol{e}_{b}^{[q]})^{\top
}W_{[q]}(z)\check{C}(z).$$ We are going to obtain its limit when $%
z\rightarrow z_{b,j}$, which are the aforementioned $2dq$ zeros of the
perturbation matrix $W_{[q]}(z)$. Thus observe that%
\begin{align*}
\lim_{z\rightarrow z_{b,j}}(\boldsymbol{e}_{b}^{[q]})^{\top }W_{[q]}(z)\check{C}%
(z)
&=\begin{multlined}[t][.7\textwidth]
	(\boldsymbol{e}_{b}^{[q]})^{\top }\left( \oint_{\mathbb{T}}\frac{1}{z-x%
}W_{[q]}(z)\left( W_{[q]}(x)\right) ^{-1} \d\mu
(x)A(x^{-1})\right.\\+\left.W_{[q]}(z)\left\langle \d\mu _{s}(x),\frac{1}{z-x}%
A(x^{-1})\right\rangle \right) N_{G}
\end{multlined}\\
&=\lim_{z\rightarrow z_{b,j}}(\boldsymbol{e}_{b}^{[q]})^{\top
}W_{[q]}(z)\left\langle \d\mu _{s}(x),\frac{1}{z-x}A(x^{-1})\right\rangle 
N_{G}.
\end{align*}%
Next, from \eqref{SigmaSingular01}, the above expression becomes%
\begin{align*}
\lim_{z\rightarrow z_{b,j}}(\boldsymbol{e}_{b}^{[q]})^{\top
}W_{[q]}(z)\left\langle \d\mu _{s}(x),\frac{1}{z-x}A(x^{-1})\right\rangle 
N_{G}
&=\lim_{z\rightarrow z_{b,j}}(\boldsymbol{e}_{b}^{[q]})^{\top
}W_{[q]}(z)\left\langle \sum_{b^{\prime }=1}^{q}\boldsymbol{e}_{b^{\prime
}}^{[q]}\boldsymbol{\upsilon }_{b^{\prime }}^{[p]}(x),\frac{1}{z-x}%
A(x^{-1})\right\rangle N_{G}\\
&=\begin{multlined}[t][.4\textwidth]
	\lim_{z\rightarrow z_{b,j}}(\boldsymbol{e}_{b}^{[q]})^{\top }\sum_{b^{\prime
}=1}^{q}\sum_{a=1}^{p}\sum_{j^{\prime }=1}^{2d}W_{b^{\prime }}(z)\boldsymbol{e}%
_{b^{\prime }}^{[q]}(\boldsymbol{e}_{a}^{[p]})^{\top }m_{b^{\prime },a,j^{\prime
}}\\\times \left\langle \delta (z-z_{b^{\prime },j^{\prime }}),\frac{1}{z-x}%
A(x^{-1})\right\rangle N_{G}
\end{multlined}
\end{align*}

Next, the Kronecker delta $\delta _{b,b^{\prime }}=(\boldsymbol{e}%
_{b}^{[q]})^{\top }\boldsymbol{e}_{b^{\prime }}^{[q]}$ removes the sum $b$'s,
and knowing that%
\begin{align*}
\left\langle \delta (z-z_{b^{\prime },j^{\prime }}),\frac{1}{z-x}%
A(x^{-1})\right\rangle =\frac{1}{z-z_{b^{\prime },j^{\prime }}}%
A(z_{b^{\prime },j^{\prime }}^{-1}),
\end{align*}%
and%
\begin{align*}
W_{b^{\prime }}(z)=\prod\limits_{s=1}^{2d}(z-z_{b^{\prime },s}),
\end{align*}%
we get%
\begin{align*}
\lim_{z\rightarrow z_{b,j}}(\boldsymbol{e}_{b}^{[q]})^{\top
}W_{[q]}(z)\left\langle \d\mu _{s},\frac{1}{z-x}A(x^{-1})\right\rangle 
N_{G}&=\lim_{z\rightarrow z_{b,j}}\left( \sum_{a=1}^{p}\sum_{j^{\prime
}=1}^{2d}\prod\limits_{s=1}^{2d}(z-z_{b,s})(\boldsymbol{e}_{a}^{[p]})^{\top
}m_{b,a,j^{\prime }}\frac{1}{z-z_{b,j^{\prime }}}A(z_{b,j^{\prime
}}^{-1})\right) N_{G}\\
&=\lim_{z\rightarrow z_{b,j}}\left( \sum_{a=1}^{p}\sum_{j^{\prime
}=1}^{2d}\prod\limits_{s=1,s\neq j^{\prime }}^{2d}(z-z_{b,s})(\boldsymbol{e}%
_{a}^{[p]})^{\top }m_{b,a,j^{\prime }}A(z_{b,j^{\prime }}^{-1})\right) 
N_{G}.
\end{align*}%
Finally, having the limit $z\rightarrow z_{b,j}$ of the above expression,
only survive the terms with $j=j^{\prime }$, so we conclude%
\begin{align}
\lim_{z\rightarrow z_{b,j}}(\boldsymbol{e}_{b}^{[q]})^{\top }W_{[q]}(z)\check{C}%
(z)=\sum_{a=1}^{p}(\boldsymbol{e}_{a}^{[p]})^{\top }\prod\limits_{s=1,s\neq
j}^{2d}(z_{b,j}-z_{b,s})m_{b,a,j}A(z_{b,j}^{-1})N_{G}.
\label{RemainGer}
\end{align}



\subsection{Christoffel--Geronimus formula for $\check{A}(z)$}



Next, we provide the corresponding Christoffel--Geronimus formula for the CMV
Laurent polynomials $\check{A}(z)$. With analogous notation as in the
Christoffel case, the highest nonzero band in the connection matrix $N_G$ will be denoted as $N_{G,\ell }$ see \eqref{bandedmatrixv06}.

From \eqref{defmatrixNGer} we see that the structure of the connection
matrix $N_{G}$\ depends on the polynomials at the entries of the
perturbation matrix $W_{[q]}(z)$. Indeed, denoting by $N_{G,n,\ell }$
the elements of the upper triangular matrix $N_{G}$ (i.e., $N_{G,\ell,m}$ if $n>\ell $), thus for $\ell \in\N_0$, $b\in
\{1,\ldots ,q\}$ we get%
\begin{align*}
\sum_{n=\ell }^{\ell +2dq}\left( W_{b}(z)\mathcal{C}%
_{n}^{(b)}(z)-C_{n}^{(b)}(z)\right) N_{G,n,\ell }=0.
\end{align*}%
Taking \eqref{RemainGer} into account, at the $2dq$\ zeros of the
perturbation matrix $W_{[q]}(z)$ we reach%
\begin{align*}
\sum_{n=\ell }^{\ell +2dq}\left( \left(
\sum_{a=1}^{p}\prod\limits_{s=1,s\neq
j}^{2d}(z_{b,j}-z_{b,s})m_{b,a,j}\,A_{n}^{(a)}(z_{b,j}^{-1})\right)
-C_{n}^{(b)}(z_{b,j})\right) N_{G,n,\ell }=0.
\end{align*}

As in the Christoffel case, we are interested in solving the above linear
system of equations for the entries $N_{G,n,\ell }$, $n\geq \ell $,
and therefore we restrict ourselves to the situation where all the zeros $%
z_{b,j}$ are simple, and different for every $b\in1,\ldots ,q$. Thus,%
\begin{multline*}
\sum_{n=\ell }^{\ell +2dq-1}\left( \left(
\sum_{a=1}^{p}\prod\limits_{s=1,s\neq
j}^{2d}(z_{b,j}-z_{b,s})m_{b,a,j}\,A_{n}^{(a)}(z_{b,j}^{-1})\right)
-C_{n}^{(b)}(z_{b,j})\right) N_{G,n,\ell }  
\\=-\left( \left( \sum_{a=1}^{p}\prod\limits_{s=1,s\neq
j}^{2d}(z_{b,j}-z_{b,s})m_{b,a,j}\,A_{\ell +2dq}^{(a)}(z_{b,j}^{-1})\right)
-C_{\ell +2dq}^{(b)}(z_{b,j})\right) \tilde N_{G,\ell}  \notag
\end{multline*}%
Next, let us define%
\begin{align*}
F_{n,j}^{(b)}\coloneq C_{n}^{(b)}(z_{b,j})-\sum_{a=1}^{p}\prod%
\limits_{s=1,s\neq
j}^{2d}(z_{b,j}-z_{b,s})m_{b,a,j}\,A_{n}^{(a)}(z_{b,j}^{-1}).
\end{align*}

Taking into account the connection expression $\check{A}%
^{(a)}(z^{-1})=A^{(a)}(z^{-1})\,N_{G}$ we finally obtain%
\begin{multline}
\frac{1}{{\tilde N_{G,\ell}}}\check{A}_{\ell
}^{(a)}(z^{-1})=A_{\ell +2dq}^{(a)}(z^{-1})-{%
\begin{bmatrix}
F_{\ell +2dq,1}^{(1)} & \cdots & F_{\ell +2dq,2d}^{(1)} & \cdots & F_{\ell
+2dq,1}^{(q)} & \cdots & F_{\ell +2dq,2d}^{(q)}%
\end{bmatrix}%
}\\
\times {%
\begin{bmatrix}
F_{\ell ,1}^{(1)} & \cdots & F_{\ell ,2d}^{(1)} & \cdots & F_{\ell ,1}^{(q)}
& \cdots & F_{\ell ,2d}^{(q)} \\ 
F_{\ell +1,1}^{(1)} & \cdots & F_{\ell +1,2d}^{(1)} & \cdots & F_{\ell
+1,1}^{(q)} & \cdots & F_{\ell +1,2d}^{(q)} \\ 
\vdots &  & \vdots &  & \vdots &  & \vdots \\ 
F_{\ell +2dq-1,1}^{(1)} & \cdots & F_{\ell +2dq-1,2d}^{(1)} & \cdots & 
F_{\ell +2dq-1,1}^{(q)} & \cdots & F_{\ell +2dq-1,2d}^{(q)}%
\end{bmatrix}%
}^{-1}%
\begin{bmatrix}
A_{\ell }^{(a)}(z^{-1}) \\ 
A_{\ell +1}^{(a)}(z^{-1}) \\ 
\vdots \\ 
A_{\ell +2dq-1}^{(a)}(z^{-1})%
\end{bmatrix}%
.  \label{GerBfinal}
\end{multline}

As a quasi-determinant in \eqref{QuasiDet}, we write%
\begin{align*}
\frac{1}{{\tilde N_{G,\ell}}}\check{A}^{(a)}(z^{-1})=\Theta
_{\ast }{%
\begin{bmatrix}
F_{\ell ,1}^{(1)} & \cdots & F_{\ell ,2d}^{(q)} & A_{\ell }^{(a)}(z^{-1}) \\ 
\vdots &  & \vdots & \vdots \\ 
F_{\ell +2dq-1,1}^{(1)} & \cdots & F_{\ell +2dq-1,2d}^{(q)} & A_{\ell
+2dq-1}^{(a)}(z^{-1}) \\ 
F_{\ell +2dq,1}^{(1)} & \cdots & F_{\ell +2dq,2d}^{(q)} & A_{\ell
+2dq}^{(a)}(z^{-1})%
\end{bmatrix}%
.}
\end{align*}

Thus, we have proved the following



\begin{pro}
The Christoffel--Geronimus formula for the perturbed CMV mixed multiple
Laurent OPUC on the step-line $\check{A}(z^{-1})$ can be written as%
\begin{align*}
\frac{1}{{\tilde N_{G,\ell}}}\check{A}^{(a)}(z^{-1})={\frac{%
\begin{vmatrix}
F_{\ell ,1}^{(1)} & \cdots & F_{\ell ,2d}^{(q)} & A_{\ell }^{(a)}(z^{-1}) \\ 
\vdots &  & \vdots & \vdots \\ 
F_{\ell +2dq-1,1}^{(1)} & \cdots & F_{\ell +2dq-1,2d}^{(q)} & A_{\ell
+2dq-1}^{(a)}(z^{-1}) \\ 
F_{\ell +2dq,1}^{(1)} & \cdots & F_{\ell +2dq,2d}^{(q)} & A_{\ell
+2dq}^{(a)}(z^{-1})%
\end{vmatrix}%
}{%
\begin{vmatrix}
F_{\ell ,1}^{(1)} & \cdots & F_{\ell ,2d}^{(q)} \\ 
\vdots &  & \vdots \\ 
F_{\ell +2dq-1,1}^{(1)} & \cdots & F_{\ell +2dq-1,2d}^{(q)}%
\end{vmatrix}%
}.}
\end{align*}
\end{pro}



\subsection{Cauchy transforms of the kernel}



In order to obtain the Christoffel--Geronimus formula for $\check{B}(z^{-1})$%
, we first need to obtain the Cauchy transforms of the kernels. We can
consider two Cauchy transforms of the kernel, one with respect to the
variable $x$ and one with respect to the variable $y$. From \eqref{CDKernel(x,y)} we have the $p\times q$ matrix $%
K^{[n]}(x,y)=A^{[n]}(x^{-1})B^{[n]}(y)$. We know from \eqref{skfC} and \eqref{skfD}
\begin{align*}
C(z)=\oint_{\mathbb{T}}\frac{1}{z-x}\d\mu (x)A(x^{-1}),\qquad
D(z)=\oint_{\mathbb{T}}\frac{1}{z-x}B(x)\d\mu (x)
\end{align*}%
where $C(z)$ is $q\times \infty $ and $D(z)$ is has size $\infty \times p$.
Thus, we define the Cauchy transforms of the kernel%
\begin{align*}
K_{C}^{[n]}(z,y)\coloneq\oint_{\mathbb{T}}\d\mu (x)\frac{%
K^{[n]}(x,y)}{z-x} =&\oint_{\mathbb{T}}\d\mu (x)\frac{%
A^{[n]}(x^{-1})B^{[n]}(y)}{z-x}=C^{[n]}(z)B^{[n]}(y), \\
K_{D}^{[n]}(x,z)\coloneq\oint_{\mathbb{T}}\frac{K^{[n]}(x,y)}{z-y} \d\mu (y) =&\oint_{\mathbb{T}}\frac{A^{[n]}(x^{-1})B^{[n]}(y)}{z-y}%
\d\mu (y)=A^{[n]}(x^{-1})D^{[n]}(z),
\end{align*}%
where $K_{C}^{[n]}(z,y)$ is $q\times q$ and $K_{D}^{[n]}(x,z)$ is $p\times p$%
. Thus 
\begin{align}
\check{K}_{D}^{[n]}(x,z)\coloneq\oint_{\mathbb{T}}\frac{\check{K}%
^{[n]}(x,y)}{z-y}\d\check{\mu}(y)=\oint_{\mathbb{T}}\frac{\check{A}%
^{[n]}(x^{-1})\check{B}^{[n]}(y)}{z-y}\d\check{\mu}(y)=\check{A}%
^{[n]}(x^{-1})\check{D}^{[n]}(z).  \label{KDcheck}
\end{align}%
Applying $A(z^{-1})\,N_{G}=\check{A}(z^{-1})$ from \eqref{relatABGer}, $%
N_{G}\,\check{D}(z)=D(z)$ from \eqref{skfCconn}, and the shape of $%
N_{G}$, the above becomes%
\begin{align*}
\check{K}_{D}^{[n]}(x,z) =&\check{A}^{[n]}(x^{-1})\check{D}^{[n]}(z)=\left(
A(z^{-1})\,N_{G}\right) ^{[n]}\check{D}^{[n]}(z) \\
=&A^{[n]}(z^{-1})\,N_{G}^{[n]}\mathsf{\,}\check{D}%
^{[n]}(z)=A^{[n]}(z^{-1})\d^{[n]}(z)-A^{[n]}(z^{-1})N_{G}^{[n,2dq]}%
\check{D}^{[n,2dq ]}(z) \\
=&K_{D}^{[n]}(x,z)-A^{[n]}(z^{-1})N_{G}^{[n,2dq]}\check{D}%
^{[n,2dq ]}(z).
\end{align*}

We thus conclude%
\begin{align}
\check{K}_{D}^{[n]}(x,z)=K_{D}^{[n]}(x,z)-A^{[n]}(z^{-1})N_{G}^{[n,2dq]}%
\check{D}^{[n,2dq ]}(z).  \label{ConnDCauchy}
\end{align}

On the other hand, concerning the Cauchy transform of the other kernel, we
have%
\begin{align}
\check{K}_{C}^{[n]}(z,y)\coloneq\oint_{\mathbb{T}}\d\check{\mu}(x)%
\frac{\check{K}^{[n]}(x,y)}{z-x}=\oint_{\mathbb{T}}\d\check{\mu}(x)%
\frac{\check{A}^{[n]}(x^{-1})\check{B}^{[n]}(y)}{z-x}=\check{C}^{[n]}(z)%
\check{B}^{[n]}(y).  \label{KCcheck}
\end{align}%
Thus%
\begin{align}
W_{[q]}(z)\check{K}_{C}^{[n]}(z,y)=\left( W_{[q]}(z)\check{C}%
^{[n]}(z)\right) \,\check{B}^{[n]}(y).  \label{WKcheck}
\end{align}%
From \eqref{skfCconn}, and the sape of $N_{G}$, the above becomes%
\begin{align*}
W_{[q]}(z)\check{K}_{C}^{[n]}(z,y) =&\left( W_{[q]}(z)\check{C}%
^{[n]}(z)\right) \check{B}^{[n]}(y) \\
=&\left( \oint_{\mathbb{T}}\frac{W_{[q]}(z)-W_{[q]}(x)}{z-x}\d\check{\mu}(x)\check{A}^{[n]}(x^{-1})+C^{[n]}(z)\,N_{G}^{[n]}\right) \check{B}^{[n]}(y) \\
=&C^{[n]}(z)\,N_{G}^{[n]}\check{B}^{[n]}(y)+\oint_{\mathbb{T}}%
\frac{W_{[q]}(z)-W_{[q]}(x)}{z-x}\d\check{\mu}(x)\check{A}^{[n]}(x^{-1})%
\check{B}^{[n]}(y) \\
=&C^{[n]}(z)\,N_{G}^{[n]}\check{B}^{[n]}(y)+\oint_{\mathbb{T}}%
\frac{W_{[q]}(z)-W_{[q]}(x)}{z-x}\d\check{\mu}(x)\check{K}^{[n]}(x,y).
\end{align*}%
Having into account the projection properties of the kernels, and being $%
n>2d $, we obtain (see \eqref{defdeltaW})%
\begin{align*}
\oint_{\mathbb{T}}\frac{W_{[q]}(z)-W_{[q]}(x)}{z-x}\d\check{\mu}(x)%
\check{K}^{[n]}(x,y)=\frac{W_{[q]}(z)-W_{[q]}(y)}{z-y}=\delta W_{[q]}(z,y).
\end{align*}%
Next, applying $N_{G} \check{B}(z)=B(z)W_{[q]}(z)$ from \eqref{relatABGer}
\begin{align*}
W_{[q]}(z)\check{K}_{C}^{[n]}(z,y) =&C^{[n]}(z)\,N_{G}^{[n]}\check{B}%
^{[n]}(y)+\delta W_{[q]}(z,y) \\
=&C^{[n]}(z)B^{[n]}(y)W_{[q]}(y)-C^{[n]}(z)N_{G}^{[n,2dq]}\check{B}%
^{[n,2dq]}(y)+\delta W_{[q]}(z,y) \\
=&K_{C}^{[n]}(z,y)W_{[q]}(y)-C^{[n]}(z)N_{G}^{[n,2dq]}\check{B}%
^{[n,2dq]}(y)+\delta W_{[q]}(z,y)
\end{align*}%
We therefore conclude%
\begin{align}
W_{[q]}(z)\check{K}_{C}^{[n]}(z,y)=K_{C}^{[n]}(z,y)W_{[q]}(y)-C^{[n]}(z)%
N_{G}^{[n,2dq]}\check{B}^{[n,2dq]}(y)+\delta W_{[q]}(z,y).
\label{ConnCCauchy}
\end{align}



\subsection{Christoffel--Geronimus formula for $\check{B}(z^{-1})$}



We next give the Christoffel--Geronimus formula for the CMV Laurent
polynomials $\check{B}(z)$. We begin with \eqref{ConnCCauchy} above%
\begin{align*}
W_{[q]}(z)\check{K}_{C}^{[n]}(z,y)=K_{C}^{[n]}(z,y)W_{[q]}(y)-C^{[n]}(z)%
N_{G}^{[n,2dq]}\check{B}^{[n,2dq]}(y)+\delta W_{[q]}(z,y),
\end{align*}%
where $W_{[q]}(z)$, $\check{K}_{C}^{[n]}(z,y)$, $K_{C}^{[n]}(z,y)$,  $W_{[q]}(y)$, and $\delta W_{[q]}(z,y)$ are $q\times q$ matrices, and $C^{[n]}(z) $ is a $q\times n$ matrix, $N_{G}^{[n,2dq]}$ is a $n\times 2dq$ matrix, and $\check{B}^{[n,2dq]}(y)$ is a $2dq\times q$ matrix.

From \eqref{WKcheck} we recall that
\begin{align*}
W_{[q]}(z)\check{K}_{C}^{[n]}(z,y)=\left( W_{[q]}(z)\check{C}%
^{[n]}(z)\right) \,\check{B}^{[n]}(y),
\end{align*}%
where $W_{[q]}(z)$ is $q\times q$, $\check{C}^{[n]}(z)$ is $q\times n$, and $%
\check{B}^{[n]}(y)$ is $n\times q$. Next, we multiply above by $(\boldsymbol{e}%
_{b}^{[q]})^{\top }$ and compute the limit $z\rightarrow z_{b,j}$, reaching%
\begin{align*}
\lim_{z\rightarrow z_{b,j}}(\boldsymbol{e}_{b}^{[q]})^{\top }W_{[q]}(z)\check{K}%
_{C}^{[n]}(z,y) =&\lim_{z\rightarrow z_{b,j}}(\boldsymbol{e}_{b}^{[q]})^{\top
}\left( W_{[q]}(z)\check{C}^{[n]}(z)\right) \,\check{B}^{[n]}(y) \\
=&\lim_{z\rightarrow z_{b,j}}(\boldsymbol{e}_{b}^{[q]})^{\top }\left( W_{[q]}(z)%
\mathcal{C}^{[n]}(z)\right) \,N_{G}^{[n]}\check{B}^{[n]}(y)
\end{align*}%
which, taking \eqref{NBcheckn} and \eqref{RemainGer} into account, and replacing in \eqref{ConnCCauchy} yields
\begin{multline*}
\left( \sum_{a=1}^{p}(\boldsymbol{e}_{a}^{[p]})^{\top }\prod\limits_{s=1,s\neq
j}^{2d}(z_{b,j}-z_{b,s})m_{b,a,j}A^{[n]}(z_{b,j}^{-1})\right) \left( -%
N_{G}^{[n,2dq]}\check{B}^{[n,2dq]}(y)+B^{[n]}(y)W_{[q]}(y)\right) \\
=(\boldsymbol{e}_{b}^{[q]})^{\top }K_{C}^{[n]}(z_{b,j},y)W_{[q]}(y)-(\boldsymbol{e}%
_{b}^{[q]})^{\top }C^{[n]}(z_{b,j})N_{G}^{[n,2dq]}\check{B}%
^{[n,2dq]}(y)+(\boldsymbol{e}_{b}^{[q]})^{\top }\delta W_{[q]}(z_{b,j},y).
\end{multline*}

From here we get%
\begin{multline*}
\left( (\boldsymbol{e}_{b}^{[q]})^{\top }C^{[n]}(z_{b,j})-\sum_{a=1}^{p}(%
\boldsymbol{e}_{a}^{[p]})^{\top }\prod\limits_{s=1,s\neq
j}^{2d}(z_{b,j}-z_{b,s})m_{b,a,j}A^{[n]}(z_{b,j}^{-1})\right) N_{G}^{[n,2dq]}\check{B}^{[n,2dq]}(y) \\
=\left( (\boldsymbol{e}_{b}^{[q]})^{\top }C^{[n]}(z_{b,j})-\left(
\sum_{a=1}^{p}(\boldsymbol{e}_{a}^{[p]})^{\top }\prod\limits_{s=1,s\neq
j}^{2d}(z_{b,j}-z_{b,s})m_{b,a,j}A^{[n]}(z_{b,j}^{-1})\right) \right)
B^{[n]}(y)W_{[q]}(y) \\
+(\boldsymbol{e}_{b}^{[q]})^{\top }\delta W^{(b)}(z_{b,j},y).
\end{multline*}%
Next, let us define%
\begin{align*}
F_{n,j}^{(b)}\coloneq C_{n}^{(b)}(z_{b,j})-\sum_{a=1}^{p}\prod%
\limits_{s=1,s\neq
j}^{2d}(z_{b,j}-z_{b,s})m_{b,a,j}A_{n}^{(a)}(z_{b,j}^{-1}),
\end{align*}%
which are the entries of%
\begin{align*}
\mathcal{F}=%
\begin{bmatrix}
F_{n,1}^{(1)} & \cdots & F_{n+2dq-1,1}^{(1)} \\ 
\vdots &  & \vdots \\ 
F_{n,2dq}^{(q)} & \cdots & F_{n+2dq-1,2dq}^{(q)}%
\end{bmatrix},
\end{align*}%
and%
\begin{align*}
\mathcal{K}_{b,b^{\prime },j}^{[n]} =&\left( K_{C}^{[n]}\right)
_{b,b^{\prime }}(z_{b,j},y) -\sum_{a=1}^{p}\prod\limits_{s=1,s\neq
j}^{2d}(z_{b,j}-z_{b,s})m_{b,a,j}K_{a,b^{\prime
}}^{[n]}(z_{b,j},y)W_{[q]}^{(b^{\prime })}(y)+\delta
W^{(b)}(z_{b,j},y)\delta _{b,b^{\prime }},
\end{align*}%
which are the entries of the $2dq\times q$ matrix
\begin{align*}
\mathcal{K}^{[n]}(y)=%
\begin{bmatrix}
\mathcal{K}_{1,1,1}^{[n]}(y) & \cdots & \mathcal{K}_{1,q,1}^{[n]}(y) \\ 
\vdots &  & \vdots \\ 
\mathcal{K}_{q,1,2d}^{[n]}(y) & \cdots & \mathcal{K}_{q,q,2d}^{[n]}(y)%
\end{bmatrix}.
\end{align*}%
Observe that $\mathcal{F}$ is $2dq\times 2dq$, $N_{G}^{[n,2dq]}\check{B}%
^{[n,2dq]}(y)$ is $2dq\times q$, $(\boldsymbol{e}_{1}^{[q]})^{\top }\delta
W_{[q]}(z_{b,j},y)$ and $\mathcal{K}$ are row vectors with size $1\times q$.
In terms of quasi-determinants, for $b\in {1,\ldots ,q}$ we have%
\begin{align*}
\mathcal{F}=%
\begin{bmatrix}
F_{n-2d+1}(z_{1,1}) & \cdots & F_{n}(z_{1,1}) \\ 
\vdots &  & \vdots \\ 
F_{n-2d+1}(z_{q,2dq}) & \cdots & F_{n}(z_{q,2dq})%
\end{bmatrix},
\end{align*}%
Observe that $\mathcal{F}$ is $2dq\times 2dq$, $N_{G}^{[n,2dq]}\check{B}%
^{[n,2dq]}(y)$ is $2dq\times q$ matrix, $(\boldsymbol{e}_{1}^{[q]})^{\top }\delta
W_{[q]}(z_{b,j},y)$ and $\mathcal{K}$ are row vectors with size $1\times q$.
In terms of quasi-determinants, for $b\in {1,\ldots ,q}$ we have

\begin{align*}
\check{B}^{(b)}&=\frac{-1}{\tilde N_{G,2dq}}\Theta _{\ast }
\left[\begin{NiceArray}{c|c} \mathcal{F} & \mathcal{K}_{b}^{[n]} \\
\hline (\boldsymbol{e}_{1}^{[2dq]})^{\top } & 0 
\end{NiceArray}\right]=\frac{-1}{%
\tilde N_{G,2dq}}(\boldsymbol{e}_{1}^{[2dq]})^{\top }\mathcal{F}^{-1}%
\mathcal{K}_{b}^{[n]}=\frac{-1}{\tilde N_{G,2dq}}\frac{%
\begin{vNiceMatrix}[margin,hvlines] \mathcal{F} & \mathcal{K}_{b}^{[n]} \\
\hline (\boldsymbol{e}_{1}^{[2dq]})^{\top } & 0 \\ \end{vNiceMatrix}}{\left\vert 
\mathcal{F}\right\vert },\\
\check{B}_{n}^{(b)}&=\frac{-1}{\tilde N_{G,2dq}} \frac{
\begin{vNiceMatrix}
 F_{n+1,1}^{(1)} & \Cdots & F_{n+2dq-1,1}^{(1)} &   \mathcal{K}_{1,b,1}^{[n]}(y) \\
\Vdots & & \Vdots &  \Vdots \\ 
F_{n+1,2dq}^{(q)} & \Cdots & F_{n+2dq-1,1}^{(q)} & \mathcal{K}_{q,b,2d}^{[n]}(y)
\end{vNiceMatrix}
}{
\begin{vNiceMatrix} 
F_{n,1}^{(1)} & \Cdots & F_{n+2dq-1,1}^{(1)} \\
 \Vdots & & \Vdots \\
F_{n,2dq}^{(q)} & \Cdots & F_{n+2dq-1,1}^{(q)}
\end{vNiceMatrix}
} .
\end{align*}

\section*{Conclusions and Outlook}

This paper develops the framework for mixed multiple orthogonal Laurent polynomials, beginning with CMV mixed multiple Laurent orthogonality and constructing the moment matrix along with its Gauss–Borel factorization. It then explores the orthogonality properties, recurrence relations, Szegő matrices, and Christoffel–Darboux kernels that define these polynomials.

The analysis includes both Christoffel and Geronimus perturbations, examining their connector matrices, connection formulas, and specific formulas that modify polynomial functions. Geronimus perturbations are further detailed with singular parts and Cauchy transforms of the kernel, enhancing the understanding of these polynomials and their structural behavior.

The orthogonality developed in this paper does not connect directly with that in \cite{MV-CA08}; however, it is related to a variant weighted version of it. In fact, there are alternative approaches to constructing the moment matrix, inspired by properties in the scalar case, which also lead to a consistent theory. Both approaches represent promising lines of research. Additionally, it would be valuable to identify sufficiently general perturbations beyond the diagonal case for which Christoffel perturbations can be established, following the developments in the real case (see \cite{AAM,MR,MR1}).












\section*{Acknowledgments}



EJH acknowledges Direcci\'{o}n General de Investigaci\'{o}n e Innovaci\'{o}n, Consejer\'{i}a
de Educaci\'{o}n e Investigaci\'{o}n of the Comunidad de Madrid (Spain)and
Universidad de Alcal\'{a}, research project CM/JIN/2021-014, \textit{Proyectos de I+D
para J\'{o}venes Investigadores de la Universidad de Alcal\'{a} 2021}, and the
Ministerio de Ciencia e Innovaci\'{o}n-Agencia Estatal de Investigaci\'{o}n MCIN/AEI/10.13039/501100011033 and the European Union
``NextGenerationEU''/PRTR, research project TED2021-129813A-I00. This research was conducted while EJH was visiting the ICMAT (Instituto de Ciencias Matem\'{a}ticas) from jan-2023 to jan-2024, under the Program \textit{Ayudas de Recualificaci\'{o}n del Sistema Universitario Espa\~{n}ol para 2021-2023 (Convocatoria 2022) - R.D. 289/2021 de 20 de abril (BOE de 4 de junio de 2021)}. This author also wish to thank the ICMAT, Universidad de Alcal\'{a}, and the Plan de Recuperaci\'{o}n, Transformaci\'{o}n y Resiliencia (NextGenerationEU) of the Spanish Government for their support.

MM acknowledges Spanish \textquotedblleft Agencia Estatal de Investigaci\'{o}n\textquotedblright\ research project  [PID2021-
122154NB-I00], Ortogonalidad y Aproximaci\'{o}n con Aplicaciones en Machine Learning y Teor\'{i}a de la Probabilidad.



\section*{Declarations}

\begin{itemize}
\item \textbf{Conflict of interest/Competing interests: The authors declare
that they have no conflict of interest.}
\end{itemize}

\begin{itemize}
\item \textbf{Availability of data and materials: Data sharing not
applicable to this article as no datasets were generated or analyzed during
the current study.}
\end{itemize}

\end{document}